\documentclass[12pt,a4paper,reqno]{amsart}

\usepackage[utf8]{inputenc}
\usepackage[T1]{fontenc}
\usepackage[english]{babel}
\usepackage[margin=2.5cm]{geometry}
\usepackage{enumitem}
\usepackage{caption}
\usepackage{subcaption}
\usepackage{booktabs}
\usepackage{microtype}

\usepackage[colorlinks]{hyperref}
\usepackage{tikz}
\usepackage{pgfplots}
\pgfplotsset{compat=newest}

\usepackage{amsmath,amssymb,amsfonts,amsthm}
\usepackage{mathtools,mathrsfs}
\usepackage[capitalize,nameinlink]{cleveref}
\allowdisplaybreaks{}
\numberwithin{equation}{section}

\definecolor{myBlue}{HTML}{1F77B4}
\definecolor{myGreen}{HTML}{2CA02C}
\definecolor{myOrange}{HTML}{FF7F0E}
\definecolor{myRed}{HTML}{D62728}

\hypersetup{colorlinks = True,
    linkcolor  = myRed,
    citecolor  = myGreen,
    urlcolor   = myBlue,
}

\newcommand{\Ab}{\mathbb{A}}
\newcommand{\Cb}{\mathbb{C}}
\newcommand{\Db}{\mathbb{D}}
\newcommand{\Ib}{\mathbb{I}}
\newcommand{\Nb}{\mathbb{N}}
\newcommand{\Pb}{\mathbb{P}}
\newcommand{\Rb}{\mathbb{R}}
\newcommand{\Sb}{\mathbb{S}}
\newcommand{\Tb}{\mathbb{T}}
\newcommand{\Zb}{\mathbb{Z}}

\newcommand{\Cs}{\mathscr{C}}
\newcommand{\Ls}{\mathscr{L}}
\newcommand{\Ss}{\mathscr{S}}

\newcommand{\Dc}{\mathcal{D}}
\newcommand{\Lc}{\mathcal{L}}
\newcommand{\Nc}{\mathcal{N}}
\newcommand{\Pc}{\mathcal{P}}
\newcommand{\Rc}{\mathcal{R}}
\newcommand{\Vc}{\mathcal{V}}

\newcommand{\Hr}{\mathrm{H}}

\newcommand{\Lr}{\mathrm{L}}
\newcommand{\Ir}{\mathrm{I}}
\newcommand{\Tt}{\mathtt{T}}
\newcommand{\Kt}{\mathtt{K}}
\newcommand{\Pt}{\mathtt{P}}
\newcommand{\Ct}{\mathtt{C}}

\newcommand{\e}{\mathsf{e}}
\newcommand{\im}{\mathsf{i}\mkern1mu}

\newcommand{\bJ}{\mathop{}\!\mathsf{J}}

\newcommand{\Hu}{\mathop{}\!\mathsf{H}^{(1)}}

\newcommand{\bI}{\mathop{}\!\mathsf{I}}

\newcommand{\sca}{\mathsf{sc}}
\newcommand{\inc}{\mathsf{in}}
\newcommand{\loc}{\mathrm{loc}}
\newcommand{\comp}{\mathrm{comp}}
\newcommand{\out}{\mathsf{out}}
\newcommand{\bds}{\mathsf{bds}}

\newcommand{\pla}{\mathsf{pla}}
\newcommand{\wgm}{\text{WGM}}
\newcommand{\fem}{\mathsf{fem}}

\newcommand{\di}[1]{\mathop{}\!\mathrm{d}#1}
\DeclareMathOperator{\OO}{\mathcal{O}}

\DeclareMathOperator{\Div}{div}
\DeclareMathOperator{\supp}{supp}
\newcommand{\Ind}[1]{\mathop{}\!\mathbf{1}_{#1}}
\DeclareMathOperator{\Card}{\mathrm{Card}}
\DeclareMathOperator{\Dist}{\mathrm{dist}}
\DeclareMathOperator{\Diag}{\mathrm{diag}}

\DeclareMathOperator{\spe}{Spec}
\DeclareMathOperator{\spess}{Spec_{ess}}
\DeclareMathOperator{\spdis}{Spec_{dis}}

\newcommand{\abs}[1]{\left\lvert#1\right\rvert}
\newcommand{\norm}[1]{\left\lVert#1\right\rVert}
\newcommand{\opnorm}[1]{{\left\vert\kern-0.25ex\left\vert\kern-0.25ex\left\vert#1\right\vert\kern-0.25ex\right\vert\kern-0.25ex\right\vert}}
\newcommand{\plr}[1]{\left(#1\right)}
\newcommand{\clr}[1]{\left[#1\right]}
\newcommand{\alr}[1]{\left\langle#1\right\rangle}
\newcommand{\ilr}[1]{[\![#1]\!]}

\newcommand{\restr}[2]{\left. #1\right\rvert_{#2}}

\newcommand{\setst}[2]{\left\lbrace#1\ \middle\vert\ #2\right\rbrace}

\newcommand{\lapc}{a}
\newcommand{\cavc}{a_\mathsf{c}}
\newcommand{\cav}{\mathsf{c}}
\newcommand{\cavb}{b_\mathsf{c}}

\newcommand{\Rsv}{\mathfrak{Res}}

\newcommand{\eps}{\varepsilon}
\newcommand{\vp}{\varphi}

\newcommand{\lu}{\underline{\lambda}}        \newcommand{\ku}{\underline{k}}
           \newcommand{\uu}{\underline{u}}              \newcommand{\Mu}{\underline{\mu}}            \newcommand{\vu}{\underline{v}}              \newcommand{\ls}{\breve{\lambda}}            \newcommand{\Obf}{\boldsymbol{\mathsf{L}}}   \newcommand{\Rbf}{\boldsymbol{\mathsf{R}}}   \newcommand{\seta}{\varsigma}                \newcommand{\ftz}{\tau_0}                    \newcommand{\htz}{\widehat{\tau}_0}

\theoremstyle{definition}
\newtheorem{definition}{Definition}[section]
\newtheorem{notation}[definition]{Notation}

\theoremstyle{plain}
\newtheorem{lemma}[definition]{Lemma}
\newtheorem{theorem}[definition]{Theorem}
\newtheorem{corollary}[definition]{Corollary}

\theoremstyle{remark}
\newtheorem{remark}[definition]{Remark}

\definecolor{coAR}{HTML}{D62728}

\begin{document}

\title[Unbounded sign-changing transmission problems]{Scattering resonances in unbounded transmission problems with sign-changing coefficient}

\author[C. Carvalho]{Camille Carvalho}
\address{Department of Applied Mathematics,
    University of California, Merced,
    5200 North Lake Road,
    Merced, CA 95343, USA
    \emph{and}
    Univ Lyon, INSA Lyon, UJM, UCBL, ECL, CNRS UMR 5208, ICJ, F-69621, France.
}
\email{camille.carvalho@insa-lyon.fr}

\author[Z. Moitier]{Zo{\"\i}s Moitier}
\address{POEMS, CNRS, Inria, ENSTA Paris, Institut Polytechnique de Paris, 828 Boulevard des Mar\'echaux, 91120, Palaiseau, France
}
\email{zois.moitier@ensta-paris.fr}

\date{\today}

\subjclass[2010]{35P25, 35B40, 78A45}

\keywords{Helmholtz Equation; Scattering resonances; Sign-changing coefficient; Asymptotic expansions.}

\begin{abstract}
    It is well-known that classical optical cavities can exhibit localized phenomena associated to scattering resonances, leading to numerical instabilities in approximating the solution.
    This result can be established via the ``quasimodes to resonances'' argument from the black-box scattering framework.
    Those localized phenomena concentrate at the inner boundary of the cavity and are called whispering gallery modes.
    In this paper we investigate scattering resonances for unbounded transmission problems with sign-changing coefficient (corresponding to optical cavities with negative optical properties, for example made of metamaterials).
    Due to the change of sign of optical properties, previous results cannot be applied directly, and interface phenomena at the metamaterial-dielectric interface (such as the so-called surface plasmons) emerge.
    We establish the existence of scattering resonances for arbitrary two-dimensional smooth metamaterial cavities.
    The proof relies on an asymptotic characterization of the resonances, and showing that problems with sign-changing coefficient naturally fit the black box scattering framework.
    Our asymptotic analysis reveals that, depending on the metamaterial's properties, scattering resonances situated closed to the real axis are associated to surface plasmons.
    Examples for several metamaterial cavities are provided.
\end{abstract}

\maketitle

\section{Introduction}

Unbounded transmission problems with sign-changing coefficients arise in electromagnetics, in particular when one considers Maxwell's equations in the time harmonic regime (with Transverse Electric or Transverse Magnetic polarization) in dielectric-metamaterial structures (typically a bounded metamaterial cavity surrounded by a dielectric).
Contrary to common materials, metamaterials such as the Negative-Index Metamaterials (NIM) exhibit unusual optical properties: for instance a real-valued negative effective dielectric permittivity and/or a negative effective permeability at some frequency range.
There is a great interest in modeling metamaterial cavities to confine and control light.
In particular, at optical frequencies, localized interface surface waves called surface plasmons can arise at dielectric-metamaterial interfaces~\cite{maier2007plasmonics}.
The field of plasmonics is very active as surface plasmons offer strong light enhancement, with applications to next-generation sensors, antennas, high-resolution imaging, cloaking and other~\cite{sannomiya2008situ}.
However, surface plasmons are very sensitive to the geometry and therefore challenging to capture, experimentally and numerically~\cite{BCCC16, helsing2018helmholtz}.
Mathematically, surface plasmons are solutions of the homogeneous Maxwell's equations, they are oscillatory waves along the dielectric-metamaterial interface while exponentially decreasing in both transverse directions.

In classical transmission problems (meaning dielectric-dielectric structures), it has been shown that light can be confined by exciting the so-called Whispering Gallery Modes (WGM)~\cite{RigDum11}.
WGM are essentially supported in the neighborhood of the interior cavity boundary and are associated to \emph{scattering resonances}~\cite{BalDauMoi21}.
It is well-known that the approximation of light scattering in dielectric optical micro-cavities can be drastically affected by WGM, in particular if the excitation wavenumber of the source is close to a WGM resonance~\cite{MoiSpe2019, BalDauMoi21}.
In those cases the norm of the truncated solution operator explodes, which is observed numerically by the solution blowing-up (peaks): we call this \emph{scattering instabilities}.
Knowing the exact value of the scattering resonances is in general challenging (or impossible).
However, one can obtain an asymptotic characterization of the scattering resonances, as done in~\cite{BalDauMoi21}.

The above results do not directly apply to metamaterial cavities due to the change of sign of the optical parameter(s) and the additional interface plasmonic behaviors.
There exists a framework that allows study of a large class of scattering problems, the so-called black box scattering framework. However, it is not immediately clear that unbounded transmission problems with sign-changing coefficients fit in this framework.
In particular, well-posedness of the problem needs more attention, and spectral properties to define a \emph{black box Hamiltonian} (including self-adjointness, lower semi-bound, etc.) may not be true.
Also, surface plasmons have been mainly characterized and investigated in the context of the quasi-static approximation (\emph{e.g.}~\cite{bonnet2013radiation, grieser2014plasmonic, BCCC16, ammari2017mathematical, CaChCi17, bonnetier2018plasmonic, BHM20-hal}) --- where an analytic expression can be found --- therefore there is a need to obtain a characterization for the full problem (no quasi-static) to identify the associated metamaterial scattering resonances.

The goal of this paper is to establish the existence of metamaterial scattering resonances (causing scattering instabilities) via an asymptotic characterization of quasi-resonances (in other words the considered problems fit in the black-box scattering framework), this for various two-dimensional metamaterial cavities (arbitrary smooth shape, with one arbitrary varying negative optical parameter).
Using the \( \mathtt{T} \)-coercivity theory~\cite{BoChCi12, BCCC16, BoCaCi18}, and in the spirit of~\cite{BalDauMoi21}, we establish that the associated spectral operator of scalar transmission problem with sign-changing coefficient is a black box Hamiltonian, and we carry out an asymptotic approximation of the metamaterial scattering resonances.
In this case we find that there is an additional interface resonance family (compared to classical cavities) related to surface plasmons, and a specific scaling is required to asymptotically characterize them.
This family can be located close to the real axis, and is responsible for scattering instabilities.

The paper is organized as follows.
We present the problem and main results in \cref{sec:start}.
To illustrate the metamaterial scattering resonances and their effect, we provide a pedagogical example (case of a circular metamaterial cavity with constant negative coefficient) in \cref{sec:pedago}.
\Cref{sec:construct_quasi} presents the general approach for arbitrary metamaterial cavities, including the construction of the asymptotic approximation at any order.
\Cref{sec:resolv_results} proves their connection to the truncated solution operator (extension of the ``quasimodes to resonances'' result) and their consequence on scattering instabilities.
\Cref{sec:back_to_sca} presents numerical illustrations of the metamaterial scattering resonances, and \cref{sec:conclu} presents our concluding remarks.
\Cref{app:prop_P} provides theoretical results about the problem operator, and \cref{app:sc_asy_details} provides additional results and proofs needed in \cref{sec:construct_quasi}.

\section{Problem setting and main result}\label{sec:start}

\subsection{Mathematical settings}\label{sec:math_settings}

Let us start by introducing the unbounded transmission {problem} with sign-changing coefficient, and its spectral analogous.
We consider an open bounded connected set \( \Omega \subset \Rb^2 \) with smooth boundary \( \Gamma = \partial\Omega \), that represents a transparent (penetrable) optical \emph{cavity} characterized by \( \cavc \in \Cs^\infty\plr{\overline{\Omega}; (-\infty, 0)} \).
The cavity is surrounded by a homogeneous background.
We denote \( \lapc \in \Lr^\infty(\Rb^2) \) the piece-wise smooth function such that
\begin{equation}\label{eq:fct_lap}
    \lapc \equiv \cavc\ \text{on}\ {\Omega}
    \qquad \text{and} \qquad
    \lapc \equiv 1\ \text{on}\ \Rb^2 \setminus \overline{\Omega},
\end{equation}
(see \cref{fig:scheme_geo} for a sketch).
\begin{figure}[!hbtp]
    \begin{tikzpicture}
        \begin{scope}[scale=0.75]
            \filldraw [draw=black,thick,fill=gray!25]
            (2,0) .. controls +(0,1) and +(1,-0.25) .. (0,1.5)
            .. controls +(-1,0.25) and +(0,1) .. (-2,0)
            .. controls +(0,-1) and +(-0.5,0) .. (-1,-1.25)
            .. controls +(0.5,0) and +(-0.5,0) .. (0,-1)
            .. controls +(0.5,0) and +(-0.5,0) .. (1,-1.5)
            .. controls +(0.5,0) and +(0,-1) .. (2,0);
        \end{scope}
        \draw (1,-0.3) node {\( \Omega \)};
        \draw (2.25,-0.3) node {\( \Rb^2 \setminus \overline{\Omega} \)};
        \draw (-0.25,0.3) node {\( \lapc \equiv \cavc < 0 \)};
        \draw (-2.25,0.3) node {\( \lapc \equiv 1 \)};
    \end{tikzpicture}
    \caption{Sketch of the geometry.}\label{fig:scheme_geo}
\end{figure}
We consider the problem: For \( f \in \Lr_\comp^2(\Rb^2) \), \( g \in \Lr^{{2}}(\Gamma) \)\footnote{One could consider data in classical dual functional spaces.
    Then results presented here still hold.
}, and \( k >0 \), find \( u \in \Hr_\loc^1(\Rb^2) \) such that
\begin{equation}\label{eq:ScatteringProblem}
    \begin{dcases}
        -\Div\plr{\lapc^{-1} \, \nabla u} - k^2 u = f
         & \text{in } \Rb^2
        \\
        \clr{u}_\Gamma = 0, \quad \clr{\lapc^{-1} \, \partial_n u}_\Gamma = g
         & \text{across } \Gamma
        \\
        u \quad k\text{-outgoing}
    \end{dcases}
\end{equation}
and the associated spectral problem:  Find \( (\ell, u) \in \Cb \setminus \Rb_- \times \Hr_\loc^1(\Rb^2) \) such that \( u \not\equiv 0 \) and
\begin{equation}\label{eq:ResonanceProblem}
    \begin{dcases}
        -\Div\plr{\lapc^{-1} \, \nabla u} = \ell^2 u
         & \text{in } \Rb^2
        \\
        \clr{u}_\Gamma = 0, \quad \clr{\lapc^{-1} \, \partial_n u}_\Gamma = 0
         & \text{across } \Gamma
        \\
        u \quad \ell\text{-outgoing}
    \end{dcases}.
\end{equation}
Above, \( \Hr_\loc^1(\Rb^2 ) \coloneqq \setst{u \in \Lr_\loc^2(\Rb^2)}{\forall \chi \in \Cs_\comp^\infty(\Rb^2),\ \chi u \in \Hr^1(\Rb^2)} \) and \( n \colon \Gamma \to \Sb^1 \) is the unit normal vector outward to \( \Omega \).
Given \( X \), we denote \( \clr{X}_\Gamma(\gamma) = \lim_{x \to \gamma^+}X (x) -  \lim_{x \to \gamma^-}X (x) \), for any \( \gamma \in \Gamma \), the jump condition across \( \Gamma \).
The jump conditions \( {[u]}_\Gamma = 0 \) and \( {[\lapc^{-1} \, \partial_{n} u]}_\Gamma = 0 \) will be referred to as the \emph{transmission conditions}.
We say that \( v \) is \emph{\( k \)-outgoing} if it satisfies the outgoing wave condition:
\begin{equation}\label{eq:OWC}
    v(r, \theta) = \sum_{m \in \Zb} w_m(r)\, \e^{\im m \theta} = \sum_{m \in \Zb} c_m \Hu_m(k r)\, \e^{\im m \theta}
\end{equation}
with polar coordinates \( (r,\theta) \) such that \( r > \sup_{x\in\Omega} |x| \), \( \theta \in \Rb / 2\pi\Zb \), \( \Hu_m \) the Hankel function of the first kind of order \( m \), and \( {(c_m)}_{m \in \Zb} \in \Cb^\Zb \).
For a pair \( (\ell, u) \) solution of \cref{eq:ResonanceProblem}, \( \ell \) is called a \emph{scattering resonance} and the function \( u \) is a \emph{resonant mode} associated to \( \ell \).

We define \( P \colon u \mapsto -\Div(\lapc^{-1} \, \nabla u) \) the \( \Lr^2(\Rb^2) \) operator from \cref{eq:ResonanceProblem} with the domain \( \Dc(P) \coloneqq \setst{u \in \Lr^2(\Rb^2)}{\Div(\lapc^{-1} \, \nabla u) \in \Lr^2(\Rb^2)} \)\footnote{One can show that
    \( \Dc(P) = \setst{u \in \Hr^1\plr{\Rb^2}}{
        \Delta \restr{u}{\Omega} \in \Lr^2(\Omega),\ \Delta\restr{u}{\Rb^2 \setminus \overline{\Omega}} \in \Lr^2\plr{\Rb^2 \setminus \overline{\Omega}},\ \clr{\lapc^{-1} \partial_n u}_\Gamma = 0
    } \) (see \cref{lem:P_domain}).
    This second definition will be heavily used in \Cref{sec:resolv_results}.\label{ft:DP}
}.
We also define the local version of the domain \( \Dc_\loc(P) \coloneqq \setst{u \in \Lr_\loc^2(\Rb^2)}{\forall \chi \in \Cs_\comp^\infty(\Rb^2),\ \chi u \in \Dc(P)} \).

\bigskip

For classical cavities (\( \cavc >0 \)), one can show that \cref{eq:ScatteringProblem} is well-posed in \( \Hr_\loc^1(\Rb^2) \), the operator \( (P, \Dc(P)) \) is self-adjoint, its spectrum is real and admits a lower bound.
This allows us in particular to work in the framework of the \emph{black box scattering}~\cite[Definition 4.6]{DyaZwo2019}, where one can check that there is an underlying \emph{black box Hamiltonian} (see~\cref{lem:bbH} for more details).
We can define \( \Rsv \colon k \mapsto \plr{P - k^2}^{-1} \) the resolvent\footnote{\( \Rsv \) is defined on the upper-half of the complex plane (\( \Im(k) > 0 \)).
    Using the \emph{black box scattering} framework (see~\cite{DyaZwo2019}), we can extend the resolvent to \( \Cb \setminus \Rb_- \).
} associated to \( P \).
An asymptotic characterization of the scattering resonances close to the real axis (called \textit{quasi-resonances} \( \ku_m \)) is provided in~\cite{BalDauMoi21}, and with the ``quasimodes to resonances'' result \cite{stefanov1995distribution, stefanov1996neumann,TanZwo98,stefanov1999quasimodes}, it is proved that true resonances \( {(\ell_m)}_m \) are super-algebraically close to quasi-resonances \( \ku_m \).
As a consequence the solution of \cref{eq:ScatteringProblem} blows-up for \( k = \ku_m \) (and the norm of the truncated resolvent \( \Rsv(\ku_m) \) explodes).

Due to the change of sign of \( \lapc \), the ``quasimodes to resonances'' result from the black box scattering framework doesn't directly apply in our case.
First, well-posedness of \cref{eq:ScatteringProblem} in \( \Hr_\loc^1(\Rb^2) \) is not guaranteed as \(u \mapsto -\Div(\lapc^{-1} \, \nabla u) \) does not necessarily define a Fredholm operator (or in other words the coercivity of the associated weak form of \cref{eq:ScatteringProblem} is not guaranteed).
Additionally, spectral requirements on \( P \) to be a black box Hamiltonian are not obvious.
Finally, it is not clear whether there exist resonances close to the real axis that are associated to localized interface modes (potentially related to surface plasmons).

The goal of this paper is to show that the ``quasimodes to resonances'' result still applies for unbounded transmission problems with sign-changing coefficient, and to provide an asymptotic characterization of the scattering resonances.

\begin{remark}\label{rem:2.1}
    \hfill
    \begin{itemize}[leftmargin=*]
        \item The \( k \)-outgoing condition defined in \cref{eq:OWC} is equivalent to \( v \) satisfying the so-called Sommerfeld radiation condition if, and only if, \( k >0 \).
              This outgoing condition is more general, and will be also used for the associated spectral problem \cref{eq:ResonanceProblem}, where one can have \( \ell \)-outgoing solutions with \( \ell \in \Cb \setminus \Rb_- \).
        \item The fact that we look for \( \ell \in \Cb \setminus \Rb_- \) allows us to work on the logarithmic plane. This guarantees analyticity of the Hankel function~\cite[Chapter 10]{Nist}, and it is necessary for spectral problems set in spaces of even dimension~\cite[Definition 4.6]{DyaZwo2019}, typically \cref{eq:ResonanceProblem} set in \(\Rb^2\).
        \item Depending on the polarization (TE / TM), the optical cavity is characterized by a permittivity \( \lapc = \varepsilon \) and a permeability \( \mu = 1 \) or a permeability \( \lapc = \mu \) and a permittivity \( \varepsilon = 1 \).
              Metamaterials are commonly characterized by \( \eps < 0 \) and/or \( \mu <0 \).
              The cavity is embedded in a normalized homogeneous background characterized by \( \mu = 1 \), and \( \eps = 1 \).

        \item \Cref{eq:ScatteringProblem} includes the scattering by a plane wave.
    \end{itemize}
\end{remark}

\subsection{Main result}\label{ssec:main}

Our main goal is to establish the existence of a discrete sequence of scattering resonances close to the positive real axis, which is done in two steps.
First, we derive approximate solutions of the resonance problem \cref{eq:ResonanceProblem} called \emph{quasi-pairs}~\cite[Definition 2.1]{BalDauMoi21} (\cref{thm:quasi_resonances}); then we show that there exist true resonances close to the approximate ones (\cref{thm:resonances}), which rely on showing that the ``quasimodes to resonances'' result from the black box scattering applies for \cref{eq:ResonanceProblem}.
For ease of reading, we (re)define quasi-pairs as follows:

\begin{definition}\label{def:quasi_res}
    A \emph{quasi-pair} for the resonance problem \cref{eq:ResonanceProblem} is formed by a sequence \( {(\lu_m)}_{m \geq 1} \) of real numbers, and a sequence \( {(\uu_m)}_{m \geq 1} \) of complex valued functions that satisfy the following conditions:
    \begin{enumerate}[leftmargin=*]
        \item For any \( m \geq 1 \), the functions \( \uu_m \) are uniformly compactly supported and
              \[
                  \uu_m \in \Dc(P),
                  \quad \text{with} \quad
                  \norm{\uu_m}_{\Lr^2(\Rb^2)} = 1.
              \]

        \item We have the following quasi-pair estimate
              \begin{align}\label{eq:mtodoinfty}
                  \norm{P \uu_m - \lu_m\, \uu_m}_{\Lr^2(\Rb^2)}
                  = \OO\plr{m^{-\infty}},
                  \qquad \text{as } m \to +\infty,
              \end{align}
              with the notation \( a_m = \OO(m^{-\infty}) \) to indicate that for all \( N \in \Nb \), there exists \( C_N > 0 \) such that \( \abs{a_m} \leq C_N\, m^{-N} \), for all \( m \geq 1 \).

        \item Additionally, we say that \( \uu_m \) is localized around \( \Gamma \subset \Rb^2 \) if, for all \( \delta > 0 \), its support is mainly in \( \Gamma_\delta \coloneqq \setst{x \in \Rb^2}{\Dist(x, \Gamma) < \delta} \) neighborhood of \( \Gamma \) in the sense that
              \begin{equation}
                  \norm{\uu_m}_{\Lr^2\plr{\Gamma_\delta }}
                  = 1 - \OO\plr{m^{-\infty}},
                  \qquad \text{as } m \to +\infty.
              \end{equation}
    \end{enumerate}
    We call \( \plr{\uu_m}_{m \geq 1} \) quasi-modes, and \( \plr{\ku_m \coloneqq \sqrt{\lu_m}}_{m \geq 1} \) quasi-resonances.
\end{definition}

\begin{theorem}\label{thm:quasi_resonances}
    If \( \cavc(\gamma) \neq -1 \), for all \( \gamma \in \Gamma \), then we can construct \( \plr{\lu_m, \uu_m}_{m \geq 1} \) quasi-pairs of the resonance problem \cref{eq:ResonanceProblem}. {Moreover, we have \( \lu_m = \plr{\frac{2\pi m}{L}}^2 \, \Lambda\plr{\frac{L}{2\pi m}} \) where \( L \) is the length of the curve \( \Gamma \) and \( \Lambda \in \Cs^\infty\plr{\clr{0, \frac{L}{2\pi}}} \) (see \cref{eq:lambda}).
    The quasi-mode is of the form \( \uu_m  = \exp\plr{ \im \tfrac{2\pi m}{L} \, \Theta }\Phi \) with \( \Theta, \Phi \) smooth functions with respect to \( \frac{L}{2\pi m} \) and \( \Phi \) is exponentially decreasing on both sides of the interface \( \Gamma \) (see \cref{eq:uu})}.
    Additionally, the sign of \( \lu_m \) is given to leading order by the sign of \( 1+\restr{\cavc}{\Gamma}^{-1} \), and \( \plr{\lu_m}_{m \geq 1} \) are independent of the construction.
\end{theorem}

\begin{theorem}\label{thm:resonances}
    If \( \cavc(\gamma) \neq -1 \), for all \( \gamma \in \Gamma \), let \( \plr{\lu_m, \uu_m}_{m \geq 1} \) be the quasi-pairs of \cref{thm:quasi_resonances}.
    Then there exists a sequence of true scattering resonances \( \plr{\ell_m}_{m \geq 1} \) of \cref{eq:ResonanceProblem} close to the quasi-resonances \( \plr{\sqrt{\lu_m}}_{m \geq 1} \) in the sense that
    \[
        \ell_m^2 = \lu_m + \OO\plr{m^{-\infty}},
        \qquad \text{as } m \to +\infty.
    \]
    In addition:
    \begin{itemize}
        \item If \( \cavc(\gamma) < -1 \), for all \( \gamma \in \Gamma \), then \( \plr{\ell_m}_{m \geq 1} \) are scattering resonances with \( \Re\plr{\ell_m} > 0 \) and \( -1 \ll \Im\plr{\ell_m} < 0 \).
        \item If \( -1 < \cavc(\gamma) < 0 \), for all \( \gamma \in \Gamma \), then \(\ell_m \in \im \Rb_+ \), \( \forall m \geq 1\), and \( \plr{\ell_m^2}_{m \geq 1} \) are negative eigenvalues.
    \end{itemize}
\end{theorem}

From \cref{def:quasi_res}, recall that \( a_m = \OO(m^{-\infty}) \) indicates that for all \( N \in \Nb \), there exists \( C_N > 0 \) such that \( \abs{a_m} \leq C_N\, m^{-N} \), for all \( m \geq 1 \).
Then \cref{thm:quasi_resonances,thm:resonances} provide asymptotic estimates, which imply:
\begin{itemize}[leftmargin=*]
    \item \( \plr{\lu_m}_{m \geq 1} \) are independent of the construction in the sense that, if one has two quasi-resonances \( {(\lu_m)}_{m \geq 1}, {(\Mu_m)}_{m \geq 1} \) corresponding to the same integer \( m \), then \( \lu_m - \Mu_m = \OO(m^{-\infty}) \).
          This is demonstrated in \cref{cor:quasiUnicity}.
    \item Estimates \(  \ell_m^2 = \lu_m + \OO\plr{m^{-\infty}}\) naturally provide less accurate result for small \(m\). Some numerical illustrations will be provided in \cref{sec:back_to_sca}.
\end{itemize}
Contrary to the classical cavities (\( \cavc > 0 \)), the value of \( \cavc \) can lead to two different behaviors: from \cref{thm:quasi_resonances,thm:resonances} we only have one sequence of resonances \( \plr{\ell_m}_{m \geq 1} \) close to the positive real axis (in the \( \ell_m \) plane) in the case \( \cavc(\gamma) < -1 \) (where we built \( {(\ku_m)}_m \in \Rb^+ \)), and none in the case \( -1 < \cavc(\gamma) < 0 \) (where we obtained \( {(\ku_m)}_m \in \im \Rb \)), see~\cite{PhD_Moitier_2019, BalDauMoi21}.
From \cref{thm:resonances} one can show that the truncated resolvent explodes at the quasi-resonances, and thus scattering instabilities occur for \cref{eq:ScatteringProblem}.

\begin{corollary}\label{cor:explo_res}
    If \( \cavc(\gamma) < -1 \), for all \( \gamma \in \Gamma \), then there exists a real sequence \( \plr{\ku_m}_{m \geq 1} \) with \( \lim_{m \to +\infty} \ku_m = +\infty \) such that for all \( \chi \in \Cs_\comp^\infty(\Rb^2) \) with \( \chi \equiv 1 \) on an open neighborhood of \( \overline{\Omega} \) and for all \( N \in \Nb \), there exists a constant \( C_N > 0 \),
    \[
        \opnorm{{\chi \Rsv(\ku_m) \chi}} \geq C_N m^N,
        \qquad \forall m \geq 1.
    \]
\end{corollary}

The above results also rely on well-posedness of \cref{eq:ScatteringProblem}, and on establishing that \( P \) is a black box Hamiltonian.
This can be done using the \( \mathtt{T} \)-coercivity framework~\cite{BoChCi12, BCCC16, BoCaCi18}, allowing to compensate for the change of sign of \( \lapc \) and establishing Fredholm properties (and others) under some conditions.
\Cref{sec:resolv_results} and \Cref{app:prop_P} detail those results.
Well-posedness of \cref{eq:ScatteringProblem} in Hadamard's sense leads to the existence of a stability constant \( C(k) >0 \) such that \( \norm{u}_{\Lr^2(\Db(0,\rho))} \le C(k) (\norm{f}_{\Lr^2(\Rb^2)} + \norm{g}_{\Lr^2(\Gamma)}) \), for any open disk such that \( \overline{\Omega} \cup \supp(f) \subset \Db(0,\rho) \), see \cref{lem:Scat_pb_wellposed}.
From \cref{cor:explo_res} we deduce the following:
\begin{corollary}\label{cor:explo_stability}
    If \( \cavc(\gamma) < -1 \), for all \( \gamma \in \Gamma \), then there exists a real sequence \( \plr{\ku_m}_{m \geq 1} \) with \( \lim_{m \to +\infty} \ku_m = +\infty \) such that for all \( N \in \Nb \), there exists a constant \( C_N > 0 \),
    \[
        C(\ku_m) \geq C_N m^N,
        \qquad \forall m \geq 1.
    \]
    \Cref{eq:ScatteringProblem} suffers from scattering instabilities for \( k = \ku_m \).
\end{corollary}

The construction of the real sequence \( \plr{\lu_m}_m \) (consequently \( \plr{\ku_m}_m \)) is the fundamental element in the above results.
To illustrate how to proceed, we present a simple case in \cref{sec:pedago} where all calculations can be done explicitly, and we generalize the approach to arbitrary smooth cavities in \cref{sec:construct_quasi}.

\section{A pedagogical example}\label{sec:pedago}

In this section we consider \cref{eq:ScatteringProblem} set on a circular cavity with constant negative \( \cavc \): \( \Omega \) is a disk of radius \( R > 0 \), and \( \cavc = - \eta^2 \) with \( \eta >0 \).
Taking advantage of the geometry, we look for solution of the form:
\begin{align}\label{eq:uSol}
    u(x) = u(r,\theta) = \sum_{m \in \Zb} u_m(r, \theta) = \sum_{m \in \Zb} w_m(r)\, \e^{\im m \theta},
\end{align}
with \( (r,\theta) \in  \Rb_+ \times  \Rb / 2\pi\Zb \) the polar coordinates corresponding to the Cartesian coordinates \( x \), and \( w_m(r) = \frac{1}{2\pi} \int_{0}^{2\pi} u(r,\theta)\, \e^{-\im m \theta} \di{\theta} \), \( m \in \Zb \), the angular Fourier coefficients.
Similarly, we assume we can write \( f (x) = \sum_{m \in \Zb} f_m(r) \, \e^{\im m \theta} \), for \( x \in \Rb^2 \) with \( f_m \in \Lr_\comp^2(\Rb) \), and we can write \( g(x) = \sum_{m \in \Zb} g_m \, \e^{\im m \theta} \), for \( x \in \Gamma \) with \( g \in \Lr^{2}(\Gamma) \).

\begin{remark}\label{rem:scat_uinc}
    An example where \cref{eq:ScatteringProblem} naturally arises is the scattering by a transparent obstacle of a plane wave.
    If one considers \( u^\inc(x_1, x_2) = \e^{\im k x_2} \), with wavenumber \( k \) and direction \( {(0,1)}^\intercal \), then \cref{eq:ScatteringProblem} is satisfied by the scattered field \( u^\sca \coloneqq u- u^\inc \)  with data \( f^\inc \coloneqq \Div\plr{ \lapc^{-1} \, \nabla u^\inc} + k^2 u^\inc \) and \( g^\inc \coloneqq -\clr{\lapc^{-1} \, \partial_n u^\inc}_\Gamma \).
    Additionally, one can check that \( f_m \) is supported only in the cavity: \( f_m(r) =  k^2 ({1 - \cavc^{-1}})\bJ_m(k \, r) \), \( r \in (0,R) \), where \( \bJ_m \) denotes the Bessel function of the first kind of order \( m \).
    This expansion is obtained using the Jacobi-Anger expansion of \( u^\inc \)~\cite[Eq.\ 10.12.1]{Nist} that converges absolutely on every compact set of \( \Rb^2 \).
\end{remark}

Plugging \cref{eq:uSol} in \cref{eq:ScatteringProblem}, we obtain a family of 1D problems indexed by \( m \in \Zb \): Find \( w_m \in \Hr_\loc^1(\Rb_+,r\di{r}) \) such that
\begin{align}
    \begin{dcases}
        -\frac{1}{r} \partial_r\plr{r \, \partial_r w_m} + \dfrac{m^2}{r^2} w_m - \cavc k^2\, w_m =  {\cavc} f_m
         & \text{in } (0,R)
        \\
        -\frac{1}{r} \partial_r\plr{r \, \partial_r w_m} + \dfrac{m^2}{r^2} w_m - k^2\, w_m = f_m
         & \text{in } (R, +\infty)
        \\
        \clr{w_m}_\Gamma = 0, \quad \clr{\lapc^{-1} \, {w_m}'}_{\{R\}} = g_m
         & \text{across } \{R\}
        \\
        {w_0}'(0) = 0 \quad \text{or} \quad w_m(0) = 0\ \text{for}\ m \neq 0
         & \text{on } \{0\}
        \\
        w_m(r) \propto \Hu_m(k r)
         & r > R
    \end{dcases}\label{eq:wscEq}
\end{align}
with \( \propto \) meaning ``up to a constant''.
For \( m \neq 0 \), the term \( \frac{m^2}{r^2} w_m \) imposes a homogeneous Dirichlet boundary condition at zero~\cite{BerDau99}.
The solution is continuous at \( r = 0 \), using the outgoing wave condition we write
\begin{align}\label{eq:wscSol}
    w_m(r) = \begin{dcases}
                 \alpha_m \frac{\bI_m(\eta \, k\, r)}{\bI_m(\eta \, k\, R)} + f_{\cavc}(r),
                  & \text{if } r \le R,
                 \\
                 \beta_m \frac{\Hu_m(k \, r)}{\Hu_m(k \, R)} + f_R(r),
                  & \text{if } r > R,
             \end{dcases}
\end{align}
with \( \bI_m \) denoting the modified Bessel function of the first kind of order \( m \), and \( f_{\cavc}, f_R \) denoting particular solutions.
Our goal in this section is to investigate the associated operator (in particular the resolvent operator), therefore we do not need to write the particular solutions explicitly.
Above, the coefficients \( (\alpha_m,\beta_m) \) are solution of
\begin{equation}\label{eq:transmissionSystem}
    A_m^\eta(k R) \begin{pmatrix}
        \alpha_m \\ \beta_m
    \end{pmatrix}
    = \begin{pmatrix}
        f_{\cavc}(R) -f_{R}(R) \\
        g_m +  {\cavc}^{-1}f'_{\cavc}(R) -f'_{R}(R)
    \end{pmatrix}
    ,\ A_m^\eta(z) = \begin{pmatrix}
        1
         & -1
        \\
        -\frac{1}{\eta} \frac{\bI_m'(\eta \, z)}{\bI_m(\eta \, z)}
         & -\frac{{\Hu_m}'(z)}{\Hu_m(z)}
    \end{pmatrix}.
\end{equation}
The above system comes from the transmission conditions at \( r = R \).

\begin{remark}
    Since \( k >0 \) and the problem is well-posed for \( \eta \neq 1 \) (see \cref{lem:Scat_pb_wellposed}), coefficients \( (\alpha_m,\beta_m) \) are uniquely defined and \( \det(A_m^\eta(k R))  \neq 0 \), with
    \begin{equation} \label{eq:detM}
        \det(A_m^\eta(z))
        \coloneqq -\eta^{-1} \, \frac{\bI_m'(\eta \, z)}{\bI_m(\eta \, z)}
        - \frac{{\Hu_m}'(z)}{\Hu_m(z)},
        \quad \forall z \in \Cb^*.
    \end{equation}
\end{remark}

Now that we have an explicit expression of \( \plr{A_m^\eta(k R)}_{m \in \Zb} \), we can analyze its behavior for various wavenumbers \( k \) and values of \( \cavc \) (namely \( \eta \)).
For numerical purposes, we truncate \cref{eq:uSol} to order \( M \), leading to consider the sequence of operators \( \plr{A_{-M}^\eta(k R), \ldots, A_0^\eta(k R), \ldots, A_M^\eta(k R)} \).
We choose here \( M = 32 \) and \( R = 1 \).
The resolvent of this spectral numerical scheme is \( \Ab_k^{-1} \) where
\begin{equation}
    \Ab_k \coloneqq \Diag\plr{A_{-M}^\eta(k R), \ldots, A_0^\eta(k R), \ldots, A_M^\eta(k R)}.
\end{equation}
To look at the stability of this scheme, we look at the spectral norm of \( \Ab_k^{-1} \) noted \( \opnorm{\Ab_k^{-1}}_2 \).
\Cref{fig:diskResponse} represents the log plot of \( \opnorm{\Ab_k^{-1}}_2 \) with respect to \( k \), for various values of \( \cavc \).
One observes that there exists a sequence \( {(k_m)}_m \) such that \( \opnorm{\Ab_{k_m}^{-1}}_2 \) peaks when \( \cavc \in (-\infty, -1) \), while \( \opnorm{\Ab_k^{-1}}_2 \) remains bounded when \( \cavc \in (-1, 0) \).
In the first case, the sequence \( \plr{\opnorm{\Ab_{k_m}^{-1}}_2}_{m \ge 1} \) grows exponentially~\cite{MoiSpe2019,HipMoiSpence21}.
We refer to those peaks as \emph{scattering instabilities}.

\begin{figure}[!hbtp]
    \centering
    \subfloat[\( \cavc \in \{-1.3,-1.2,-1.1\} \)]{\includegraphics{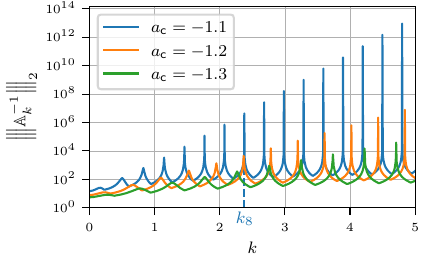}}\hspace{1em}
    \subfloat[\( \cavc \in \{-0.9,-0.8,-0.7\} \)]{\includegraphics{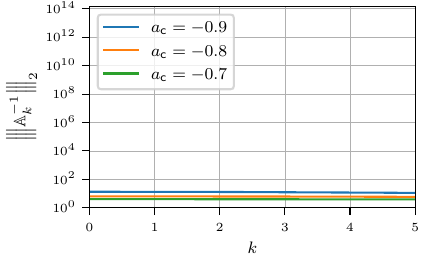}}
    \caption{Semi-log plot of the function \( k \mapsto \opnorm{\Ab_k^{-1}}_2 \) with respect to \( k > 0 \) for \( \cavc \in \{-1.3,-1.2,-1.1\} \) (left), for \( \cavc \in \{-0.9,-0.8,-0.7\} \) (right).
        The value \( k_8 \) marked on the graph corresponds to the reference value used in \cref{fig:disk_mode_-1.1,fig:plotPlasmon}.
    }\label{fig:diskResponse}
\end{figure}

The above results provide the following:
\begin{itemize}\item While \cref{eq:ScatteringProblem} is well-posed for all \( k >0 \), the associated resolvent operator explodes for a sequence of wavenumbers \( \plr{k_m}_{m \geq 1} \).
    \item This phenomenon occurs only for \( \cavc < -1 \).
\end{itemize}
In what follows we investigate the associated spectral problem to identify the resonances causing the scattering instabilities.
We then use semi-classical analysis to characterize the sequence \( \plr{k_m}_{m \geq 1} \), and study their relationship to surface plasmons.

\subsection{Scattering resonances for the disk}\label{sec:DiskRes}

As done in the previous section, \cref{eq:ResonanceProblem} set on a disk can be rewritten as a family of one-dimensional problems indexed by \( m \in \Zb \): Find \( (\ell, w_m) \in \Cb \setminus \Rb_- \times \Hr^1_\loc(\Rb_+, r\di{r}) \setminus \{0\} \), such that
\begin{equation}\label{eq:eigenField_m}
    \begin{dcases}
        -\frac{1}{r} \partial_r\plr{r \, \partial_r w_m} + \dfrac{m^2}{r^2} w_m - \cavc \ell^2 \, w_m = 0
         & \text{in } (0,R)
        \\
        -\frac{1}{r} \partial_r\plr{r \, \partial_r w_m} + \dfrac{m^2}{r^2} w_m - \ell^2 \, w_m = 0
         & \text{in } (R, +\infty)
        \\
        \clr{w_m}_\Gamma = 0, \quad \clr{\lapc^{-1} \, w_m'}_{\{R\}} = 0
         & \text{across } \{R\}
        \\
        {w_0}'(0) = 0 \quad \text{or} \quad w_m(0) = 0\ \text{for}\ m \neq 0
         & \text{on } \{0\}
        \\
        w_m(r) \propto \Hu_m(\ell  r)
         & r > R
    \end{dcases}
\end{equation}
Similarly, we write
\begin{align}\label{eq:wmSol}
    w_m(r) = \begin{dcases}
                 \alpha_m \frac{\bI_m(\eta \, \ell \, r)}{\bI_m(\eta \, \ell \, R)},
                  & \text{if } r \le R,
                 \\
                 \beta_m \frac{\Hu_m(\ell \, r)}{\Hu_m(\ell \, R)},
                  & \text{if } r > R,
             \end{dcases}
\end{align}
however this time, the pair \( (\ell, w_m) \) is solution of \cref{eq:eigenField_m} if, and only if, there exists \( {(\alpha_m, \beta_m)}^\intercal \in \ker\plr{A_m^\eta(\ell R)} \setminus {(0, 0)}^\intercal \), with \( A_m^\eta(\ell R) \) defined in \cref{eq:transmissionSystem}.
Given \( m \in \Zb \), and using \cref{eq:detM}, we define the set of resonances
\begin{equation}\label{eq:Recav_m}
    \Rc[\cavc, R](m) = \setst{\ell \in \Cb \setminus \Rb_-}{\det\plr{A_m^\eta(\ell R)} = 0}.
\end{equation}
Finally, we define the set of resonances of Problem \cref{eq:ResonanceProblem}
\begin{equation}\label{eq:Recav}
    \Rc[\cavc, R] \coloneqq \bigcup_{m \in \Zb} \Rc[\cavc,R](m).
\end{equation}

\begin{remark}
    Given \( \ell \in \Rc[\cavc, R](m) \), one finds \( \alpha_m = c \) and \( \beta_m = c \) with \( c \in \Cb^* \) since the resonant modes are defined up to some normalization.
\end{remark}

\begin{remark}\label{rem:sym_m}
    Since \( \bI_{-m} = \bI_m \) and \( \Hu_{-m} = {(-1)}^m \Hu_m \), for all \( m \in \Zb \), see~\cite[Eq.\ 10.27.1 and 10.4.2]{Nist}, by symmetry all the resonances \( \ell \), corresponding to \( m \neq 0 \), are of multiplicity \( 2 \), and the two associated modes are conjugate, given by \( u_m(r, \theta) \coloneqq w_m(r) \, \e^{ \pm \im m \theta} \).
    It turns out  \( \Rc[\cavc, R] = \bigcup_{m \in \Nb} \Rc[\cavc, R](m) \).
\end{remark}

The resonances set \( {(\Rc[\cavc, R](m))}_m \) defined in \cref{eq:Recav_m} cannot be computed analytically, however one can use contour integration techniques {on \cref{eq:detM}} to compute a subset \( \Rc_N[\cavc, R] \coloneqq \bigcup_{m = 0}^N \Rc[\cavc, R](m) \subset  \Rc[\cavc, R] \) (see~\cite{KraVan00, cxroots}).
\Cref{fig:diskCxroots} represents the set \( \Rc_N[\cavc, 1] \), for the unit disk and for various permittivities \( \cavc \).
The color bar indicates the value of \( m \).

In classical cavities (\( \cavc > 0 \)), resonances of \cref{eq:ResonanceProblem} are split into two categories (at least for \( \cavc > 1 \)~\cite{BalDauMoi21}): inner resonances \( \Rc_\mathsf{inn}[\cavc, R] \) associated to resonant modes essentially supported inside the cavity \( \Omega \), and outer resonances \( \Rc_\out[\cavc,R] \) associated to resonant modes essentially supported in the exterior of the cavity \( \Rb^2 \setminus \overline{\Omega} \).
The inner resonance category includes the so-called \emph{Whispering Gallery Modes} (WGM), associated to resonances \( \ell_{\wgm} \) such that \( -1 \ll \Im (\ell_\wgm) < 0 \)~\cite{ChoKim10,BalDauMoi21}.
In particular the approximation of \cref{eq:ScatteringProblem} can be deteriorated if one chooses \( k  = \Re (\ell_\wgm) \), where those modes can be excited~\cite[Section 6.2]{MoiSpe2019}.
When \( \cavc < 0 \) we split the resonances into three categories.
From \cref{fig:diskCxroots,fig:disk_mode_-1.1,fig:disk_mode_-1.1,fig:disk_mode_-0.9}, we conclude:
\begin{itemize}[leftmargin=*]
    \item The main family of interest represented by `\(+\)' in \cref{fig:diskCxroots} are associated to resonant modes essentially supported on the interface \( \Gamma \) (see \cref{fig:disk_mode_-0.9,fig:plotPlasmon} for an example).
          We refer to those modes as surface plasmons waves (SPW), and we call this family the interface resonances \( \Rc_\pla[\cavc, 1] \).
          We denote the interface resonances \( \plr{\ell_m}_{m \geq 1} \) so that \( \Rc_\pla[\cavc, 1] = \setst{\ell_m}{m \in \Nb^*} \).
          Observe that the interface resonances' nature changes depending on \(\cavc\): if \( \cavc < -1 \), then \(\ell_m\) is a resonance close to the positive real axis with \( \Re\plr{\ell_m} > 0 \) and \( -1 \ll \Im\plr{\ell_m} < 0 \) (in the \(\ell_m\) plane); if \( -1 < {\cavc} < 0 \), then \( \ell_m \in \im\Rb_+ \) so \( \ell_m^2 \) is a negative eigenvalue.
    \item The outer resonances \( \ell_\out \in \Rc_\out[\cavc, 1] \) (\(\ell_\out\) are represented as `\(\bullet\)' in \cref{fig:diskCxroots}) are resonances with a negative imaginary part (in the \(\ell_\out\) plane). The outer resonant modes are essentially supported outside the cavity (see \cref{fig:disk_mode_-0.9,fig:disk_mode_-1.1} for an example).
    \item The last family (represented as `\(\times\)' in \cref{fig:diskCxroots}) corresponds to pure imaginary eigenvalues of the operator \( P \) on \( \Lr^2(\Rb^2) \) (consequently \( \ell^2 \in \Rb_- \)).  The associated modes are essentially supported inside the cavity (like inner resonant modes, see~\cref{fig:disk_mode_-0.9,fig:disk_mode_-1.1} for an example). They contain Whispering Gallery Modes. Because of their particular nature, they are sometimes called bound states \cite[Chapter 1]{DyaZwo2019}, and we denote them \( \ell_\bds \in \Rc_\bds[\cavc, 1] \) (consequently \( \ell^2_\bds \in \Rb_-\)).
\end{itemize}
\begin{figure}[!hbtp]
    \centering
    \subfloat[\( \cavc = -1.1 \)]{\includegraphics{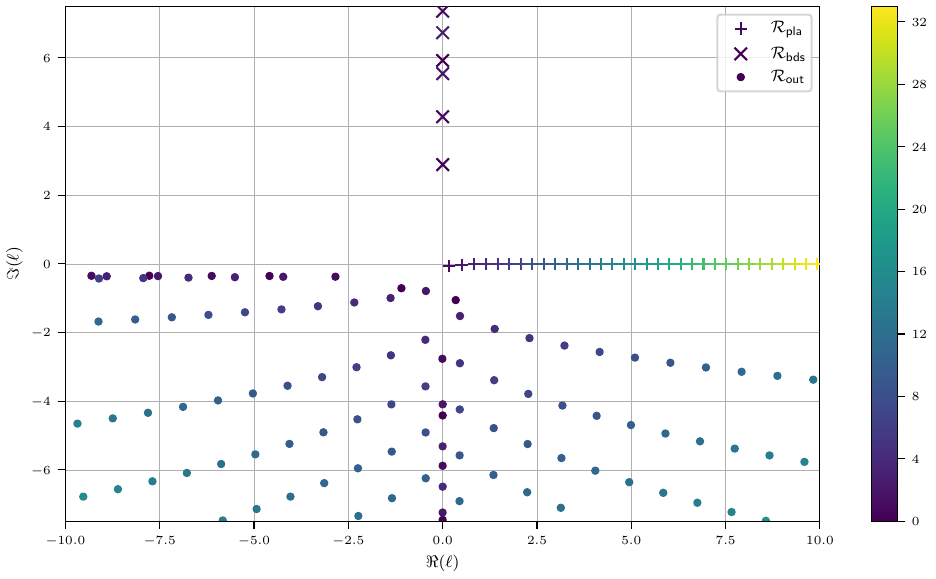}}\\
    \subfloat[\( \cavc = -0.9 \)]{\includegraphics{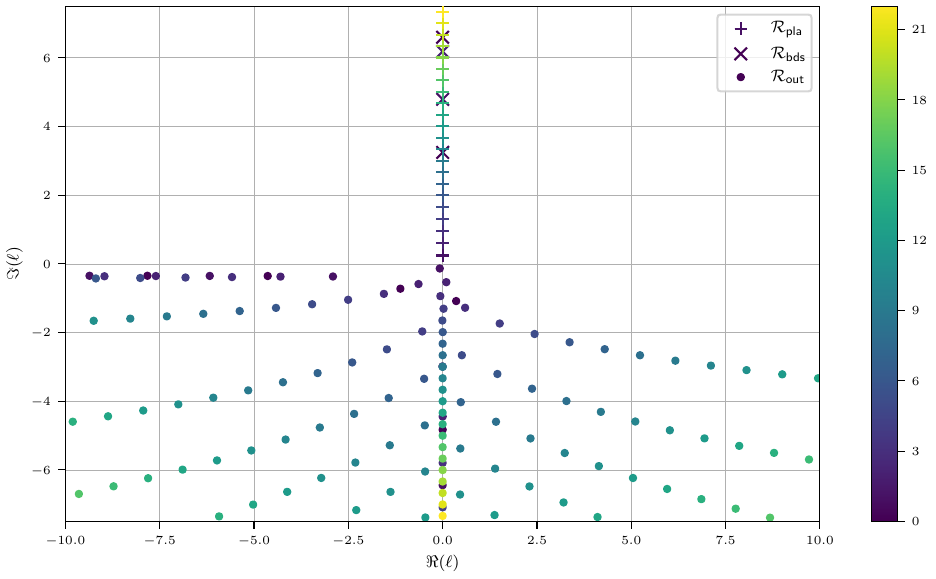}}
    \caption{Graph of the sets \( \Rc_{33}[-1.1, 1] \) (top) and \( \Rc_{22}[-0.9, 1] \) (bottom) in the complex plane \( (\Re(\ell), \Im(\ell)) \) with \( \ell \in \Cb \setminus \Rb_-\).
        Those sets are computed using complex contour integration~\cite{cxroots} on the analytic function \cref{eq:detM}. }\label{fig:diskCxroots}
\end{figure}

\begin{figure}[!hbtp]
    \centering
    \subfloat[\( \ell_8 \approx 2.377 - \im 4.194 \cdot 10^{-6} \)]{\includegraphics{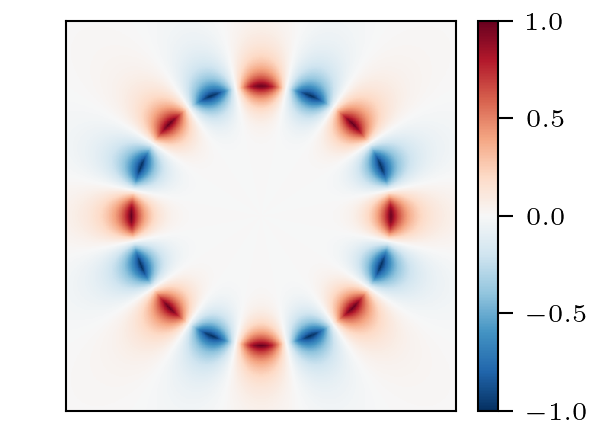}}
    \subfloat[\( \ell_\bds \approx \im 19.330  \)]{\includegraphics{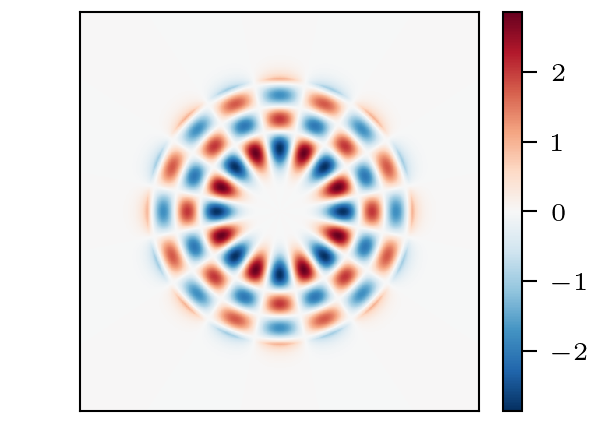}}
    \subfloat[\( \ell_\out \approx 3.174 - \im 4.129 \)]{\includegraphics{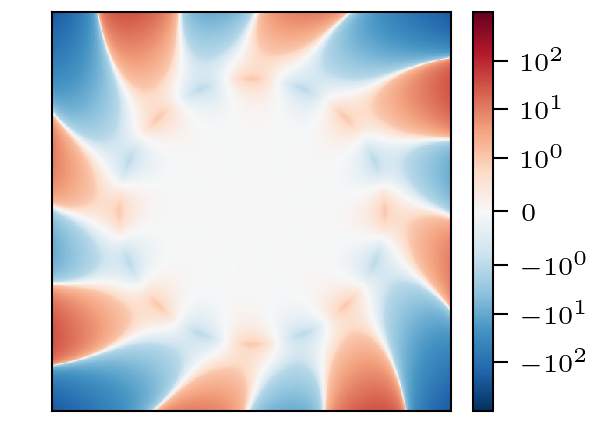}}
    \caption{Real part of some resonant modes \( u_8(r,\theta) \) for \( \cavc = -1.1 \) with their corresponding resonances below.}\label{fig:disk_mode_-1.1}
\end{figure}

\begin{figure}[!hbtp]
    \centering
    \subfloat[\( \ell_8 \approx \im 2.663 \)]{\includegraphics{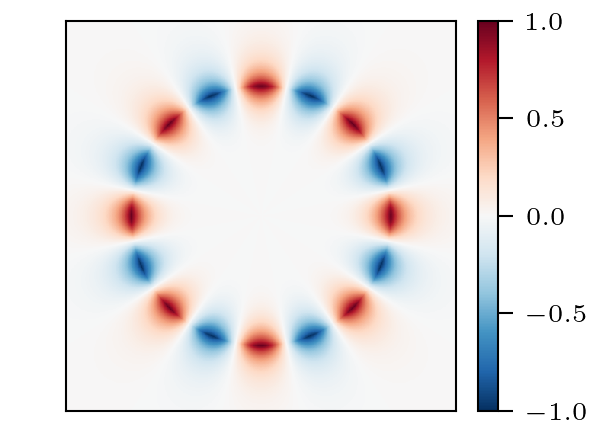}}
    \subfloat[\( \ell_\bds \approx \im 17.718 \)]{\includegraphics{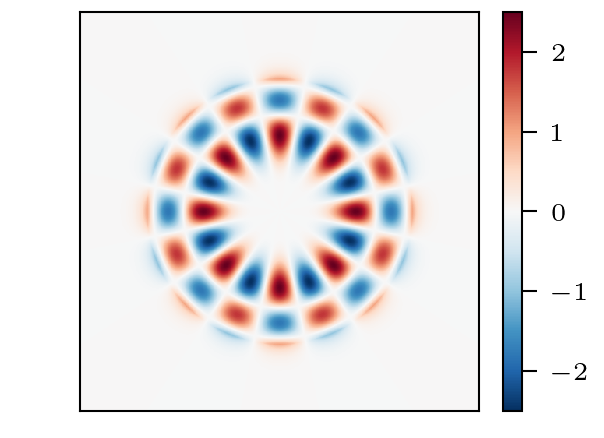}}
    \subfloat[\( \ell_\out \approx 5.229 - \im 2.664 \)]{\includegraphics{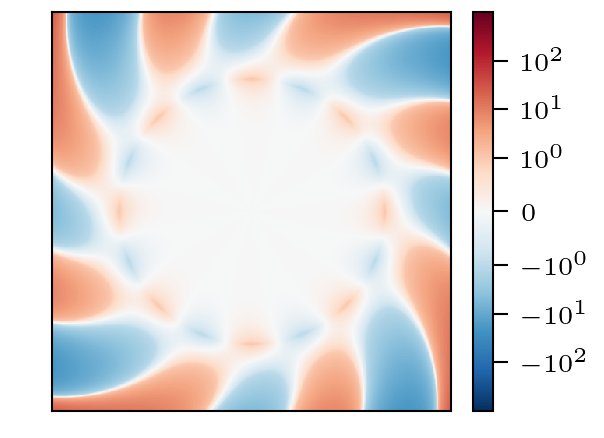}}
    \caption{Real part of some resonant modes \( u_8(r,\theta) \) for \( \cavc = -0.9 \) with their corresponding resonances below.}\label{fig:disk_mode_-0.9}
\end{figure}
\begin{figure}[!hbtp]
    \centering
    \subfloat[\( \cavc = -1.1 \)]{\includegraphics{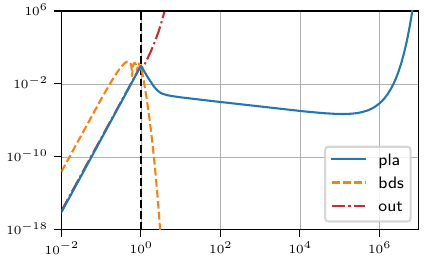}}
    \hspace{1em}
    \subfloat[\( \cavc = -0.9 \)]{\includegraphics{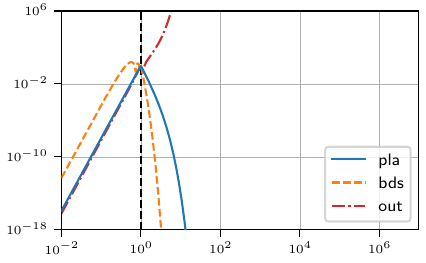}}

    \caption{Log-Log plots of the radial component \( r \mapsto w_8(r) \) of the three types of resonances shown in \cref{fig:disk_mode_-1.1,fig:disk_mode_-0.9} for \( \cavc = -1.1 \) (left) and \( \cavc = -0.9 \) (right).}\label{fig:plotPlasmon}
\end{figure}

In the end, we write \( \Rc[\cavc, R] = \Rc_\out[\cavc, R] \cup \Rc_\bds[\cavc, R] \cup \Rc_\pla[\cavc, R] \).
As mentioned before, the interface resonances are quite peculiar as their nature changes depending on \( \cavc \).
As illustrated in \cref{fig:diskCxroots}, they correspond to complex resonances such that \( \Re\plr{\ell_m} > 0 \) and \( \Im\plr{\ell_m} < 0 \) in the \(\ell_m\) plane when \( \cavc < -1 \), while they are pure imaginary eigenvalues \( \ell_m \in \im \Rb_+\) when \( -1 < \cavc < 0 \).
For the first case, one observes that \(\Re(\ell_m)\) diverges towards \( +\infty \) as \( m \to \infty \), and their negative imaginary part \( \Im(\ell_m) < 0 \) tends to \( 0 \) exponentially fast as \( m \to \infty \).
Additionally, a closer observation gives us that \( \Re (\ell_m) \propto m \).
\Cref{fig:plotPlasmon} represents the behavior of \( w_8 \) for the three types of resonances far from the boundary for \( \cavc \in \{ -1.1, -0.9 \} \).
As discussed above, the support of the bound states and outer resonant modes is mainly inside and outside the cavity, respectively.
The modes associated to interface resonances are locally exponentially decreasing moving away from the interface, which is the mathematical characterization of surface plasmons~\cite{maier2007plasmonics, BCCC16}.
In the next section, we characterize to leading order these interface resonances family \( \plr{\ell_m}_{m \ge 1} \) by performing asymptotic expansion as \( m \to \infty \).
In particular, we will confirm that \( \Re (\ell_m) \propto m \).

\begin{remark}\label{rem:ll2}
    As seen above, it is convenient to identify the change of behavior of the interface resonances using the sign of \( \Re (\ell^2_m) \).
    In what follows we provide asymptotic expansions of \( {(\ell^2_m)}_{m \ge 1} \) instead of the resonances \( {(\ell_m)}_{m \ge 1} \).
\end{remark}

\begin{remark}\label{rem:cx_res}
    Going back to the \cref{eq:ScatteringProblem}, it turns out that the dashed blue line in \cref{fig:diskResponse} corresponds to the real part an interface resonance: \( k_8 = \Re\plr{\ell_8} \approx 2.377 \), and \( \ell_8 \in \Rc_\pla[-1.1, 1] \).
    Additionally, given data associated to \( k > 0 \), the interface modes associated to \( \ell \in \Rc_\pla[-0.9, 1] \) (in other words \( \Re\plr{\ell^2} < 0 \)) cannot be excited as illustrated in \cref{fig:diskResponse}.
    One can also perform the same computations for a lossy circular cavity.
    In that case the interface resonances plunge further into the complex plane (their imaginary part gets more significant in absolute value, moving the resonances away from the real axis).
    Excitation of those resonances is then more difficult to observe.
\end{remark}

\subsection{Interpretation with Schrödinger operator for the disk}\label{sec:heuristicDisk}

From \cref{sec:DiskRes} we found that plasmonic resonances \( {(\ell_m)}_m \) are such that \( \Re (\ell^2_m) \) changes {sign} depending on \( \cavc \) (\emph{i.e.}\ \( \eta \)).
In this section we use asymptotic expansions to explain this change of behavior at leading order.
To do so, we provide an analogy with the Schrödinger operator.
We define \( \ls = m^{-2}\, \ell^2 \), and we rewrite Problem \cref{eq:eigenField_m} as
\begin{equation}\label{eq:eigenField_m_schro}
    \begin{dcases}
        - m^{-2}\frac{1}{r} \partial_r \plr{r\, \partial_r w^\pm_m} + \dfrac{1}{r^2} w^\pm_m = \lapc(r) \, \ls w^\pm_m
         & \text{in } (0,R) \cup (R,+\infty)
        \\
        w^-_m(R) = w^+_m(R)\ \text{and}\ -\eta^{-2}\, \partial_r w^-_m(R) = \partial_r w^+_m(R)
         & \text{across } \{R\}
        \\
        {w_0^-}'(0) = 0\ \text{and}\ w^+_m \in \Ss([R, +\infty))
    \end{dcases}
\end{equation}
with \( \ls \) the new spectral parameter, \( w_m^\pm \) restrictions of \( w_m \) in each material, and \( \Ss(\Rb_+) \) denoting the Schwartz space.
We replace the outgoing wave condition by the requirement that \( w_m^+ \) belongs to the Schwartz space in order to characterize exponentially decreasing behaviors from both sides close to the interface (\emph{i.e.}\ surface plasmons).
To identify this behavior, first we rescale the problem \cref{eq:eigenField_m_schro} by \( \xi = r / R -1 \) such that \( r = R \) corresponds to \( \xi = 0 \).
We then define \( v^\pm_m(\xi) = w^\pm_m(R\, (1+\xi)) \),
satisfying in particular
\[
    -m^{-2}\, \Ls v^\pm_m + V v^\pm_m = a(\xi) \, R^2 \ls \, v^\pm_m
    \qquad \text{in } (-1, 0) \text{ and } (0, +\infty),
\]
where \( \Ls(\xi,\partial_\xi) = \frac{1}{1+\xi} \partial_\xi ((1+\xi)\, \partial_\xi) \) is a positive elliptic operator (Laplacian like) and \( V(\xi) = \frac{1}{{(1+\xi)}^2} \) is a potential.
In that sense, the operator \( v \mapsto  (-m^{-2}\, \Ls  + V ) v \) can be interpreted as a Schrödinger operator.
To construct localized modes at the interface, we consider the principal part of \( - m^{-2} \Ls + V \) with its coefficients frozen at \( \xi = 0 \), corresponding to \( -m^{-2}\partial_\xi^2 + 1 \).
It is then natural to rescale by \( \rho = m \xi \), and the leading order behavior becomes
\begin{equation}\label{eq:eigenField_m_schroScale}
    \begin{dcases}
        -\partial_\rho^2 \vp^- + \vp^- = -\eta^2\, R^2\ls \, \vp^-
         & \text{in}\ (-\infty,  0)
        \\
        -\partial_\rho^2 \vp^+ + \vp^+ = R^2\ls \, \vp^+
         & \text{in}\ (0, +\infty)
        \\
        \vp^-(0) = \vp^+(0)\ \text{and}\ \eta^{-2}\, \partial_\rho \vp^-(0) = \partial_\rho \vp^+(0)
         & \text{across}\ \{0\},
        \\
        \vp^\pm \in \Ss(\Rb_\pm)
    \end{dcases}
\end{equation}
with \( \vp^\pm(\rho) = v^\pm_m(\xi) \).
Note that the condition \( v^-_m(-1) = \vp^-(-m) = 0 \) becomes \( \vp^- \in \Ss(\Rb_-) \) to keep a localized behavior as \( m \to +\infty \).
Solutions of \cref{eq:eigenField_m_schroScale} are given by \( (\ls, \vp^\pm)  = ( R^{-2}(1 - \eta^{-2}), \e^{ -\eta^{\mp 1}\, |\rho|}) \), where the modes are exponentially decreasing on both sides of the interface \( \rho = 0 \).
Back to \cref{eq:eigenField_m_schro}, we have found a pair \( (\lu_m, \underline{w}_m^\pm) \) characterizing \( (\ell^2, w_m^\pm) \), with the leading behavior given by
\begin{equation}\label{eq:wm_asy}
    \lu_m = \frac{m^2}{R^2} \plr{1-\eta^{-2}} + \OO\plr{m} ,
    \quad \text{and} \quad
    \underline{w}_m^\pm(r) = \exp\plr{-\eta^{\mp 1} m\abs{\frac{r}{R}-1}}  + \OO\plr{m^{-1}}.
\end{equation}
We conclude:
\begin{itemize}\item when \( \cavc < -1 \), or \( \eta > 1 \), surface plasmons waves are associated to scattering resonances with \( \Re\plr{\ell} > 0 \) (at first order);
    \item when \( -1 < \cavc < 0 \), or \( 0 < \eta < 1 \), surface plasmons waves are associated to negative eigenvalues with \( \ell  \in \im \Rb_+ \) (at first order).
\end{itemize}
We have then asymptotically characterized SPW by building pairs \( {(\lu_m, \underline{w}_m)}_{m \ge 1} \).
Upon proper justification that \( \underline{w}_m(r) e^{\im m \theta} \in  \Dc(P) \) and that \( k  = \ku_m \coloneqq \sqrt{\lu_m} \) affects the resolvent, the obtained results match the observed behaviors in previous sections, and provide accurate predictions.

The case of the circular cavity with constant \( \cavc \) is quite intuitive, and the leading order computations can be done explicitly.
In the next sections we generalize the approach, to any order, for the general case (arbitrary shaped smooth boundary, and varying coefficients \( \cavc \in \Cs^\infty(\overline{\Omega}; (-\infty, 0) \setminus \{-1\}) \)), and justify the connection between the formal expansions (\Cref{sec:construct_quasi}) and the resolvent operator (as well as the scattering instabilities, consequently) (\Cref{sec:resolv_results}).
To that aim, we will use semi-classical WKB (Wentzel-Kramers-Brillouin) expansions along the interface and matched asymptotic expansions in the transverse direction to the interface in a tubular neighborhood of the interface.
The higher order terms allow to show a super-algebraic behavior of the peaks seen in \cref{fig:diskResponse}, explaining the exponential increase asymptotically.

\begin{remark}\label{rem:trapping}
    The circular cavity allows to clearly separate the resonant modes into three categories, where
    the support is clearly identified.
    Due to this clear separation, we say that this is a \textit{non-trapping} cavity.
    A trapping cavity (typically a crescent shape) may allow a combination of localized interface modes with localized outer modes in region encapsulated by the cavity.
    In the latter, the proposed asymptotic approach may not include the combined modes (the proposed specific scaling above is only adapted for localized plasmonic behaviors).
    For simplicity, all numerical examples will consider non-trapping cavities, that illustrate the fact that scattering resonances exhibit localized behaviors associated to SPW.
\end{remark}

\section{Quasi-pair for unbounded transmission problems with sign-changing coefficient}\label{sec:construct_quasi}

In this section we prove \cref{thm:quasi_resonances} which consists of constructing approximate solutions of the resonance problem \cref{eq:ResonanceProblem}.
Those solutions are called quasi-pairs, in the sense of \cref{def:quasi_res}.
The proof is organized as follows:
\begin{itemize}[leftmargin=*]
    \item We define a tubular neighborhood where we set up the problem, and we define formal expansions (\cref{sec:expan_setup}).

    \item We compute the expansion terms by solving a family of problems indexed by the order of the expansions (\cref{sec:computation_terms}).

    \item We show that the obtained expansions are quasi-pairs in the sense of \cref{def:quasi_res} (\cref{lem:quasiEstimation}), and that the quasi-resonances are independent of the construction (\cref{cor:quasiUnicity}).
          Details are given in \cref{sec:thm:quasi_resonances}.
\end{itemize}
We end \cref{sec:construct_quasi} with comments on the first expansion terms of \( \plr{\lu_m}_{m \geq 1} \).

\subsection{Formal expansion setup}\label{sec:expan_setup}

Recall that \( \Omega \subset \Rb^2 \) is a cavity with smooth boundary \( \Gamma \), see \cref{sec:math_settings}.
Let \( L \) be the length of \( \Gamma \), and \( \eta \coloneqq \sqrt{-\cavc} \) a positive smooth function up to the interface so that we have \( \cavc = -\eta^2 \).
We define a tubular neighborhood \( \Vc_\delta \) of the interface \( \Gamma \).
Let \( \gamma \colon \Tb_L \to \Gamma \) be a counterclockwise curvilinear parametrization of the curve \( \Gamma \) with the notation \( \Tb_L \coloneqq \Rb / L\Zb \).
Let \( n = {(\gamma_2', -\gamma_1')}^\intercal \) be the unit exterior normal to \( \Omega \) and \( \kappa = \det\plr{\gamma', \gamma''} \colon \Tb_L \to \Rb \) be the signed curvature.
We define the open tubular neighborhood, see~\cite{Moon1988}, by
\begin{equation}\label{eq:tubularNeighborhood}
    \Vc_\delta \coloneqq \setst{\gamma(s)+\xi n(s)}{(s,\xi) \in \Tb_L \times (-\delta,\delta)}
\end{equation}
which is schematically represented in \cref{fig:tubularNeighborhood}.

\begin{figure}[!htbp]
    \centering

    \caption{Tubular neighborhood and notations: \( s \) denotes an arc-length parametrization of the curve \( \gamma \), and \( \xi \) is the normal variable.}\label{fig:tubularNeighborhood}
\end{figure}

We now consider the problem:
\begin{align}\label{eq:probX}
    \begin{dcases}
        P u = \lambda u
         & \text{in}\ \Omega \cap \Vc_\delta \ \text{and}\ \plr{\Rb^2 \setminus \overline{\Omega}} \cap \Vc_\delta
        \\
        \clr{u}_\Gamma = 0 \ \text{and}\ \clr{\lapc^{-1}\, \partial_n u}_\Gamma = 0
         & \text{across}\ \Gamma, \quad
        \\
        u = 0
         & \text{on}\ \partial\Vc_\delta
    \end{dcases}
\end{align}
where \( P = -\Div(\lapc^{-1}\, \nabla) \) with \( \lapc \) defined in \cref{eq:fct_lap}.
By \cref{def:quasi_res}, the quasi-pairs are compactly supported therefore the outgoing condition does not play a role in their construction.
We replace in particular the outgoing wave condition by a homogeneous Dirichlet boundary condition in order to construct localized quasi-pairs.

The change of variables from the tubular coordinates \( (s,\xi) \in \Tb_L \times (-\delta,\delta) \) to the Cartesian coordinates \( x \in \Vc_\delta \) is a smooth diffeomorphism for \( 0 < \delta < {(\max_{\Tb_L} |\kappa|)}^{-1} \).
In this tubular coordinate system the operator \( P \) becomes
\begin{equation}\label{eq:opLtubular}
    P = -g^{-1}\, \Div_{s,\xi}\plr{\lapc^{-1}\, G\, \nabla_{s,\xi}}
\end{equation}
where \( g(s,\xi) = 1 + \xi \kappa(s) > 0 \) and
\( G(s,\xi) = \begin{pmatrix}
    {g(s,\xi)}^{-1} & 0        \\
    0               & g(s,\xi)
\end{pmatrix} \).

\bigskip

For the general case, we use a WKB (Wentzel-Kramers-Brillouin) framework~\cite{BabBul61} in order to provide an asymptotic expansion of the spectral parameter as the number of oscillations along the interface \( \Gamma \), denoted \( m \) in \cref{sec:heuristicDisk}, goes to infinity.
We introduce a small parameter \( h > 0 \) (later to be linked to \( m \)) and the ansatz for the quasi-pair \( (\lambda, u) \):
\begin{align}\label{eq:WBK_subs}
    u(s,\xi) = w(s,\xi)\, \exp\plr{\tfrac{\im}{h} \, \theta(s)}
    \quad \text{and} \quad \lambda = h^{-2}\, \ls
\end{align}
where \( \frac{1}{h} \theta \colon [0, L] \to \Cb \) is the fast phase along the interface, \( w \colon \Tb_L \times (-\delta, \delta) \to \Cb \) is the slow amplitude, and \( \ls \in \Cb \) is the spectral parameter.
In order for the function \( u \) in \cref{eq:WBK_subs} to be a smooth function in \( \Vc_\delta \setminus \Gamma \), we need to add the constraint that the function \( s \mapsto \e^{\frac{\im}{h}\theta(s)} \in \Cs^\infty(\Tb_L) \).
The phase function is chosen to be complex to simplify the computations, however we can always put the imaginary part into the amplitude \( w \).
Following~\cite{BabBul61}, we formally expand the unknowns \( w \), \( \theta \), and \( \ls \) with respect to \( h \) as
\begin{align}\label{eq:WKB_expan}
    w(s,\xi) = \sum_{n \ge 0} w_n(s,\xi)\ h^n, \quad
    \theta(s) = \sum_{n \ge 0} \theta_n(s)\ h^n, \quad
    \text{and} \quad
    \ls = \sum_{n \ge 0} \ls_n\ h^n.
\end{align}
The system \cref{eq:probX} with the new unknowns \cref{eq:WBK_subs} becomes
\begin{equation}\label{eq:prob_WKB}
    \begin{dcases}
        \Lc_h[\lapc](w, \theta) = \ls \, w
         & \text{in}\ \Tb_L \times \clr{(-\delta, 0) \cup (0, \delta)}
        \\
        \clr{w}_{\Tb_L \times \{0\}} = 0\ \text{and}\ \clr{\lapc^{-1}\, \partial_\xi w}_{\Tb_L \times \{0\}} = 0
         & \text{across } \Tb_L \times \{0\}
        \\
        w = 0
         & \text{on } \Tb_L \times \{-\delta, \delta \}
    \end{dcases}.
\end{equation}
Above, \( \Lc_h[\lapc](w, \theta) = h^2\, \e^{-\frac{\im}{h}\theta}\, P\plr{w\, \e^{\frac{\im}{h}\theta}} \), and it can be decomposed as
\begin{equation}\label{eq:opL}
    \Lc_h[\lapc](w, \theta) = \Lc_h^3[\lapc](w, \theta, \theta) + \Lc_h^2[\lapc](w, \theta) + \Lc_h^1[\lapc](w)
\end{equation}
where \( \Lc_h^j[\lapc] \) are \( j \)-linear for \( j \in \{1, 2, 3\} \) and
\begin{subequations}
    \begin{align}
        \Lc_h^3[\lapc](w, \theta, \vartheta)
         & = g^{-2}\, \lapc^{-1}\, w\, \partial_s\theta \, \partial_s\vartheta, \label{eq:opL3}
        \\[1ex]
        \Lc_h^2[\lapc](w, \theta)
         & = - h\, \im \plr{g^{-2}\, \lapc^{-1}\, \partial_s w\, \partial_s\theta
            + g^{-1}\, \partial_s\plr{g^{-1}\, \lapc^{-1}\, w\, \partial_s\theta}} , \label{eq:opL2}
        \\[1ex]
        \Lc_h^1[\lapc](w)
         & = -h^2\, g^{-1}\, \plr{\partial_\xi\plr{g\, \lapc^{-1}\, \partial_\xi w} + \partial_s\plr{g^{-1}\, \lapc^{-1}\, \partial_s w}}. \label{eq:opL1}
    \end{align}
\end{subequations}
In the above decomposition, only \( \Lc_h^1[\lapc] \) involves derivatives with respect to \( \xi \).
Since \( g \) (resp.~\( \eta  = \sqrt{ - \cavc} > 0 \)) is a smooth function on \( \Tb_L \times (-\delta, \delta) \) (resp.~\( \Tb_L \times (-\delta, 0] \)), then \( G \) is smooth, and we write the formal Taylor expansions at \( \xi = 0 \):
\begin{equation}\label{eq:formal_gM_eta}
    g(s, \xi) = 1+\xi\kappa(s), \quad
    G(s, \xi) = \sum_{n \ge 0} \frac{\partial_\xi^n G(s,0)}{n!}\, \xi^n , \quad
    \eta(s,\xi) = \sum_{n \ge 0} \frac{\eta_n(s)}{n!}\, \xi^n,
\end{equation}
where \( \eta_n(s) = \partial_\xi^n \eta(s,0) \).
Since \( g \) and \( \eta \) do not vanish on \( \Tb_L \times \{0\} \), the formal expansions of \( g^{-1} \), \( g^{-2} \), and \( \eta^{-2} \) about \( \xi = 0 \) can be computed with \cref{eq:formal_gM_eta}.

\bigskip

Like in \cref{sec:heuristicDisk}, we introduce the scaled variable \( \rho = h^{-1} \xi \) for the normal variable \( \xi \in (-\delta, \delta) \), and we define
\[
    \vp^\pm(s,\rho) = w(s, h\rho),
    \qquad \text{for}\ (s, \rho) \in \Tb_L \times \Rb_\pm.
\]
Then, with \( g = g(s, h \rho) \) we rewrite
\begin{equation}\label{eq:opL1_scaled}
    \Lc_h^1[\lapc](\vp^\pm) = -g^{-1}\, \partial_\rho\plr{\lapc^{-1}\, g\, \partial_\rho \vp^\pm} - h^2\, g^{-1}\, \partial_s\plr{\lapc^{-1}\, g^{-1}\, \partial_s \vp^\pm}.
\end{equation}
Problem \cref{eq:prob_WKB} becomes the formal problem: Find \( {(\vp_n^\pm)}_{n \in \Nb} \in {\Cs^\infty(\Tb_L, \Ss(\Rb_\pm))}^\Nb \), \( \plr{\exp(\frac{\im}{h} \theta_n)}_{n \in \Nb} \in {\Cs^\infty(\Tb_L)}^\Nb \), and \( {(\ls_n)}_{n \in \Nb} \in \Cb^\Nb \) such that
\begin{equation}\label{eq:prob_Ph}
    \begin{cases}
        \Lc_h[\lapc]\big( \sum_n \vp_n^\pm \, h^n, \sum_n \theta_n\, h^n\big) = \big(\sum_n \ls_n\, h^n\big) \big(\sum_n \vp_n^\pm \, h^n\big)
         & \text{in } \Tb_L \times \Rb_\pm^*
        \\[
        1ex]
        \sum_n \vp_n^-(s,0)\, h^n = \sum_n \vp_n^+(s,0)\, h^n
         & \text{on } \Tb_L \times \{0\}
        \\[
        1ex]
        -{\eta_0(s)}^{-2}\, \sum_n \partial_\rho\vp_n^-(s,0)\, h^n = \sum_n \partial_\rho\vp_n^+(s,0)\, h^n
         & \text{on } \Tb_L \times \{0\}
    \end{cases}.
\end{equation}
Note that for simplicity we extend the scaled domain \( \Tb_L \times (-\frac{\delta}{h}, \frac{\delta}{h}) \) to the domain \( \Tb_L \times \Rb \) in order to be independent of \( h \) in \cref{eq:prob_Ph}, and we replace the homogeneous Dirichlet boundary condition on \( \Tb_L \times \{-\frac{\delta}{h}, \frac{\delta}{h}\} \) by the conditions \( \rho \mapsto \vp^\pm(s,\rho) \in \Ss(\Rb_\pm) \) for all \( s \in \Tb_L \).
One can always multiply the quasi-mode by a cutoff function \( \xi \mapsto \chi(\xi) \) to be in the domain \( \Tb_L \times (-\frac{\delta}{h}, \frac{\delta}{h}) \), as done later in \cref{eq:uuDef}.
With \cref{eq:opL} and \cref{eq:formal_gM_eta}, we can formally expand the operators \( \Lc_h^j[-\eta^2] = \sum_{n \ge 0} \Obf_n^{j,-}\, h^n \) and \( \Lc_h^j[1] = \sum_{n \ge 0} \Obf_n^{j,+}\, h^n \) where \( \Obf_n^{j,\pm} \) are independent of \( h \), for \( j \in \{1,2,3\} \).
From Problem \cref{eq:prob_Ph}, we obtain the family of problems \( {(\Pc_n)}_{n \in \Nb} \) by identifying powers of \( h \): Find \( \vp_n^\pm \in \Cs^\infty(\Tb_L, \Ss(\Rb_\pm)) \), \( \exp(\frac{\im}{h} \theta_n) \in \Cs^\infty(\Tb_L) \), and \( \ls_n \in \Cb \) such that
\begin{equation}\label{eq:prob_Pn}
    \begin{dcases}
        \sum_{p \in \Nb_n^4} \Obf_{p_1}^{3,\pm}\plr{\vp_{p_2}^\pm, \theta_{p_3}, \theta_{p_4}}
        + \sum_{p \in \Nb_n^3} \Obf_{p_1}^{2,\pm}\plr{\vp_{p_2}^\pm, \theta_{p_3}}
        + \sum_{p \in \Nb_n^2} \Obf_{p_1}^{1,\pm}\plr{\vp_{p_2}^\pm}
        = \sum_{p \in \Nb_n^2} \ls_{p_1}\, \vp_{p_2}^\pm
        \\[
        1ex]
        \vp_n^-(s,0) = \vp_n^+(s,0) \quad \text{and} \quad -{\eta_0(s)}^{-2}\, \partial_\rho\vp_n^-(s,0) = \partial_\rho\vp_n^+(s,0)
    \end{dcases}
\end{equation}
with the notation \( \Nb_n^d = \{p \in \Nb^d \mid p_1 + \cdots + p_d = n\} \).

\subsection{Computation of the expansion terms}\label{sec:computation_terms}

First, we set some notation that will be useful throughout the rest of the section.

\begin{notation}\label{not:asy_expan}
    We recall that \( \eta_0 = \restr{\eta}{\Gamma} = \sqrt{-\restr{\cavc}{\Gamma}} \).
    Since we assume that \( 1+\restr{\cavc}{\Gamma}^{-1} = 1-\eta_0^{-2} \neq 0 \), we can define the scalar \( \seta = \pm 1 \) to be the sign of \( 1-\eta_0^{-2} \), the functions \( \ftz = \abs{ 1-\eta_0^{-2} }^{-\frac{1}{2}} \), and \( \htz = \frac{\ftz}{\alr{\ftz}} \) where \( \alr{\cdot} \) is the mean along the interface \( \Gamma \) define by
    \[
        \alr{f} \coloneqq \frac{1}{L} \int_{\Tb_L} f(s) \di{s},
        \qquad \forall f \in \Lr^1(\Tb_L).
    \]
\end{notation}

One can obtain the expressions for the \( \Obf_0^{j,\pm} \).
\begin{lemma}\label{lem:Obf0}
    The first terms of the expansions of \( \Lc_h^3[\lapc] \), \( \Lc_h^2[\lapc] \), and \( \Lc_h^1[\lapc] \), are given by
    \begin{align*}
        \Obf_0^{3,-}(\phi, \theta, \vartheta)
         & = -\eta_0^{-2}\, \phi \, \partial_s \theta \, \partial_s \vartheta,
         &
        \Obf_0^{2,-}(\phi, \theta)
         & = 0,
         &
        \Obf_0^{1,-}(\phi)
         & = \eta_0^{-2}\, \partial_\rho^2 \phi,
        \\
        \Obf_0^{3,+}(\phi, \theta, \vartheta)
         & = \phi \, \partial_s \theta \, \partial_s \vartheta,
         &
        \Obf_0^{2,+}(\phi, \theta)
         & = 0,
         &
        \Obf_0^{1,+}(\phi)
         & = - \partial_\rho^2 \phi.
    \end{align*}
\end{lemma}
\begin{proof}
    From the expressions~\eqref{eq:opL3},~\eqref{eq:opL2}, and~\eqref{eq:opL1_scaled} and using the expansions~\eqref{eq:formal_gM_eta} with the change of variable \( \xi = h \rho \) gives the expressions in the lemma.
\end{proof}

Using \cref{lem:Obf0}, we rewrite Problem \( (\Pc_0) \) as: Find \( \vp_0^\pm \in \Cs^\infty(\Tb_L, \Ss(\Rb_\pm)) \), \( \theta_0 \in \Cs^\infty([0, L]) \), and \( \ls_0 \in \Cb \) such that \( (\vp_0^-,\vp_0^+) \not\equiv (0,0) \), \( \exp(\frac{\im}{h} \theta_0) \in \Cs^\infty(\Tb_L) \), and
\begin{align}\label{eq:prob_P0}
    \begin{dcases}
        \partial_\rho^2 \vp_0^- - \plr{{\theta_0'}^2 + \eta_0^2\, \ls_0} \vp_0^- = 0
         & \text{in}\ \Tb_L \times \Rb_-
        \\
        \partial_\rho^2 \vp_0^+ - \plr{{\theta_0'}^2 - \ls_0} \vp_0^+ = 0
         & \text{in}\ \Tb_L \times \Rb_+
        \\
        \vp_0^-(s,0) = \vp_0^+(s,0)
         & \text{on}\ \Tb_L \times \{0\}
        \\
        -{\eta_0(s)}^{-2}\, \partial_\rho\vp_0^-(s,0) = \partial_\rho\vp_0^+(s,0)
         & \text{on}\ \Tb_L \times \{0\}
    \end{dcases}.
\end{align}

\begin{lemma}\label{lem:sol0}
    One can choose \( h = \frac{L}{2\pi m} \) for \( m \in \Nb^* \) so that \( (\vp_0^\pm,  \theta_0,\ls_0 ) \) is given by
    \[
        \ls_0 = \frac{\seta}{\alr{\ftz}^2}, \quad
        \theta_0(s) = \int_0^s \htz(t) \di{t}, \quad
        \text{and} \quad
        \vp_0^\pm(s, \rho) = \alpha(s)\exp\plr{ -|\rho| \, \htz(s)\, {\eta_0(s)}^{\mp 1} },
    \]
    with \( \alpha \in \Cs^\infty(\Tb_L, \Cb^*) \) (and \(\seta, \, \ftz \) defined in \cref{not:asy_expan}), is solution of Problem \( (\Pc_0) \) defined in \cref{eq:prob_P0}.
\end{lemma}
\noindent{}The proof is detailed in \cref{sc_asy_details_0}.

\begin{remark}\label{rem:rescaling}
    \begin{itemize}[leftmargin=*]
        \item If we unravel the scaling and return to tubular coordinates, for \( m \ge 1 \) and \( (s, \xi) \in \Tb_L \times \Rb \), we formally have a pair \( (\lu_m, \uu_m) \)
              \begin{align*}
                  \lu_m         & = \plr{\frac{2\pi m}{L}}^2 \clr{\ls_0 + \OO\plr{m^{-1}}},          \\
                  \uu_m(s, \xi) & = \e^{\im \tfrac{2\pi m}{L}\, \clr{\theta_0(s) + \OO\plr{m^{-1}}}}
                  \begin{dcases}
                      \vp_0^-\plr{s, \tfrac{2\pi m}{L}\xi}, & \text{if}\ \xi \le 0
                      \\[1ex]
                      \vp_0^+\plr{s, \tfrac{2\pi m}{L}\xi}, & \text{if}\ \xi > 0
                  \end{dcases} + \OO\plr{m^{-1}},
              \end{align*}
              which characterizes surface plasmons at leading order.
        \item We remark that the leading order term, solution of \cref{eq:prob_P0}, can be seen as the leading order solution of a planar problem of the form \( -\Div(\lapc^{-1} \nabla v) = \nu v \) on \( \Tb_L \times \Rb \) with \( \lapc(s, y) = - \eta_0^2 \) on the lower half-plane, \( \lapc \equiv 1 \) on the upper half-plane, and \( \nu \in \Rb \).
    \end{itemize}
\end{remark}

\begin{remark}\label{rem:sol0_unicity}
    The construction relies on several choices that are not unique.
    \begin{itemize}[leftmargin=*]
        \item One can choose the main phase to satisfy \( \theta_0' = \htz \) or \( \theta_0' = -\htz \).
              Then one can construct two modes corresponding to \( \uu_m \) and \( \overline{\uu_m} \) (see \cref{rem:mul2}), where \( \overline{\,\cdot \,} \) is the complex conjugate.

        \item The function \( \theta_0 \) is defined up to a constant \( c \).
              Then \( \uu_m \) in \cref{rem:rescaling} is defined up to \( \e^{\im\frac{2\pi m}{L}c} \).
              For simplicity, we consider \( c = 0 \) as we normalize in the end.

        \item The functions \( \vp_0^\pm \) are defined up to a function \( \alpha \colon \Tb_L \to \Cb^* \), which contributes to the phase of \( \uu_m \) and therefore affects the number of oscillations along the interface.
              One can always shift indices so that \( {(\lu_m, \uu_m)}_{m \ge 1-q_\alpha} \), for some \( q_\alpha \in \Zb \), corresponds to a wave with \( m \) oscillations along the interface.

        \item We choose \( h = \frac{L}{2\pi m} \) to simplify the computations however other choices can be made, as long as we have \( h \propto m^{-1} \).
    \end{itemize}
\end{remark}

Now, to compute the higher order term of the expansion, from \cref{eq:prob_Pn}, \cref{lem:Obf0}, and \cref{lem:sol0}, for \( n \ge 1 \), we can rewrite Problem \( (\Pc_n) \) as: Find \( \vp_n^\pm \in \Cs^\infty(\Tb_L, \Ss(\Rb_\pm)) \), \( \exp(\im \, h^{n-1}\, \theta_n) \in \Cs^\infty(\Tb_L) \), and \( \ls_n \in \Cb \) such that
\begin{align}\label{eq:prob_Pn_expand}
    \begin{dcases}
        \partial_\rho^2 \vp_n^- - \htz^2\, \eta_0^2\, \vp_n^- = \plr{2\htz \, \theta_n' + \eta_0^2\, \ls_n} \vp_0^- + \eta_0^2\,S_{n-1}^-
         & \text{in}\ \Tb_L \times \Rb_-
        \\
        \partial_\rho^2 \vp_n^+ - \htz^2\, \eta_0^{-2}\, \vp_n^+ = \plr{2\htz \, \theta_n' - \ls_n} \vp_0^+ - S_{n-1}^+
         & \text{in}\ \Tb_L \times \Rb_+
        \\
        \vp_n^-(s,0) = \vp_n^+(s,0)
         & \text{on}\ \Tb_L \times \{0\},
        \\
        -{\eta_0(s)}^{-2}\, \partial_\rho\vp_n^-(s,0) = \partial_\rho\vp_n^+(s,0)
         & \text{on}\ \Tb_L \times \{0\}
    \end{dcases}
\end{align}
where
\begin{multline}\label{eq:Spm}
    S_{n-1}^\pm = \sum_{p = 1}^{n-1} \ls_{n-p}\, \vp_p^\pm
    - \sum_{p = 0}^{n-1} \Obf_{n-p}^{1,\pm}\plr{\vp_p^\pm}
    - \sum_{p \in \Ib_n^3} \Obf_{p_1}^{2,\pm}\plr{\vp_{p_2}^\pm, \theta_{p_3}}
    \\
    - \Obf_n^{2,\pm}\plr{\vp_0^\pm, \theta_0}
    - \sum_{p \in \Ib_n^4} \Obf_{p_1}^{3,\pm}\plr{\vp_{p_2}^\pm, \theta_{p_3}, \theta_{p_4}}
    - \Obf_n^{3,\pm}\plr{\vp_0^\pm, \theta_0, \theta_0}
\end{multline}
with \( \Ib_n^d = \setst{p \in \ilr{0, n-1}^d}{p_1 + \cdots + p_d = n} \).

\begin{lemma}\label{lem:soln}
    Define \( (\vp_0^\pm, \theta_0, \ls_0) \) according to \cref{lem:sol0}.
    For \( n \ge 1 \), there exists a solution \( (\vp_n^\pm,\theta_n, \ls_n )\in \Cs^\infty(\Tb_L, \Ss(\Rb_\pm)) \times \Cs^\infty(\Tb_L)\times \Cb \) of Problem \( (\Pc_n) \) defined in \cref{eq:prob_Pn_expand}.
    In particular, \( \vp_n^\pm \) is given by
    \[
        \vp_n^\pm(s, \rho) = P_n^\pm(s, \rho) \exp\plr{ -|\rho|\, \htz(s)\, {\eta_0(s)}^{\mp 1} },
    \]
    with polynomials \( P_n^\pm \in \Cs^\infty(\Tb_L, \Pb) \)\footnote{We denote \( \Pb \) the space of polynomial of a single variable with complex coefficients.}.
\end{lemma}
\noindent{}The proof is detailed in \cref{sc_asy_details_n}.

\begin{remark}\label{rem:asy_ls}
    In addition to \cref{rem:sol0_unicity}, \( {(\theta_n)}_{n \ge 0} \) and \( {(\vp_n)}_{n \ge 0} \) are not uniquely defined at each step of the construction.
    However, the sequence \( {(\ls_n)}_{n \ge 0} \) will be unique (see \cref{cor:quasiUnicity}).
\end{remark}

\subsection{Proof of the \texorpdfstring{\cref{thm:quasi_resonances}}{~\ref{thm:quasi_resonances}}}\label{sec:thm:quasi_resonances}

Based on formal series \( \sum_n \vp_n^\pm \, h^n \), \( \sum_n  \theta_n\, h^n \), and \( \sum_n \ls_n\, h^n \) with \( h = \frac{L}{2 \pi m} \), we now construct quasi-pairs in the sense of \cref{def:quasi_res}.
This step is necessary to justify that our formal expansions capture scattering resonances.
First we use Borel's Lemma~\cite[Theorem 1.2.6]{HorI} for \( \ls \) and \( \theta \), and a direct generalization on the Fr\'echet space \( \Cs^\infty(\Tb_L, \Ss(\Rb_\pm)) \)~\cite[Lemma A.5]{BalDauMoi21} for \( \vp^\pm \) to establish:

\begin{lemma}\label{lem:quasiConstruct}
    There exist \( \Phi^\pm \in \Cs^\infty([0, \frac{L}{2\pi}] \times \Tb_L, \Ss(\Rb_\pm)) \), \( \Theta \in \Cs^\infty([0, \frac{L}{2\pi}] \times \Tb_L) \), and \( \Lambda \in \Cs^\infty([0, \frac{L}{2\pi}]) \) such that, for \( N \ge 1 \), \( h \in [0, \frac{L}{2\pi}] \), \( s \in \Tb_L \), and \( \rho \in \Rb_\pm \), we have
    \begin{align*}
        \Phi^\pm(h; s, \rho) & = \sum_{n = 0}^{N-1} \vp_n^\pm(s, \rho)\, h^n + h^N\, R_N^\pm(h; s, \rho)
        \\
        \Theta(h; s)         & = \sum_{n = 0}^{N-1} \theta_n(s)\, h^n + h^N\, R_N^\Theta(h; s)
        \\
        \Lambda(h)           & = \sum_{n = 0}^{N-1} \ls_n\, h^n + h^N\, R_N^\Lambda(h)
    \end{align*}
    where \( R_N^\pm \in \Cs^\infty([0, \frac{L}{2\pi}] \times \Tb_L, \Ss(\Rb_\pm)) \), \( R_N^\Theta \in \Cs^\infty([0, \frac{L}{2\pi}] \times \Tb_L) \), \( R_N^\Lambda \in \Cs^\infty([0, \frac{L}{2\pi}]) \).
\end{lemma}

From those functions, we now define the scalars \( \lu_m \) and the functions \( \uu_m \) in the tubular neighborhood as
\begin{subequations}\label{eq:uuDef}
    \begin{align}
        \lu_m
         & = \plr{\tfrac{2\pi m}{L}}^2 \, \Lambda\plr{\tfrac{L}{2\pi m}} = \plr{\tfrac{2\pi m}{L}}^2\, \sum_{n \in \Nb} \ls_n\, \plr{\tfrac{2\pi m}{L}}^n \label{eq:lambda}
        \\[
        1ex]
        \uu_m(s, \xi)
         & = \chi(\xi)\, \exp\plr{\im \tfrac{2\pi m}{L}\, \Theta\plr{\tfrac{L}{2\pi m}; s}}
        \begin{dcases}
            \Phi^-\plr{\tfrac{L}{2\pi m}; s, \tfrac{2\pi m}{L}\xi}, & \text{if}\ \xi \le 0
            \\[0.5ex]
            \Phi^+\plr{\tfrac{L}{2\pi m}; s, \tfrac{2\pi m}{L}\xi}, & \text{if}\ \xi > 0
        \end{dcases},\label{eq:uu}
    \end{align}
\end{subequations}
where \( \chi \) is a cutoff function, \( \chi \in \Cs_\comp^\infty((-\delta, \delta)) \) and \( \chi \equiv 1 \) on \( \clr{ -\frac{\delta}{2}, \frac{\delta}{2} } \).
In what follows, we establish that \cref{eq:uuDef} is a quasi-pair.
First we have:

\begin{lemma}\label{lem:quasiEstimation}
    The pair \( {(\lu_m, \uu_m)}_{m \ge 1} \) defined in \cref{eq:uuDef} satisfies the following:
    \begin{itemize}
        \item[(i)] \( \uu_m \) is uniformly compactly supported and smooth in \( \Omega \) and \( \Rb^2 \setminus \overline{\Omega} \).

        \item[(ii)] \( \uu_m \) satisfies \( \clr{\uu_m}_\Gamma = \OO\plr{m^{-\infty}} \) and \( \clr{\lapc^{-1}\, \partial_n \uu_m}_\Gamma = \OO\plr{m^{-\infty}} \).

        \item[(iii)] \( \uu_m \) admits the norm expansion
            \[
                \norm{\uu_m}_{\Lr^2(\Rb^2)} = b \, m^{-\frac{1}{2}} + \OO(m^{-\frac{3}{2}})
                \qquad \text{with } b > 0.
            \]

        \item[(iv)] Let \( \underline{R}_m \coloneqq P \uu_m - \lu_m\, \uu_m \) be the reminder defined in \( \Omega \) and \( \Rb^2 \setminus \overline{\Omega} \), then we have
            \[
                \norm{\underline{R}_m}_{\Lr^2(\Omega)} + \norm{\underline{R}_m}_{\Lr^2(\Rb^2 \setminus \overline{\Omega})} = \OO\plr{m^{-\infty}}.
            \]

        \item[(v)] If two quasi-pairs \( {(\lu_m, \uu_m)}_{m \ge 1} \), \( {(\Mu_m,\vu_m)}_{m \ge 1} \) satisfy (i)--{(iv)}, and the quasi-modes have the same leading phase \( \theta_0(s) = \int_0^s \htz(t) \di{t} \) then:
            \[
                \int_{\Rb^2} \uu_m\, \overline{\vu_m} \di{x} = z_0\, m^{-1} + \OO(m^{-2})
                \quad \text{and} \quad
                \int_{\Rb^2} \uu_m\, \vu_m \di{x} = \OO(m^{-\infty})
            \]
            with \( z_0 \in \Cb^* \).
    \end{itemize}
\end{lemma}

\begin{remark}\label{rem:scale_ps}
    Items \emph{(iii)} and \emph{(v)} of \cref{lem:quasiEstimation} give us
    \[
        \int_{\Rb^2} \frac{\uu_m}{\norm{\uu_m}_{\Lr^2(\Rb^2)}}\, \frac{\overline{\vu_m}}{\norm{\overline{\vu_m}}_{\Lr^2(\Rb^2)}} \di{x} = z_0' + \OO(m^{-1}), \quad \text{with } z_0' \in \Cb^*.
    \]
\end{remark}

\begin{remark}
    At this point \( \uu_m \not\in \Dc(P) \) because the transmission conditions are not exactly satisfied, therefore it is not yet a quasi-pair in the sense of \cref{def:quasi_res}.
\end{remark}

\begin{proof}[Proof of \cref{lem:quasiEstimation}]
    Recall that we set \( h = \frac{L}{2\pi m} \), and to simplify notations we denote \( \chi_h \colon \rho \mapsto \chi(\rho h) \), \( \Phi_h^\pm \colon (s, \rho) \mapsto \Phi^\pm(h; s, \rho) \), \( \Theta_h \colon s \mapsto \Theta(h; s) \), and \( \Lambda_h = \Lambda(h) \).

    \textit{(i)} By definition of \( {(\uu_m)}_{m \ge 1} \) in \cref{eq:uu}, \textit{(i)} is satisfied.

    \medskip

    \textit{(ii)} Using \cref{lem:quasiConstruct} and that each functions \( \vp_n^\pm \) satisfies the transmission conditions via \cref{lem:soln}, one can show that \( {[\uu_m]}_\Gamma = \OO(m^{-N}) \) and \( {[\lapc^{-1}\, \partial_n \uu_m]}_\Gamma  = \OO(m^{-N}) \) for all \( N \ge 0 \), which is the definition of \( \OO(m^{-\infty}) \).

    \medskip

    \textit{(iii)} We introduce the weighted \( \Lr^2 \) semi-norm on \( \Tb_L \times \Rb_\pm \)
    \begin{equation}
        \norm{f}_{\Lr_\pm^2[h]}^2 = \int_{\Tb_L} \int_{\Rb_\pm \cap (-\frac{\delta}{h}, \frac{\delta}{h})} |f(s,\rho)|^2\ h(1+\kappa(s)\rho h) \di{\rho}\di{s}.
    \end{equation}
    Form \cref{eq:uuDef}, we obtain
    \[
        \norm{\uu_m}_{\Lr^2(\Rb^2)}^2 = \norm{\chi_h \Phi_h^- \e^{\frac{\im}{h}\, \Theta_h}}_{\Lr_-^2[h]}^2 + \norm{\chi_h \Phi_h^+ \e^{\frac{\im}{h}\, \Theta_h}}_{\Lr_+^2[h]}^2.
    \]
    From \cref{lem:sol0} and \cref{lem:quasiConstruct} for \( N = 1 \), we have
    \begin{align*}
        \Theta_h(s)         & = \int_0^s \htz(t) \di{t} + \theta_1(s) h + h^2\, R_2^\Theta(h; s)
        \\
        \Phi_h^\pm(s, \rho) & = \alpha(s)\, \exp\plr{-|\rho|\, \htz(s)\, {\eta_0(s)}^{\mp 1}} + h\, R_1^\pm(h; s, \rho)
    \end{align*}
    where \( R_2^\Theta \in \Cs^\infty([0, \frac{L}{2\pi}] \times \Tb_L) \) and \( R_1^\pm \in \Cs^\infty([0, \frac{L}{2\pi}] \times \Tb_L, \Ss(\Rb_\pm)) \).
    We deduce that
    \[
        \abs{ \norm{\chi_h \Phi_h^\pm \e^{\frac{\im}{h}\, \Theta_h}}_{\Lr_\pm^2[h]}^2
        - \norm{\chi_h\, \alpha \, \e^{-|\rho|\, \htz \eta_0^{\mp 1}} \e^{\im \, \theta_1}}_{\Lr_\pm^2[h]}^2 }
        \le C_1^\pm h^2
    \]
    for \( C_1^\pm \) some positive constant.
    We write \( \norm{\chi_h\, \alpha \, \e^{-|\rho|\, \htz \eta_0^{\mp 1}} \e^{\im \, \theta_1}}_{\Lr_\pm^2[h]}^2 = I_1^\pm + I_2^\pm + I_3^\pm \),
    \begin{align*}
        I_1^\pm & = h \int_{\Tb_L} \int_{\Rb_\pm} |\alpha(s)|^2 \e^{\mp 2 \rho \, \htz(s) {\eta_0(s)}^{\mp 1}} \e^{-2\Im\theta_1(s)} \di{\rho}\di{s} =  h \int_{\Tb_L} \frac{|\alpha(s)|^2 \e^{-2\Im\theta_1(s)}}{2 \htz(s) {\eta_0(s)}^{\mp 1}} \di{s},
        \\
        I_2^\pm & = h \int_{\Tb_L} \int_{\Rb_\pm} (|\chi(\rho h)|^2-1) |\alpha(s)|^2 \e^{\mp 2 \rho \, \htz(s) {\eta_0(s)}^{\mp 1}} \e^{-2\Im\theta_1(s)} \di{\rho}\di{s},
        \\
        I_3^\pm & = h^2 \int_{\Tb_L} \int_{\Rb_\pm} |\chi(\rho h) \alpha(s)|^2 \e^{\mp 2 \rho \, \htz(s) {\eta_0(s)}^{\mp 1}} \e^{-2\Im\theta_1(s)} \, \kappa(s)\rho \di{\rho}\di{s}.
    \end{align*}
    One can show that \( I_2^\pm = \OO(h^\infty) \) using \cref{lem:int_hoo}.
    Since \( \chi \) is bounded and the function \( (h ; s, \rho) \mapsto {|\alpha|}^2 \e^{\mp 2 \rho \, \htz \eta_0^{\mp 1}} \e^{-2\Im\theta_1} \kappa \rho \) is in \( \Cs^\infty([0, \frac{L}{2\pi}] \times \Tb_L, \Ss(\Rb_\pm)) \) there exists a constant \( C_3^\pm \) such that \( |I_3^\pm| \le C_3^\pm h^2 \).
    Combining the results we get
    \[
        \norm{\uu_m}_{\Lr^2(\Rb^2)}^2 = b^2\, m^{-1} + \OO(m^{-2})
    \]
    with
    \[
        b^2 = \frac{L}{2 \pi} \frac{I_1^+ + I_1^-}{h}
        = \frac{L}{2 \pi} \int_{\Tb_L} |\alpha(s)|^2\, \e^{-2\Im (\theta_1(s))}\ \frac{{\eta_0(s)}^{-1} + \eta_0(s)}{2 \htz(s)} \di{s} > 0.
    \]

    \medskip

    \textit{(iv)} Revisiting the change of variables in tubular coordinates and the scaling, we get
    \begin{subequations}\label{eq:esti_norm_Rm}
        \begin{align}
            \norm{\underline{R}_m}_{\Lr^2(\Omega)}
             & = h^{-2}\, \norm{\e^{\im \, h^{-1}\, \Theta_h} \big( \Lc_h[\lapc](\,\cdot, \Theta_h) - \Lambda_h\big)\, \big( \chi_h\, \Phi_h^-\big) }_{\Lr_-^2[h]},
            \\
            \norm{\underline{R}_m}_{\Lr^2(\Rb^2 \setminus \overline{\Omega})}
             & = h^{-2}\, \norm{\e^{\im \, h^{-1}\, \Theta_h} \big( \Lc_h[\lapc](\,\cdot, \Theta_h) - \Lambda_h\big)\, \big( \chi_h\, \Phi_h^+\big)}_{\Lr_+^2[h]}
        \end{align}
    \end{subequations}
    with \( \Lc_h[\lapc] \) defined in \cref{eq:prob_WKB}.
    \Cref{lem:quasiConstruct} with \( N = 1 \) and \cref{lem:sol0} give the estimation \( \Im\Theta_h = \OO(h) \), so there exists \( c_\Theta > 0 \) such that \( |\e^{\im \, h^{-1}\, \Theta_h}| \le c_\Theta \).
    Introducing the commutator \( \clr{ \Lc_h[\lapc](\,\cdot, \Theta_h), \chi_h} \) of the differential operator \( \Phi \mapsto \Lc_h[\lapc](\Phi, \Theta_h) \) with the scaled cutoff function \( \chi_h \), we deduce from \cref{eq:esti_norm_Rm}
    \begin{equation}
        \norm{\underline{R}_m}_{\Lr^2(\Omega)} \le c_\Theta \, h^{-2}\, \plr{\Nc_- + \Nc_-'}
        \quad \text{and} \quad
        \norm{\underline{R}_m}_{\Lr^2(\Rb^2 \setminus \overline{\Omega})} \le c_\Theta \, h^{-2}\, \plr{\Nc_+ + \Nc_+'}
    \end{equation}
    where \( \Nc_\pm = \norm{\chi_h \big(\Lc_h[\lapc](\,\cdot, \Theta_h) - \Lambda_h\big)\, \Phi_h^\pm}_{\Lr_\pm^2[h]} \) and \( \Nc'_\pm  = \norm{\clr{\Lc_h[\lapc](\,\cdot, \Theta_h), \chi_h} \, \Phi_h^\pm}_{\Lr_\pm^2[h]} \).
    Let's start with \( \Nc_\pm \).
    We write for \( N \ge 1 \),
    \begin{multline*}
        \Lc_h[\lapc](\Phi_h^\pm, \Theta_h)
        = \sum_{n = 0}^{N-1} h^n\, \plr{\Obf_n^{\pm, 3}(\Phi_h^\pm, \Theta_h, \Theta_h) + \Obf_n^{\pm, 2}(\Phi_h^\pm, \Theta_h) + \Obf_n^{\pm, 1}(\Phi_h^\pm)}
        \\
        + h^N \plr{\Rbf_N^{\pm, 3}(h; \Phi_h^\pm, \Theta_h, \Theta_h) + \Rbf_N^{\pm, 2}(h; \Phi_h^\pm, \Theta_h) + \Rbf_N^{\pm, 1}(h; \Phi_h^\pm)}
    \end{multline*}
    where \( \Rbf_N^{\pm, j}(h) \) are \( j \)-linear second order differential operators such that all the coefficients in \( \chi_h\, \Rbf_N^{\pm, j}(h) \) are smooth bounded functions for \( j \in \{1,2,3\} \).
    We use \cref{lem:quasiConstruct} with different \( N \) for each occurrence of \( \Phi_h^\pm \) and \( \Theta_h \), and we obtain
    \begin{multline}\label{eq:remainder_op}
        \Lc_h[\lapc](\Phi_h^\pm, \Theta_h) - \Lambda_h \Phi_h^\pm
        \\
        \begin{aligned}[t]
            = h^N \Bigg[
             & \sum_{n = 0}^{N-1} \sum_{p \in \Nb_{N-n}^3} \Obf_n^{\pm, 3}(R_{p_1}^\pm(h), R_{p_2}^\Theta(h), R_{p_3}^\Theta(h)) + \Rbf_N^{\pm, 3}(h; R_0^\pm(h),R_0^\Theta(h),R_0^\Theta(h))
            \\
             & + \sum_{n = 0}^{N-1} \sum_{p \in \Nb_{N-n}^2} \Obf_n^{\pm, 2}(R_{p_1}^\pm(h), R_{p_2}^\Theta(h)) + \Rbf_N^{\pm, 2}(h; R_0^\pm(h),R_0^\Theta(h))                                \\
             & + \sum_{n = 0}^{N-1} \plr{\Obf_n^{\pm, 1} - \ls_n} R_{N-n}^\pm(h) + \Rbf_N^{\pm, 1}(h; R_0^\pm(h)) - R_N^\Lambda(h)R_0^\pm(h) \Bigg]
        \end{aligned}
    \end{multline}
    where we used the relations in \cref{eq:prob_Pn}, giving us that for all \( Q \in \Nb \)
    \[
        \sum_{p \in \Nb_Q^4} \Obf_{p_1}^{3,\pm}\plr{\vp_{p_2}^\pm, \theta_{p_3}, \theta_{p_4}}
        + \sum_{p \in \Nb_Q^3} \Obf_{p_1}^{2,\pm}\plr{\vp_{p_2}^\pm, \theta_{p_3}}
        + \sum_{p \in \Nb_Q^2} \plr{ \Obf_{p_1}^{1,\pm} - \ls_{p_1}} \vp_{p_2}^\pm
        = 0.
    \]
    The coefficients in the operator \( \chi_h \Lc_h[\lapc](\,\cdot, \Theta_h) \) are smooth bounded functions in \( \Tb_L \times \Rb_\pm \) (see \cref{eq:opL3,eq:opL2,eq:opL1_scaled}).
    From \cref{eq:remainder_op}, we get \( \Nc_\pm \le h^N \, \norm{ F^\pm(h) }_{\Lr_\pm[h]} \) where \( F^\pm \in \Cs^\infty([0, \frac{L}{2\pi}] \times \Tb_L, \Ss(\Rb_\pm)) \), so we have \( \Nc_\pm \le C_N\, h^N \) for \( C_N \) a constant independent of \( h \) as \( h \to 0 \).
    Now, we consider the two commutator norms \( \Nc_\pm' \).
    We observe that the coefficients of the operators \( \clr{ \Lc_h[\lapc](\,\cdot, \Theta_h), \chi_h } \) are zero in \( \Tb_L \times (-\frac{\delta}{2h},0) \) and \( \Tb_L \times (0, \frac{\delta}{2h}) \).
    From this observation, we deduce that
    \[
        {\Nc_\pm'}^2 = \int_{\Tb_L} \int_{I_\pm(h)} |G^\pm(h; s, \rho)|^2 \di{\rho}\di{s}
    \]
    where \( G^\pm \in \Cs^\infty([0, \frac{L}{2\pi}] \times \Tb_L, \Ss(\Rb_\pm)) \) and \( I_\pm(h) \) are as in \cref{lem:int_hoo} for \( \rho = \frac{\delta}{2} \).
    We deduce that \( \Nc_\pm' = \OO(h^\infty) \), and we get \( \norm{\underline{R}_m}_{\Lr^2(\Omega)} + \norm{\underline{R}_m}_{\Lr^2(\Rb^2 \setminus \overline{\Omega})} = \OO\plr{h^{N-2}} \) for all \( N > 1 \).

    \medskip

    \textit{(v)} Let \( {(\theta_n)}_{n \ge 0} \) (resp.~\( {(\vartheta_n)}_{n \ge 0} \)) be a sequence of phases constructed for \( \uu_m \) (resp.~\( \vu_m \)) and \( \alpha \) (resp.~\( \beta \)) the function in \cref{lem:sol0}.
    A similar computation as in \textit{(iii)} gives that \( \int_{\Rb^2} \uu_m\, \overline{\vu_m} \di{x} = z_0\, h + \OO\plr{h^2} \) where
    \begin{align*}
        z_0
         & = \sum_\pm \int_{\Tb_L} \alpha(s) \overline{\beta(s)} \e^{\im\theta_1(s)-\im\overline{\vartheta_1(s)}} \int_{\Rb_\pm} \e^{\mp 2 \rho \, \htz(s) {\eta_0(s)}^{\mp 1}} \di{\rho}\di{s}
        \\
         & = \int_{\Tb_L} \alpha(s) \overline{\beta(s)}\, \e^{\im\theta_1(s)-\im\overline{\vartheta_1(s)}}\ \frac{{\eta_0(s)}^{-1} + \eta_0(s)}{2 \htz(s)} \di{s}.
    \end{align*}
    From the expression of \( \theta_1 \) and \( \vartheta_1 \) in
    Using \cref{lem:theta_1}, we get \( \im\theta_1(s)-\im\overline{\vartheta_1(s)} = -2f(s) - \int_0^s \frac{\alpha'(t)}{\alpha(t)} + \frac{\overline{\beta'(t)}}{\overline{\beta(t)}} \di{t} \) where \( f \) is a real function independent of \( \alpha \) and \( \beta \).
    A derivative computation shows that the functions
    \[
        s \mapsto \alpha(s) \exp\plr{-\int_0^s \frac{\alpha'(t)}{\alpha(t)} \di{t}} \equiv \alpha_0 \in \Cb^*
        \text{ and }
        s \mapsto \beta(s) \exp\plr{-\int_0^s \frac{\beta'(t)}{\beta(t)} \di{t}} \equiv \beta_0 \in \Cb^*
    \]
    are constant so \( z_0 = \alpha_0 \overline{\beta_0} \int_{\Tb_L} \frac{{\eta_0(s)}^{-1} + \eta_0(s)}{2 \htz(s)}\, \e^{-2f(s)} \di{s} \neq 0 \).
    Denoting \( R \) (resp.~\( S \)) the remainder in the construction of \( \uu_m \) (resp.~\( \vu_m \)), we have
    \[
        \int_{\Rb^2} \uu_m\, \vu_m \di{x}
        = \int_{\Tb_L} F(h; s)\ \e^{\im\frac{4\pi m}{L} \theta_0(s)} \di{s}
    \]
    where
    \begin{multline*}
        F(h; s) = \e^{\im R_1^\Theta(h; s) + \im S_1^\Theta(h; s)}
        \\
        \sum_\pm \int_{\Rb_\pm} \chi_u(h\rho) \chi_v(h\rho) R_0^\pm(h; s, \rho) S_0^\pm(h; s, \rho)  h(1+\rho\kappa(s)h)\di{\rho}.
    \end{multline*}
    Note that \( F \in \Cs^\infty([0,\frac{L}{2\pi}] \times \Tb_L) \).
    Since \( \theta_0' = \htz > 0 \), \( \theta_0 \) is a smooth diffeomorphism form \( \Tb_L \) to \( \Tb_L \), we perform the change of variable \( x = \theta_0(s) \)
    \[
        \int_{\Tb_L} F(h; s)\, \e^{\im\frac{4\pi m}{L} \theta_0(s)} \di{s} = \int_{\Tb_L} (\theta_0^{-1})'(x)\, F(h; \theta_0^{-1}(x))\ \e^{\im \frac{4\pi}{L} m x} \di{x}.
    \]
    From the fact that the function \( (h; x) \mapsto (\theta_0^{-1})'(x)\, F(h; \theta_0^{-1}(x)) \in \Cs^\infty([0, \frac{2\pi}{L}] \times \Tb_L) \) and the Riemann-Lebesgue lemma, we get
    \[
        \int_{\Tb_L} (\theta_0^{-1})'(x)\, F(h; \theta_0^{-1}(x))\ \e^{\im \frac{4\pi}{L} m x} \di{x} = \OO(m^{-\infty}).
    \]
\end{proof}

We now add a correction to \( \uu_m \) in order to satisfy the transmission conditions.
Consider \( {(\lu_m, \uu_m)}_{m \ge 1} \) in \cref{eq:uuDef}, satisfying \cref{lem:quasiEstimation}.
We define
\[
    \check{\uu}_m(s,\xi) = \chi(\xi) \begin{dcases}
        0,
         & \text{if } \xi \le 0,
        \\
        \clr{\uu_m}_{\Tb_L \times \{0\}}(s) + \xi \clr{\lapc^{-1}\, \partial_\xi\uu_m}_{\Tb_L \times \{0\}}(s),
         & \text{if}\ \xi > 0,
    \end{dcases}
\]
which gives \( \clr{\uu_m - \check{\uu}_m}_\Gamma = 0 \) and \( \clr{\lapc^{-1}\, \partial_n \plr{\uu_m - \check{\uu}_m}}_\Gamma = 0 \).
Using the regularity and uniform compact support of \( \uu_m - \check{\uu}_m \), we obtain \( \uu_m - \check{\uu}_m \in \Dc(P) \).
Using \cref{lem:quasiEstimation}, we have \( \norm{\check{\uu}_m}_{\Lr^2(\Rb^2)} = \OO(m^{-\infty}) \) therefore \( (P - \lu_m) (\uu_m - \check{\uu}_m) = \OO(m^{-\infty}) \).
We then replace \( \uu_m \) by
\begin{equation}\label{eq:uu_m_rescaled}
    \uu_m = \frac{\uu_m - \check{\uu}_m}{\norm{\uu_m - \check{\uu}_m}_{\Lr^2(\Rb^2)}}
\end{equation}
which now makes \( {(\lu_m, \uu_m)}_{m \ge 1} \) a quasi-pair in the sense of \cref{def:quasi_res}.
To prove \cref{thm:quasi_resonances}, we simply need to show that \( {(\lu_m)}_{m \ge 1} \) are real and independent of the construction.
To that aim we will check that \( {(\ls_n)}_{n\geq 1} \) are real and unique (see \cref{rem:asy_ls}).

\begin{lemma}\label{lem:quasi_lam_mu}
    Let \( {(\lu_m, \uu_m)}_{m \ge 1} \) and \( {(\Mu_m, \vu_m)}_{m \ge 1} \) two quasi-pairs in the sense of \cref{def:quasi_res} corresponding to the same integer \( m \) and having the same leading order phase \( \theta_0 \colon s \mapsto \int_0^s \htz(t) \di{t} \).
    Then we have the following estimate \( \lu_m - \overline{\Mu_m} = \OO\plr{m^{-\infty}} \).
\end{lemma}
\begin{proof}
    Let \( \underline{R}_m \), \( \underline{S}_m \) be the residuals \( \underline{R}_m = P\uu_m - \lu_m\, \uu_m \) and \( \underline{S}_m = P \vu_m - \Mu_m\, \vu_m \).
    By definition, the residuals satisfy \( \norm{\underline{R}_m}_{\Lr^2(\Rb^2)} = \OO(m^{-\infty}) \) and \( \norm{\underline{S}_m}_{\Lr^2(\Rb^2)} = \OO(m^{-\infty}) \).
    Using the symmetry of the operator \( P \), we get
    \[
        \plr{\lu_m - \overline{\Mu_m}} \int_{\Rb^2} \uu_m\, \overline{\vu_m} \di{x}
        = \int_{\Rb^2} \uu_m\, \overline{\underline{S}_m} \di{x} - \int_{\Rb^2} \underline{R}_m\, \overline{\vu_m} \di{x}
        = \OO\plr{m^{-\infty}}.
    \]
    From \cref{rem:scale_ps} one can show that there exists \( z_0 \in \Cb^* \) such that \( \int_{\Rb^2} \uu_m\, \overline{\vu_m} \di{x} = z_0 + \OO(m^{-1}) \).
    Then \( \lu_m - \overline{\Mu_m} = \OO(m^{-\infty}) \) as \( m \to +\infty \).
\end{proof}

\begin{corollary}\label{cor:quasiUnicity}
    The quasi-resonances \( {(\lu_m)}_{m \ge 1} \) are real and are independent of the construction.
\end{corollary}
\begin{proof}
    By applying \cref{lem:quasi_lam_mu} to \( {(\lu_m, \uu_m)}_{m \ge 1} \) and \( {(\lu_m, \uu_m)}_{m \ge 1} \) we get \( \Im\lu_m = \OO(m^{-\infty}) \) which implies that \( \Im\ls_n = 0 \) for all \( n \in \Nb \).
    Then taking \( {(\lu_m, \uu_m)}_{m \ge 1} \) and \( {(\Mu_m, \vu_m)}_{m \ge 1} \) two quasi-pairs in the sense of \cref{def:quasi_res}, from \cref{rem:sol0_unicity}, we can always assume that they have the same leading phase \( \theta_0 \colon s \mapsto \int_0^s \htz(t) \di{t} \) (by taking \( \overline{\vu_m} \) instead of \( \vu_m \)).
    Therefore, \cref{lem:quasi_lam_mu} and the fact that the quasi-resonances are real give us \( \lu_m - \Mu_m = \OO(m^{-\infty}) \), which implies that \( \ls_n = \breve{\mu}_n \).
\end{proof}

Results from \cref{cor:quasiUnicity}, \cref{lem:quasiEstimation} and \cref{eq:uu_m_rescaled} imply \cref{thm:quasi_resonances}.
In the next section we use \cref{thm:quasi_resonances} and the ``quasimodes to resonances'' result to prove \cref{thm:resonances}, \cref{cor:explo_res}, and \cref{cor:explo_stability}. This establishes the connection between the quasi-pairs and the scattering resonances, plus their effect on the scattering instabilities.
We end this section with a few remarks.

\begin{remark}\label{rem:mul2}
    With \cref{cor:quasiUnicity}, given a quasi-pair \( {(\lu_m, \uu_m)}_{m \ge 1} \), we have a second quasi-orthogonal quasi-pair \( {(\lu_m, \overline{\uu_m})}_{m \ge 1} \) with the same quasi-resonance in the sense that, from \textit{(v)} in \cref{lem:quasiEstimation},
    \( \int_{\Rb^2} \uu_m\, \overline{\ \overline{\uu_m}\,} \di{x} = \OO\plr{m^{-\infty}} \).
    The quasi-resonances have an asymptotic multiplicity of \( 2 \), related to the chosen sign of the leading phase \( \theta_0 \) (see \cref{rem:sol0_unicity}).
\end{remark}

\begin{remark}\label{rem:generalization}
    We can generalize the hypothesis of \cref{thm:quasi_resonances} to complex-valued function \( \cavc \in \Cs^\infty(\overline{\Omega}, \Cb^*) \) as long as \( \restr{\cavc}{\Gamma} \neq -1 \) and \( \rho \mapsto \vp_0^\pm(s, \rho) \) in \cref{lem:sol0} are exponentially decreasing for \( \rho \to \pm \infty \).
    In other words we need
    \[
        \Re\plr{ \htz(s)\, {\eta_0(s)}^{\pm 1} } > 0
        \qquad \text{where }
        \ftz(s) = \plr{ 1-{\eta_0(s)}^{-2} }^{-\frac{1}{2}}
        \text{ and }
        \htz = \frac{\ftz}{\alr{\ftz}}
    \]
    and considering the principal branch of the square root.
    However, if \( \cavc \) is complex non-real, the operator \( P \) is non-self-adjoint and \cref{lem:quasi_lam_mu}, \cref{cor:quasiUnicity}, \cref{rem:mul2} are not true anymore.
\end{remark}

\subsection{First expansion terms of \texorpdfstring{\( \lu_m \)}{lu}}\label{sec:symbolic}

We provide here a few terms of the asymptotic expansions of \( \lu_m \) to identify their key features.
The coefficients \( \ls_n \) are computed using formulas in the proof of \cref{lem:soln} via SymPy~\cite{Sympy}, and symbolic codes are available in the GitHub repository~\cite{CM-repo}.

\subsubsection*{General cavity with varying coefficient}
We set the coefficients \( \eta_0(s) = \eta(s, 0) \) and \( \eta_1(s) = \partial_\xi\eta(s, 0) \), we obtain
\begin{equation}\label{eq:lexpan_gen}
    \lu_m = \plr{\frac{2 \pi m}{L}}^2 \frac{\seta}{\alr{\ftz}^2} \clr{1
        -\alr{\frac{\eta_0^2-1}{\eta_0} \kappa + \frac{\eta_1}{{\eta_0}^2 ({\eta_0}^2-1)}} \plr{\frac{L}{2 \pi m}}
        + \OO\plr{m^{-2}}
    }.
\end{equation}

Looking at the first terms one can see that:
\begin{itemize}[leftmargin=*]
    \item The sign comes from the leading term and depends on \(\varsigma = \text{sign}(1-\eta_0^2)\) (see \cref{not:asy_expan}), namely on \( \cavc < -1 \) or \( -1 < \cavc < 0 \).

    \item The curvature \( \kappa \) appears only starting at the second term, it has a weak effect on the expansion.

    \item The terms blow up in the limit \( \eta_0 \to 1 \) (which corresponds to \( \cavc|_\Gamma \to -1 \)).
          This is expected as for \( \cavc \equiv -1 \) since surface plasmons waves correspond to zero eigenvalues.
\end{itemize}
One can compute higher order terms such as \( \ls_2 \), however it becomes rather cumbersome and lengthy to present here (expressions can be found in~\cite{CM-repo}).
We provide below a specific case where the expression \( \ls_2 \) is not too large.

\subsubsection*{Circular cavity  of radius \( R \) with radially varying coefficient \( \eta(r) \)}

Following previous results, we then set \( \eta_0 = \eta(R) \), \( \eta_1 = \partial_r \eta(R) \), \( \eta_2 = \partial^2_r \eta(R) \), and we obtain
\begin{equation}\label{eq:lexpan_disk}
    \lu_m = \frac{m^2}{R^2} \plr{1 - \eta_0^{-2}} \clr{1
        -\plr{\frac{\eta_0^2-1}{\eta_0 R} + \frac{\eta_1}{\eta_0^2 (\eta_0^2-1)}} \plr{\frac{R}{m}}
        + \ls_2 \plr{\frac{R}{m}}^2
        + \OO\plr{m^{-3}}
    }
\end{equation}
where
\[
    \ls_2 =
    - \frac{(\eta_0^4 + 1) (\eta_0^4 - \eta_0^2 + 1)}{2 \eta_0^4 R^2}
    + \frac{\eta_1 (\eta_0^8 + 2 \eta_0^6 - 3 \eta_0^2 + 2)}{5 \eta_0^5 (\eta_0^4 - 1) R}
    - \frac{\eta_1^2 (3 \eta_0^4 + 4 \eta_0^2 - 1)}{2 \eta_0^6 (\eta_0^4 - 1)}
    + \frac{\eta_2}{2 \eta_0^3 (\eta_0^2 - 1)}.
\]

\section{Black Box Scattering framework for unbounded transmission problems with sign-changing coefficient}\label{sec:resolv_results}

\subsection{Proof of \texorpdfstring{\cref{thm:resonances}}{\ref{thm:resonances}}}

In this section, we prove \cref{thm:resonances}.
In the case (i): \(  \cavc(\gamma) < -1 \) for all \( \gamma \in \Gamma \), it is based on the ``quasimodes to quasi-resonances'' result (in particular we follow the theorem of \textsc{Tang} and \textsc{Zworski}~\cite{TanZwo98}) from the black box scattering framework. In the case (ii): \( -1 < \cavc(\gamma) < 0 \) for all \( \gamma \in \Gamma \), the proof is based on the spectral theorem for self-adjoint operators. \\
We start by case (ii). The operator \( P \) is self-adjoint and \( \Rb_- \) only contains discrete eigenvalues (see \cref{lem:P_spectral}). Then the existence of quasi-pairs with \( \lu_m < 0 \), using \cite[Proposition 8.20]{helffer_2013}, gives us
\[
    \Dist\plr{\lu_m, \spe(P)} = \OO\plr{m^{-\infty}}.
\]
Therefore, there exists a sequence \( \plr{\ell_m}_{m \geq 1} \) such that \( \ell_m \in \im\Rb_+ \), \( \ell_m^2 \) is a negative eigenvalue, and \( \ell_m^2 = \lu_m + \OO\plr{m^{-\infty}} \).\\
For the case (i), the proof is a direct consequence of the following elements:
\begin{itemize}[leftmargin=*]
    \item the operator \( (P, \Dc(P)) \) is a \emph{black box Hamiltonian} in the sense of~\cite[Definition 4.1]{DyaZwo2019} (see \cref{lem:bbH});
    \item one can estimate the number of eigenvalues of the reference operator \( P^\sharp \) (a truncated version of the operator \( P \)) defined in \cref{def:trunc_P} (see \cref{lem:weyl_esti}).
          This allows to establish that the set of resonances, which is discrete, is not too large (one can count them).
\end{itemize}

\begin{remark}\label{rem:two_quasi_mode}
    From \cref{rem:mul2}, we have two quasi-orthogonal quasi-pairs and, as in~\cite[Theorem~7.D]{BalDauMoi21}, we have two resonances close to the quasi-resonance.
    This will be illustrated in \cref{sec:back_to_sca}.
\end{remark}

In what follows we prove \cref{lem:bbH} and \cref{lem:weyl_esti}.
Let us denote \( \Db \coloneqq B(0, R_0) \) the open disk of radius \( R_0 \) so that the cavity \( \Omega \) is compactly embedded in \( \Db \).
We denote \( \Ind{\Db} \), \( \Ind{\Rb^2 \setminus \overline{\Db}} \) the restriction on \( \Db \), \( \Rb^2 \setminus \overline{\Db} \), respectively.

\begin{lemma}\label{lem:bbH}
    The operator \( (P, \Dc(P)) \) on \( \Lr^2(\Rb^2) \) is a black box Hamiltonian in the sense of~\cite[Definition 4.1]{DyaZwo2019}, meaning that the following is satisfied:
    \begin{description}[leftmargin=12ex,font=\normalfont,style=nextline,itemsep=1ex]
        \item[\hfill (4.1.1)] we have the orthogonal decomposition \( \Lr^2(\Rb^2) = \Lr^2(\Db) \oplus \Lr^2(\Rb^2 \setminus \overline{\Db}) \).

        \item[\hfill (4.1.4)] the operator \( (P, \Dc(P)) \) is self-adjoint and \( \Ind{\Rb^2 \setminus \overline{\Db}} \Dc(P) \subset \Hr^2(\Rb^2 \setminus \overline{\Db}) \).

        \item[\hfill (4.1.5)] outside \( \Db \) the operator is equal to the Laplacian.

        \item[\hfill (4.1.6)] for all \( v \in \Hr^2(\Rb^2) \) such that \( \restr{v}{B(0, R_0+\eps)} \equiv 0 \) for \( \eps >0 \) then \( v \in \Dc(P) \).

        \item[\hfill (4.1.12)] the operator \( \Ind{\Db}\, {(P+\im)}^{-1}: \Lr^2(\Rb^2) \to \Lr^2(\Db) \) is compact.
    \end{description}
\end{lemma}
\begin{proof}
    The condition (4.1.1) is satisfied by definition.
    The condition (4.1.4) is a consequence of \cref{lem:P_spectral} and \cref{lem:P_domain}.
    The condition (4.1.5) is satisfied by definition of \( (P, \Dc(P)) \): \( \Ind{\Rb^2 \setminus \overline{\Db}} \plr{P u} = -\Delta\plr{ \Ind{\Rb^2 \setminus \overline{\Db}}(u) } \) for \( u \in \Dc(P) \).
    The condition (4.1.6) is a consequence of \cref{lem:P_domain}.
    For the condition (4.1.12), we define \( A \colon \Lr^2(\Rb^2) \to \Lr^2(\Db) \), \( u \mapsto \iota \Ind{\Db}\, {(P+\im)}^{-1} \), with the embedding \( \iota \colon \Hr^1(\Db) \to \Lr^2(\Db) \).
    The operator \( A \) is compact because \( -\im \) is in the resolvent set (\cref{lem:P_spectral}), the projection \( \Ind{\Db} \) goes from \( \Dc(P) \) to \( \Hr^1(\Db) \) (\cref{lem:P_domain}), and \( \iota \) is compact~\cite[Theorem 9.16]{Brezis2011}.
\end{proof}

Now that the operator \( (P, \Dc(P)) \) is a black box Hamiltonian, the solutions of \cref{eq:ResonanceProblem} are well-defined: this means that we have \((\ell, \, u) \in \mathbb{C} \setminus \Rb_- \times \mathcal{D}(P)\) and \cref{eq:ResonanceProblem} fits in the black box scattering framework.
Then we define the reference operator and estimate its eigenvalues.
From \cref{lem:bbH} we deduce that Conditions (1), (2), (3) in~\cite{TanZwo98} are satisfied.
\Cref{lem:weyl_esti} establishes that the last condition, Condition (4) in~\cite{TanZwo98}, is satisfied.
\begin{definition}\label{def:trunc_P}
    From the operator \( (P, \Dc(P)) \) on \( \Lr^2(\Rb^2) \), we define the reference operator \( (P^\sharp, \Dc(P^\sharp)) \) on \( \Lr^2\plr{{(\Rb / R_\sharp \Zb)}^2} \) with \( R_\sharp > R_0 \) by \( P^\sharp \colon u \mapsto -\Div\plr{\lapc_\sharp^{-1} \, \nabla u} \) and
    \[
        \Dc(P^\sharp) = \setst{u \in \Lr^2\plr{{(\Rb / R_\sharp \Zb)}^2}}{P^\sharp u \in \Lr^2\plr{{(\Rb / R_\sharp \Zb)}^2}}
    \]
    where \( \lapc_\sharp = \cavc \Ind{\overline{\Omega}} + \Ind{{(\Rb / R_\sharp \Zb)}^2 \setminus \overline{\Omega}} \) is the ``restriction'' of \( \lapc \) to \( {(\Rb / R_\sharp \Zb)}^2 \).
\end{definition}

\begin{lemma}\label{lem:weyl_esti}
    The reference operator \( (P^\sharp, \Dc(P^\sharp)) \) is self-adjoint, has discrete spectrum, and we have the following weak Weyl estimate
    \[
        \Card \plr{ \spe(P^\sharp) \cap [-\mu, \mu] } = \OO\plr{\mu}
        \qquad \text{for } \mu \geq 1.
    \]
\end{lemma}
\begin{proof}
    The proof that the reference operator is self-adjoint is the similar as in the proof of \cref{lem:P_spectral} (see also \cite[Theorem 4.2]{carvalho2015etude}).
    The spectrum is discrete because \( {(\Rb / R_\sharp \Zb)}^2 \) is a compact set.
    The weak Weyl estimation comes from~\cite[Section 4]{ManMoiVer21}, particularly from Corollary~8.
    The proofs are the same, one simply replaces \( \Hr^1_0(\Omega) \) by the zero mean function in \( \Hr^1\plr{ {(\Rb / R_\sharp \Zb)}^2 } \).
\end{proof}

\Cref{lem:weyl_esti} shows that Condition~(4) in~\cite{TanZwo98} is satisfied with \( n^\sharp = 2 \).
Now that the resonance set is well-defined and characterized by quasi-pairs, we can prove \cref{cor:explo_res}.
We will use the following result:

\begin{lemma}\label{lem:resolvent}
    For \( k \in \Cb \setminus \Rb_- \), we denote \( \Rsv(k) \colon \Lr_\comp^2(\Rb^2) \to \Dc_\loc(P) \) the meromorphic continuation of the resolvent.
    For \( k > 0 \) and \( \chi \in \Cs_\comp^\infty(\Rb^2) \), we define \( \Rsv_\chi(k) \colon \Lr^2(\Rb^2) \to \Dc(P) \) the cut-off resolvent by \( \Rsv_\chi(k) = \chi \Rsv(k) \chi \), as in~\cite[Section 3.2]{MoiSpe2019}.
\end{lemma}
\begin{proof}
    The meromorphic continuation of the resolvent is given by Theorem 4.4 in~\cite{DyaZwo2019} and \cref{lem:bbH}.
\end{proof}

\subsection{Proof of \texorpdfstring{\cref{cor:explo_res}}{\ref{cor:explo_res}}}\label{sec:proof_explo_res}

Let \( \chi \) and \( \widetilde{\chi} \) in \( \Cs_\comp^\infty(\Rb^2) \) with \( \chi,\, \widetilde{\chi} \equiv 1 \) on an open neighborhood of \( \overline{\Omega} \) such that \( \supp\plr{\widetilde{\chi}} \subset \{\chi = 1\} \).
From the definition of the quasi-pair \( \plr{\lu_m,\, \uu_m}_{m \geq 1} \), let \( \ku_m = \sqrt{\lu_m} > 0 \) and \( \vu_m = \widetilde{\chi} \uu_m / \norm{\widetilde{\chi} \uu_m}_{\Lr^2(\Rb^2)} \).
The family \( \plr{\ku_m^2,\, \vu_m}_{m \geq 1} \) is still a quasi-pair, therefore we have \( P\vu_m - \ku_m^2 \vu_m = \underline{R}_m \) with the estimation \( \norm{\underline{R}_m}_{\Lr^2(\Rb^2)} = \OO\plr{m^{-\infty}} \). Due to the fact that \( \supp\plr{\vu_m},\, \supp\plr{\underline{R}_m} \subset \supp\plr{\widetilde{\chi}} \subset \{\chi = 1\} \), we have \( \vu_m = \chi \vu_m \) and \( \underline{R}_m = \chi \underline{R}_m \). We obtain
\begin{align*}
    \plr{P - \ku_m^2} \vu_m = \underline{R}_m
     & \quad\Rightarrow\quad
    \vu_m = \Rsv\plr{\ku_m}\plr{\underline{R}_m}
    \\
     & \quad\Rightarrow\quad
    \chi \vu_m = \chi \Rsv\plr{\ku_m}\plr{\chi \underline{R}_m}
    \\
     & \quad\Rightarrow\quad
    \vu_m = \Rsv_\chi\plr{\ku_m}\plr{\underline{R}_m}.
\end{align*}
We deduce that for all \( N \in \Nb \), there exists \( C_N > 0 \) such that
\[
    1
    = \norm{\vu_m}_{\Lr^2(\Rb^2)}
    = \norm{\Rsv_\chi(\ku_m) \plr{\underline{R}_m}}_{\Lr^2(\Rb^2)}
    \leq \opnorm{\Rsv_\chi(\ku_m)} C_N^{-1} m^{-N}
\]
which gives the result.

\begin{remark}\label{rem:true_to_quasi}
    Results from \cite{stefanov2000resonances} hold as well in this case: from families of resonances close to the positive real axis, we can create quasi-resonances.
\end{remark}

\subsection{Proof of \texorpdfstring{\cref{cor:explo_stability}}{\ref{cor:explo_stability}}}

Now let \( \ku_m \coloneqq \sqrt{\lu_m} > 0 \) for \( m \ge 1 \).
Results from \cref{sec:construct_quasi} give us \( -\Div(\lapc^{-1}\, \nabla \uu_m) - \ku_m^2 \uu_m = \underline{R}_m \) with the remainder estimate \( \norm{\underline{R}_m}_{\Lr^2(\Rb^2)} = \OO(m^{-\infty}) \).
\Cref{lem:Scat_pb_wellposed} with \( g = 0 \) and \( f = \underline{R}_m \), gives us
\[
    \norm{\uu_m}_{\Lr^2(\Rb^2)} \le C(\ku_m)\ \norm{\underline{R}_m}_{\Lr^2(\Rb^2)}.
\]
Since \( \norm{\uu_m}_{\Lr^2(\Rb^2)} = 1 \) by definition and for all \( N \ge 1 \), there exists \( \widetilde{c}_N > 0 \) such that \( \norm{\underline{R}_m}_{\Lr^2(\Rb^2)} \le \widetilde{c}_N m^{-N} \) then \( \widetilde{c}_N^{-1}\, m^N \le C(\ku_m) \), for all \( m \ge 1 \).

\section{Numerical illustration of metamaterial scattering resonances}\label{sec:back_to_sca}

Using \Cref{thm:resonances} and \cref{cor:explo_res} (proved in \Cref{sec:resolv_results}), we have shown that there exist scattering resonances located close to the positive real axis when \( \cavc(\gamma) < -1 \) for all \( \gamma \in \Gamma \).
Choosing \( k = \Re (\ell) \) will lead to scattering instabilities for \cref{eq:ScatteringProblem}.
In what follows we provide several numerical examples showing the norm of the resolvent operator exploding close to scattering resonances.
First we use the Finite Element Method (FEM) to compute the scattering resonances \( \ell \) of the cavity close to the real axis (Step 1), then we compute the norm of the discretized cut-off resolvent operator for various \( k \) (Step 2).
We also compare the scattering resonances with the first terms of the obtained asymptotic expansions (Step 3).
We provide details about the steps below.
We consider three cases:
\begin{itemize}
    \item[(A)] Circular cavity of radius \( 1 \) with constant \( \cavc \equiv -1.1 \) as represented in \cref{fig:disk_epsCst_geo}.

    \item[(B)] Circular cavity of radius \( 1 \) with linearly varying permittivity \( \cavc^{\lapc_\mathsf{m}, \lapc_\mathsf{M}} \colon (x, y) \mapsto \frac{\lapc_\mathsf{m} + \lapc_\mathsf{M}}{2} + \frac{\lapc_\mathsf{M} - \lapc_\mathsf{m}}{2}x \), with \( (\lapc_\mathsf{m}, \lapc_\mathsf{M}) = (-1.2, -1.1) \), as represented in \cref{fig:disk_epsVar_geo}.

    \item[(C)] Peanut cavity with constant \( \cavc \equiv -1.1 \) as represented in \cref{fig:peanut_epsCst_geo}.
        The peanut boundary is parameterized by \( r(\theta) = 1 - \frac{3}{10}\cos(2\theta) \) with \( \theta \in \Rb / 2\pi\Zb \).
\end{itemize}

\begin{figure}[!htbp]
    \centering
    \subfloat[(A) Disk]{\label{fig:disk_epsCst_geo}
        \includegraphics{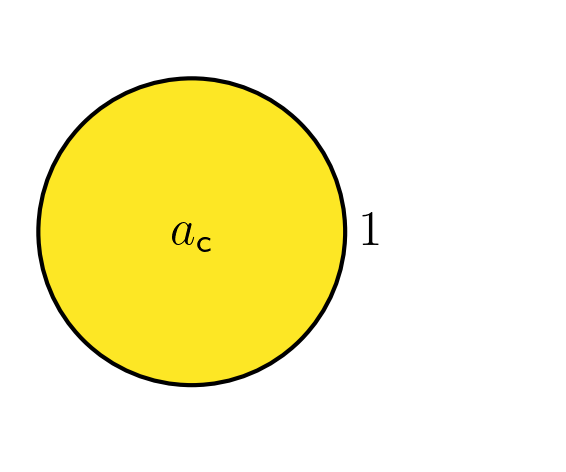}}\hspace{1em}\subfloat[(B) Disk, varying \( \cavc \)]{\label{fig:disk_epsVar_geo}
        \includegraphics{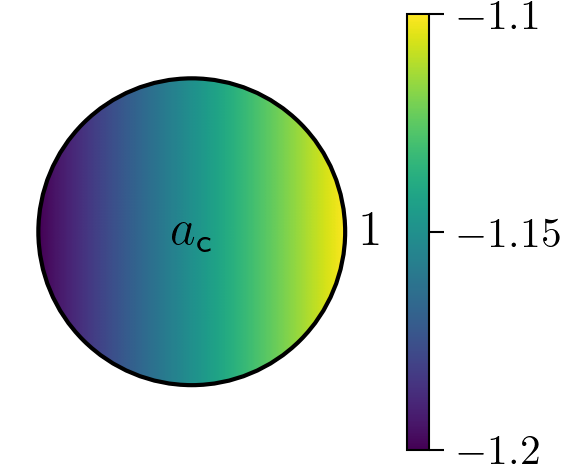}}\hspace{1em}\subfloat[(C) Peanut]{\label{fig:peanut_epsCst_geo}
        \includegraphics{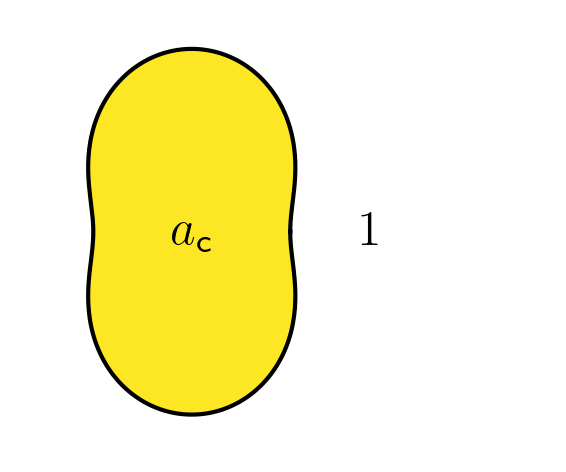}}\caption{Sketch representing the three considered configurations (A), (B), and (C), for the numerical illustration.
    }\label{fig:case_geo}
\end{figure}

\paragraph{\textbf{Step 1: computing resonances}}
In order to solve \cref{eq:ResonanceProblem}, we truncate the computational domain with a circular \emph{perfectly matched layer} (PML) as done in~\cite{PhD_Moitier_2019} (represented in green in \cref{fig:meshes}), and we consider \( \mathtt{T} \)-conforming meshes (ad hoc locally symmetric meshes along the interface \( \Gamma \)) to guarantee FEM optimal convergence and avoid spurious eigenvalues~\cite{CaChCi17, BoCaCi18}.
In practice, we build such meshes using GMSH~\cite{GMSH} and consider quadrangular elements of degree \( 3 \) embedded in a tubular neighborhood as defined in \cref{eq:tubularNeighborhood}.
We build a circular PML with radii \( r_0 = 1.25 ,\, r_1 = r_0 + 0.25 \) for the disk, and \( r_0 = 1.25 \times 1.3 , \, r_1 = r_0 + 0.25 \times 1.3 \) for the peanut.
After those transformations, the scattering resonances are approximated by the eigenvalues of the resulting FEM matrix, at least in a sector below the real axis (where the angle of the sector depends on the PML parameters).

\begin{figure}[!htbp]
    \centering
    \subfloat[Mesh for the disk]{\label{fig:mesh_disk}
        \includegraphics[width=0.25\textwidth]{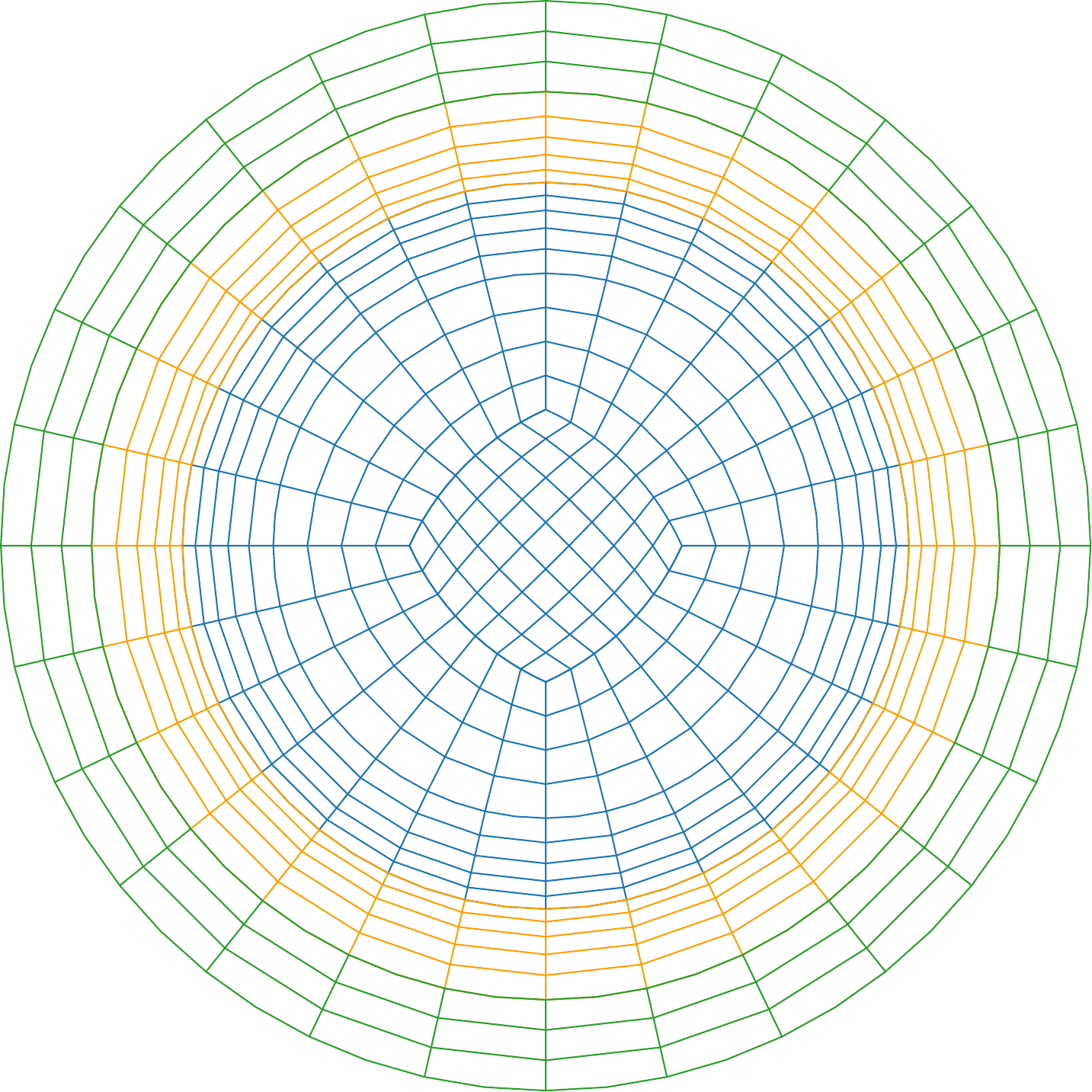}}\hspace{2em}\subfloat[Mesh for the peanut]{\label{fig:mesh_peanut}
        \includegraphics[width=0.25\textwidth]{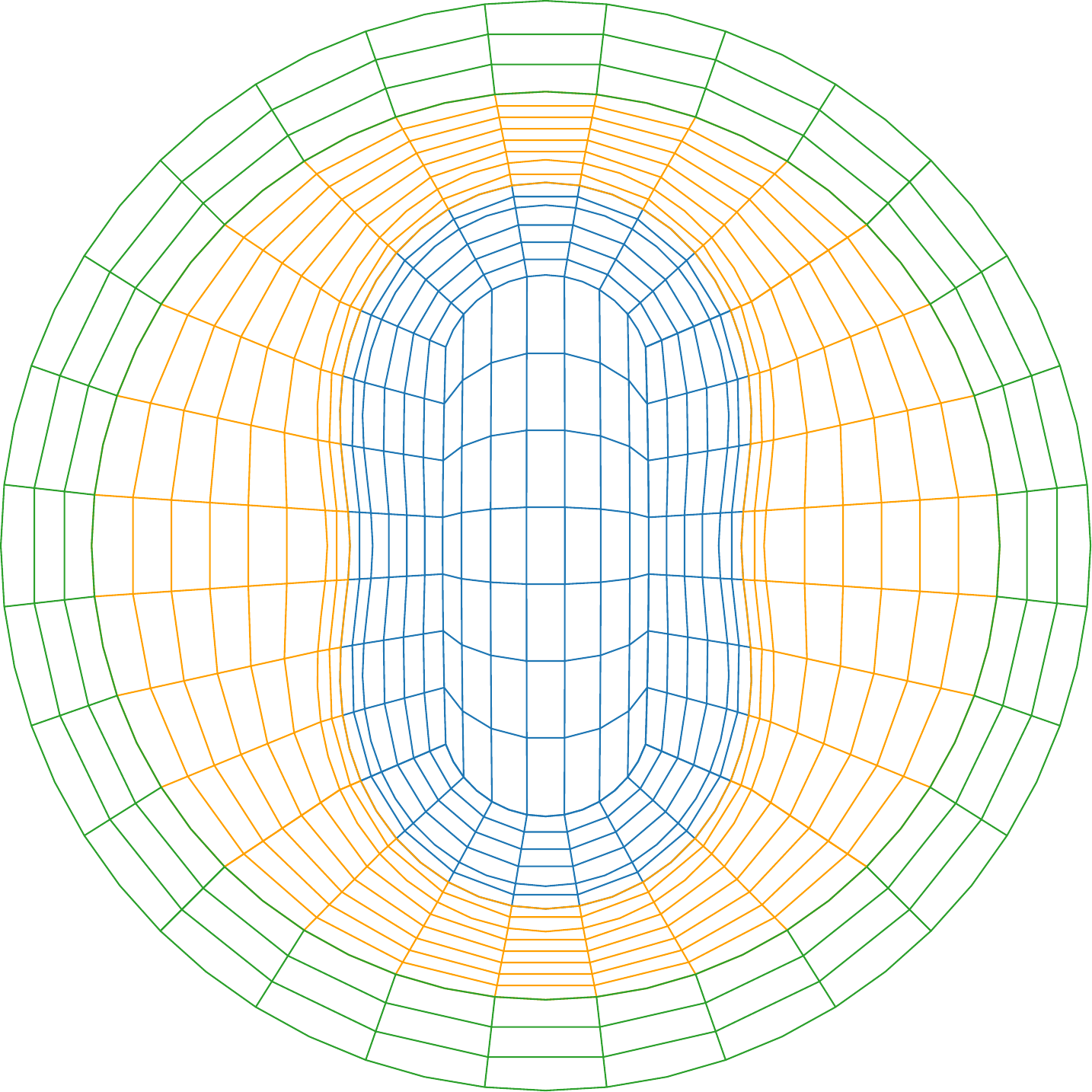}}\caption{Structured mesh for the circular cavity (left) and the peanut shape cavity (right).
        The cavity is represented in blue, the exterior domain in orange, and the PML in green.
        The mesh is locally symmetric along the interface \( \Gamma \).
        Meshes use quadrangular elements of degree \( 3 \), and we use finite elements of degree 8 in the computations (adding degrees of freedom in each element, not represented here).
    }\label{fig:meshes}
\end{figure}

The FEM computations are done using finite elements of degree \( 8 \) using \texttt{XLife++}~\cite{xlifepp}, leading to \( 33\, 713 \) degrees of freedom for all three cases.
\Cref{tab:resonances} contains computed scattering resonances values \( \ell_\fem \) for various numbers of curvilinear oscillations \( m \in \{3, 6, 12\} \), for the three cases.
As mentioned in \cref{rem:sol0_unicity,rem:mul2,rem:two_quasi_mode}, for a given \( m \), there are two resonances.
We plot in \cref{fig:modes} the two associated resonant modes for cases (B) and (C) associated to \( m=12 \).
One can observe that the size of angular oscillations changes when \( \cavc \) varies (case (B)).
Additionally, one observes that associated computed modes exhibit localized behaviors, as induced by surface plasmons waves.

\begin{table}[!htbp]
    \renewcommand{\arraystretch}{1.2}
    \renewcommand{\tabcolsep}{1em}
    \centering
    \begin{tabular}{llll}
        \toprule
        \( \ell_\fem \)
         & \( m = 3 \)
         & \( m = 6 \)
         & \( m = 12 \)
        \\\midrule
        (A)
         & \( 1.1472 - \im 10^{-2} \)
         & \( 2.072 - \im 10^{-3} \)
         & \( 3.89308 - \im 10^{*} \)
        \\
         & \( 1.1472 - \im 10^{-2} \)
         & \( 2.072 - \im 10^{-3} \)
         & \( 3.89308 - \im 10^{*} \)
        \\[1ex]
        (B)
         & \( 0.966 - \im 10^{-1.6} \)
         & \( 2.0681 - \im 10^{-2} \)
         & \( 4.21203 - \im 10^{-5} \)
        \\
         & \( 0.966 - \im 10^{-1.6} \)
         & \( 2.0681 - \im 10^{-2} \)
         & \( 4.21231 - \im 10^{-5} \)
        \\[1ex]
        (C)
         & \( 0.46 - \im 10^{-0.86} \)
         & \( 1.5455 - \im 10^{-1.91} \)
         & \( 3.2954955 - \im 10^{-3.32} \)
        \\
         & \( 0.93 - \im 10^{-1.49} \)
         & \( 1.6912 - \im 10^{-2.65} \)
         & \( 3.2990404 - \im 10^{-4.44} \)
        \\\bottomrule
    \end{tabular}
    \caption{Approximate value of the scattering resonances \( \ell_\fem \) in the three cases and for \( m \in \{3, 6, 12\} \).
        The number of digits displayed is evaluated using an estimated numerical error, and we have put a ``\( * \)'' when the value is below the estimated error.}\label{tab:resonances}
\end{table}

\begin{figure}[!htbp]
    \centering
    \subfloat{\includegraphics{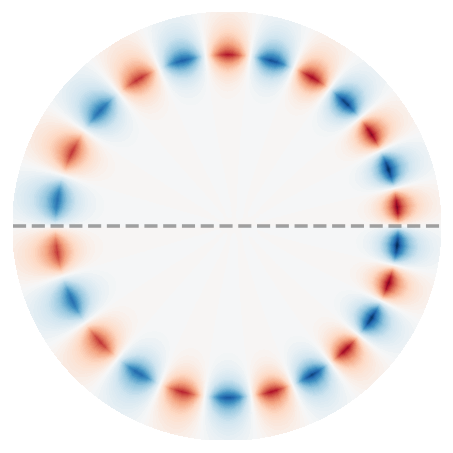}}
    \subfloat{\includegraphics{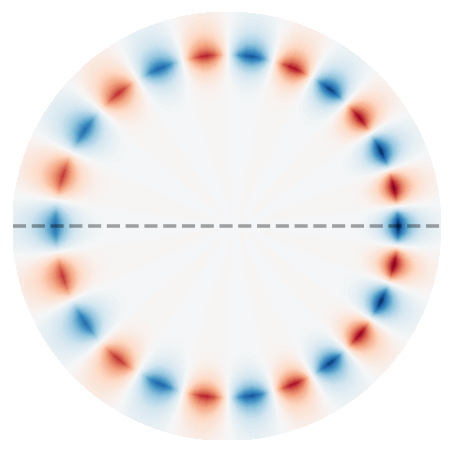}}
    \subfloat{\includegraphics{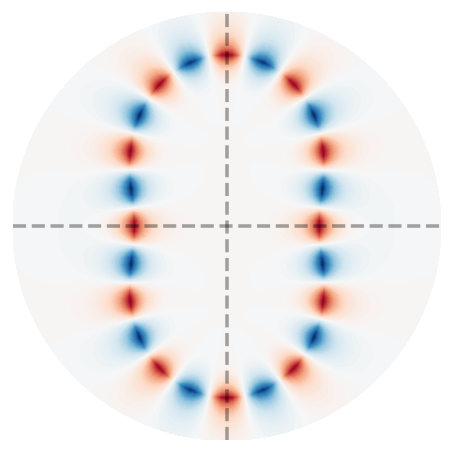}}
    \subfloat{\includegraphics{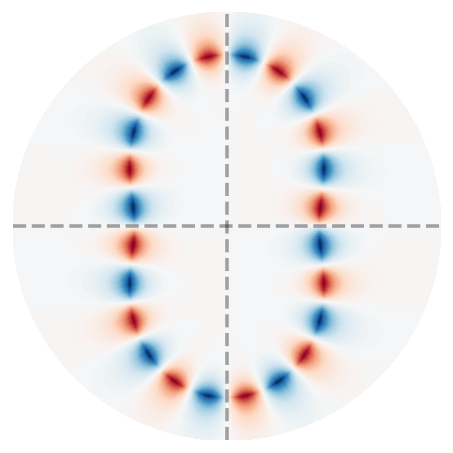}}
    \caption{Real part of the 2 resonant modes associated to \( m = 12 \) curvilinear oscillations, associated to the resonances in \cref{tab:resonances}: for case (B) (left, middle left), for case (C) (middle right, right).
        The gray dashed lines represent the symmetry axes of the problem and hence the symmetries of the modes.
    }\label{fig:modes}
\end{figure}

\paragraph{\textbf{Step 2: norm of the discretized cut-off resolvent operator}}

In \cref{sec:pedago} we computed the discrete norm of the reduced cut-off resolvent operator \( \opnorm{\Ab_k^{-1}}_2 \), obtained using separation of variables.
Here, we compute the discrete norm of a finite element version of the resolvent operator.
We equivalently rewrite \cref{eq:ScatteringProblem} on a bounded domain using a Dirichlet-to-Neumann map (DtN), leading to \cref{eq:ScatteringProblem_bounded} presented in \cref{app:prop_P}.
We use FEM with \( \mathtt{T} \)-conforming meshes such as the ones in \cref{fig:meshes} but without the PML to approximate \cref{eq:ScatteringProblem_bounded}, and we denote \( \mathbb{M}_k \) the finite element matrix of the associated operator.
Then we compute the associated discrete norm \( \opnorm{\mathbb{M}_k^{-1}}_2 \) of the finite element cut-off resolvent operator using the spectral norm by a power method on \( \plr{\mathbb{M}_k^\intercal}^{-1} \mathbb{M}_k^{-1} \) on a uniform \( k \)-grid with geometric refinement around the real part of the scattering resonances.

The FEM computations are done using finite elements of degree 8 (leading to \( 28\, 337 \) degrees of freedom for all three cases), \( 65 \) Fourier modes for the DtN~\cite{Oberai1998}, and \( k \)-grids of \( 160 \) elements for case (A), \( 150 \) elements for cases (B), (C) respectively.

\bigskip

\Cref{fig:norm_ANA_FEM} presents results for case (A), where we can compare \( \opnorm{\mathbb{M}_k^{-1}}_2 \) (dashed orange line) with \( \opnorm{\Ab_k^{-1}}_2 \) (blue line) from the analytic computations in \Cref{sec:pedago}.
Note that {the numerical schemes used in both cases are not the same}, hence we do not expect the results to identically match.
However, the sharp peaks coincide exactly, they occur at \( k = \Re (\ell_\fem) \) (\( \ell_\fem \) being the FEM scattering resonances computed in Step 1), and they exponentially grow as \( k \) increases (the \( y \)-axis is on a logarithmic scale).
The gray vertical lines correspond to the real part of the scattering resonances \( \ell_\fem \).
For larger wavenumbers \( k \), the FEM captures the scattering instabilities, but it fails to capture the peaks' intensity.
This is due to the fact that the mesh is in this case not refined enough (despite high FEM order).

\begin{figure}[!htbp]
    \centering
    \subfloat[Case (A)]{\includegraphics{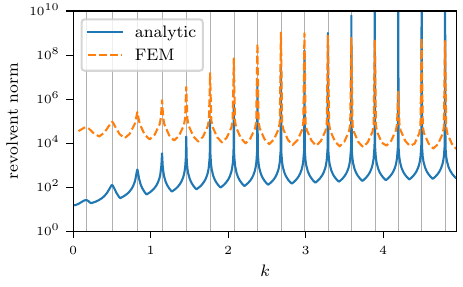}}
    \caption{Semi-log plot of the function \( k \mapsto \norm{\mathbb{M}_{k}^{-1}}_2 \) for the disk cavity with \( \cavc = -1.1 \).
        The blue line correspond to the same analytic computation as in \cref{sec:pedago}.
        The dotted orange lines correspond to FEM computations.
        The vertical grid lines are aligned on the real part of the scattering resonances.
    }\label{fig:norm_ANA_FEM}
\end{figure}

\Cref{fig:N_FEM_2} presents results for cases (B), (C), where we do not have an analytic computation to compare to.
As before, we observe that the norm of the cut-off resolvent operator peaks for \( k = \Re (\ell_\fem) \) (indicated by the gray vertical lines in the figures), and the peaks grow exponentially with respect to \( k \).
As mentioned before, we have two resonant modes corresponding to the same number of curvilinear oscillations \( m \), but they might have slightly different true resonances.
For case (C), we clearly observe this phenomenon (double peaks).
Note that for small \( m \) (\emph{i.e.}\ small real part of the scattering resonances), the norm of the resolvent does not explode.
This is due to scattering resonances having a more significant imaginary part.

\begin{figure}[!htbp]
    \centering
    \subfloat[Case (B)]{\label{fig:disk_epsVar_Ner_pla}
        \includegraphics{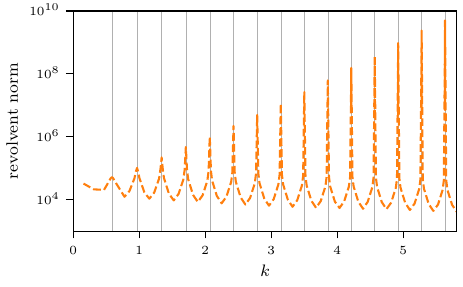}}
    \subfloat[Case (C)]{\label{fig:peanut_epsCon_Ner_pla}
        \includegraphics{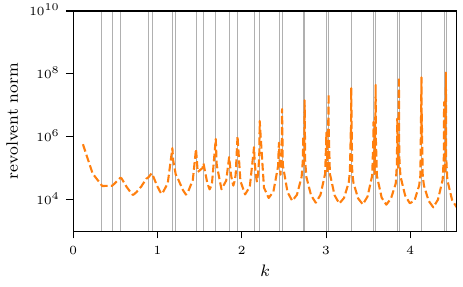}}
    \caption{Semi-log plot of the function \( k \mapsto \norm{\mathbb{M}_{k}^{-1}}_2 \) in logarithmic scale for the two cases \emph{(B)} and \emph{(C)}.
        The vertical grid lines are aligned on the real part of the scattering resonances.}\label{fig:N_FEM_2}
\end{figure}
Numerical results above illustrate the effect of scattering resonances induced by surface plasmons waves, for various metamaterial cavities (in shape and in coefficient).
\paragraph{\textbf{Step 3: comparison with quasi-resonances}}
For validation purposes, we compare the scattering resonances \(\ell_\fem\) computed in Step 1 with \(\sqrt{\lu_m}\), using \(\lu_m\) defined in \cref{eq:lexpan_gen}. In particular, we will compare with the \( \OO(m^{-1})\), the \( \OO(m^{-2})\) expansion, respectively, which corresponds to choosing \cref{eq:lexpan_gen} with one term, \cref{eq:lexpan_gen} with two terms, respectively. We will denote them \( \lu_{m,t} \), \(t = 1,2\). Recall that given \(m\), resonances may be of multiplicity two: in that case we will add the superscript \({}^s\), \(s = 0,1\), to distinguish between the two resonances. We now define the relative difference in scattering resonance with \(t\) terms at order \(m\) by
\[
    D^s_{t}(m) = \abs{\frac{\sqrt{\lu_{m,t}^{s}} - \ell^s_{\fem,m}}{\ell^s_{\fem,m}}},
    \quad  s = 0,1.
\]
Ideally, we expect that \( \lim \limits_{m \to +\infty}D^s_{t}(m) = \OO(m^{-t})\), for \(t = 1,2\).
\Cref{fig:num_vs_asy_diskcst,fig:num_vs_asy} represent \(D^s_{t}(m)\) for cases (A), (B), (C). In the case (A), there is no multiplicity (we drop the superscript \(s\)), and we can use results from \cref{sec:DiskRes} to compare with the analytic scattering resonances \(\ell_m \in \Rc_\pla[-1.1, 1]\). \Cref{fig:num_vs_asy_diskcst} illustrates that ideal behavior is reached. \Cref{fig:num_vs_asy} shows that, for cases (B) and (C), relative differences follows the anticipated slopes, which is promising (especially considering the range of \(m\) that may still be in pre-asymptotic regime). Further analysis of the asymptotic rates could be made (as done in \cite[Chpater 9]{PhD_Moitier_2019}) to verify the asymptotic rates, this requires more computations. Overall, results from \cref{fig:num_vs_asy_diskcst,fig:num_vs_asy} present reasonably small relative error between the scattering resonances and the quasi-resonances.

\begin{figure}[!hbtp]
    \centering
    \subfloat[Relative error for case (A)]{\includegraphics{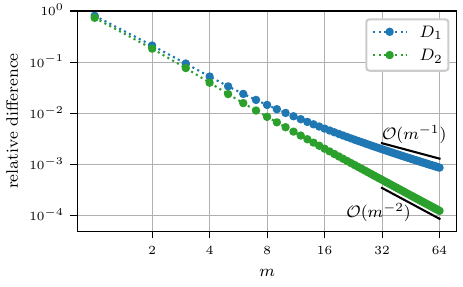}}
    \caption{Relative difference \( D_t(m): = \abs{\frac{\sqrt{\lu_{m,t}} - \ell_{m}}{\ell_{m}}} \), \(t = 1,2\), between the scattering resonances \(\ell_m \in \Rc_\pla[-1.1, 1]\) computed via the modal equation (\cref{eq:detM}), and the asymptotic expansion \( \sqrt{\lu_{m,t}} \) via \cref{eq:lexpan_gen} for case (A).
    }\label{fig:num_vs_asy_diskcst}
\end{figure}

\begin{figure}[!hbtp]
    \centering
    \subfloat[Case (B)]{\includegraphics{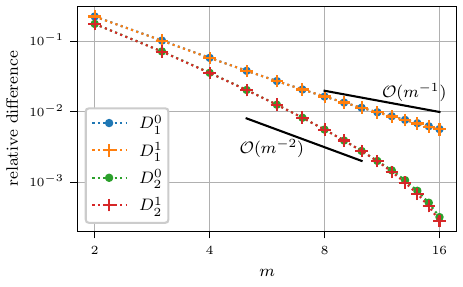}}
    \subfloat[Case (C)]{\includegraphics{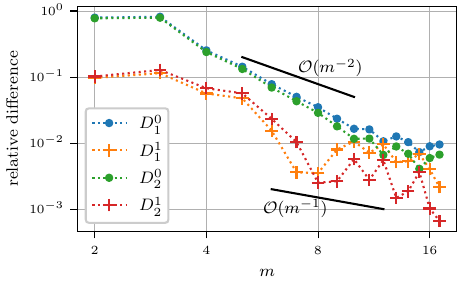}}
    \caption{Graphs of the relative difference \( D_t^s(m) \), \(s = 0,1\), \(t = 1,2\), between the scattering resonances \(\ell^s_{\fem,m} \) computed via FEM (Step 1), and the asymptotic expansion \( \sqrt{\lu^s_{m,t}} \) via \cref{eq:lexpan_gen}, for cases (B) and (C).
    }\label{fig:num_vs_asy}
\end{figure}

\section{Conclusion}\label{sec:conclu}

Similar to classical optical cavities, the scattering by negative metamaterial cavities can be significantly affected by localized waves at the boundary of the cavity.
In this paper we have shown with the black box scattering framework that there exist metamaterial scattering resonances close to the positive real axis, causing the norm of resolvent operator to explode.
Using asymptotic expansions, we have characterized those resonances to arbitrary order, and for various cavity properties (arbitrary smooth shape, varying negative permittivity, etc.).
Numerical experiments illustrate that, for non-trapping metamaterial cavities, scattering resonances are associated to localized waves corresponding to surface plasmons waves.
This study has been carried out without reducing to the quasi-static case, and the considered spectral parameter is the wavenumber in contrast to~\cite{grieser2014plasmonic, Schnitzer2019, Ammari2020, Ammari2020a}.
Our asymptotic analysis revealed that, given some incident source associated to \( k > 0 \), surface plasmon waves can only be excited when \( \cavc < -1 \) (in the case \( -1 < \cavc < 0 \) the scattering resonances are purely imaginary).
We have established that the ``quasimodes to quasi-resonances'' result still applies for unbounded transmission problems with sign-changing coefficients: then the existence of quasi-pairs implies the existence of scattering resonances close to the positive real axis which also implies the explosion of the stability constant when \( \cavc < -1 \).
FEM computations confirm that the norm of the numerical resolvent operator exhibits high intensity narrow peaks associated to the scattering resonances close to the positive real axis.

Our approach provides the construction of surface plasmons waves quasi-modes for general metamaterial cavities, to arbitrary order.
The constructed quasi-modes seem to concur with resonant modes computed for non-trapping cavities.
For trapping cavities, combinations of localized-trapped resonant modes could exist.
The approach can be carried out for multi-layered cavities (typically a dielectric cavity surrounded by an annulus of metamaterial): one can build quasi-pairs for each interface (the localization process decouples phenomena).
Then the ``quasimodes to quasi-resonances'' argument should hold similar results.
One could consider extracting those asymptotic plasmonic behaviors from the problem to relax FEM (no peaks), as done in the singular complement method~\cite{ciarlet2003singular}.
One could also, using the same expansion methods, find asymptotic characterization in the context of dispersive material cavities (in particular the case where \( \cavc \coloneqq \eps_\cav \) is the permittivity and depends on the wavenumber \( k \), such as Drude's or Lorentz' model).
In that case, our analysis confirms that surface plasmons waves can only be excited for frequencies lower than the surface plasmons' frequency~\cite{maier2007plasmonics}, however, since the domain of the operator depends on the spectral parameter, the link between quasi-pairs and scattering resonances is not clear.
Extensions to polygonal metamaterial cavities and dispersive metamaterials will be considered.
In the quasi-static case, the spectral analysis for that case reveals hypersingular plasmonic behaviors and has been well investigated~\cite{HAZARD2017, BHM20-hal}.
The proposed asymptotic expansions approach is valid for arbitrary optical parameter \( \cavc  \) (and complex-valued ones to some extent), one could also consider arbitrary double negative optical parameters \( \cavb \) and work with the double-negative PDE \( -\Div(\lapc^{-1} \, \nabla u) - b \, k^2 \, u = 0 \) (\emph{e.g.}~\cite{BCC12maxwell,Fitzpatrick2018,Ammari2019}).
Then, to deduce from the quasi-pairs existence the presence of scattering resonances becomes difficult because the operator is no longer self-adjoint.
All the derivations have been provided for two-dimensional problems, one could consider three-dimensional cavities.
In particular, results from \cref{app:prop_P} and \cref{sec:resolv_results} hold in \( \Rb^3 \), for smooth interface \(\Gamma\).
The construction of quasi-pairs may be more cumbersome, but it can be adapted using a parameterized tubular region and the use of the mean curvature.

\section*{Acknowledgments}

The authors would like to thank the reviewers for their helpful comments, C.~Tsogka and A.~D.~Kim for their feedback, and M.~Dauge for the fruitful discussions.

\section*{Funding}

This research was supported by the National Science Foundation Grant: DMS-2009366 and by the Deutsche Forschungsgemeinschaft (DFG, German Research Foundation) --- Project-ID 258734477 --- SFB 1173.

\bibliographystyle{abbrvurl}
\bibliography{../../references}

\appendix

\section{Properties of the operator \texorpdfstring{\( P \)}{P}}\label{app:prop_P}

We recall the operator \( P \colon u \mapsto -\Div(\lapc^{-1} \, \nabla u) \).
Given \( \omega \subseteq \Rb^2 \), we define the bilinear form
\begin{equation}\label{eq:form_b}
    b_\omega(u,v) = \displaystyle \int_{\omega} \lapc^{-1} \nabla u \cdot \nabla v \,\di{x}, \quad u,v \in \Hr^1(\omega).
\end{equation}
Then \( b \coloneqq b_{\Rb^2} \) is the associated bilinear form of \( (P, \Dc(P)) \), one can write \( b(u,v) = {(Pu, v)}_{\Lr^2(\Rb^2)} \) for \( u \in \Dc(P) \), \( v \in \Hr^1(\Rb^2) \), and is the associated bilinear form of \cref{eq:ScatteringProblem} for \( \Pt \colon \Hr^1(\Rb^2) \to \Hr^{-1}(\Rb^2) \) (\( b(u, v) = \langle \Pt u, v \rangle \) for \( u, v \in \Hr^{1}(\Rb^2) \), where  \( \langle \cdot, \cdot \rangle \) is the duality bracket \(\Hr^{-1}(\Rb^2) \times \Hr^{1}(\Rb^2)\)).

\begin{lemma}\label{lem:P_domain}
    The domain of \( P \) define by \( \Dc(P) = \setst{u \in \Hr^1(\Rb^2)}{\Div\plr{a^{-1} \nabla u} \in \Lr^2(\Rb^2)} \) is equivalent to
    \[
        \Dc(P) = \setst{u \in \Hr^1\plr{\Rb^2}}{
            \Delta\restr{u}{\Omega} \in \Lr^2(\Omega),\ \Delta\restr{u}{\Rb^2 \setminus \overline{\Omega}} \in \Lr^2\plr{\Rb^2 \setminus \overline{\Omega}},\ \clr{\lapc^{-1} \partial_n u}_\Gamma = 0
        }.
    \]
\end{lemma}

\begin{remark}\label{rem:P_domain}
    Without any assumption on the value of \( \cavc \) on the interface \( \Gamma \), we cannot expect \( \Hr^2 \) regularity up to the interface, see for example \cite[Theorem 1 and 2]{CacPanPos_2019}.
\end{remark}

\begin{proof}
    Let us denote
    \[
        \mathcal{E} = \setst{u \in \Hr^1\plr{\Rb^2}}{
            \Delta\restr{u}{\Omega} \in \Lr^2(\Omega),\ \Delta\restr{u}{\Rb^2 \setminus \overline{\Omega}} \in \Lr^2\plr{\Rb^2 \setminus \overline{\Omega}},\ \clr{\lapc^{-1} \partial_n u}_\Gamma = 0
        }.
    \]
    For the inclusion \( \mathcal{E} \subset \Dc(P) \), take \( u \in \mathcal{E} \), using \( \clr{\lapc^{-1} \partial_n u}_\Gamma = 0 \) and Green's identity in a distributional sense, we have
    \[
        \Div\plr{\lapc^{-1}\, \nabla u} = \begin{dcases}
            \Div\plr{\cavc^{-1} \nabla \restr{u}{\Omega}}       & \text{in } \Omega
            \\
            \Delta \restr{u}{\Rb^2 \setminus \overline{\Omega}} & \text{in } \Rb^2 \setminus \overline{\Omega}
        \end{dcases}
    \]
    and with \( \Div(\cavc^{-1} \nabla \restr{u}{\Omega}) = \cavc^{-1} \Delta \restr{u}{\Omega} + \nabla(\cavc^{-1}) \cdot \nabla \restr{u}{\Omega} \), it gives \( \Div\plr{\lapc^{-1}\, \nabla u} \in \Lr^2(\Rb^2) \).
    Therefore, we obtain \( \mathcal{E} \subset \Dc(P) \).
    For the reciprocal inclusion, let's take \( u \in \Dc(P) \) and \( v \in \Hr^1(\Rb^2) \), using Green's identity and duality bracket, we get
    \[
        \left\langle \clr{\lapc^{-1} \partial_n u}_\Gamma,\, \restr{v}{\Gamma} \right\rangle_{\Hr^{-1/2}(\Gamma),\, \Hr^{1/2}(\Gamma)}
        = b_{\Rb^2}(u,\, v) - \plr{P u, v}_{\Lr^2(\Rb^2)} = 0
    \]
    which gives \( \clr{\lapc^{-1} \partial_n u}_\Gamma = 0 \).
    With \( \Div\plr{a^{-1} \nabla u} \in \Lr^2(\Rb^2) \), we get \( \Delta\restr{u}{\Omega} = \cavc [\Div(\cavc^{-1} \nabla \restr{u}{\Omega}) - \nabla(\cavc^{-1}) \cdot \nabla \restr{u}{\Omega}] \in \Lr^2(\Omega) \) and \( \Delta \restr{u}{\Rb^2 \setminus \overline{\Omega}} \in \Lr^2(\Rb^2 \setminus \overline{\Omega}) \).

\end{proof}

\begin{lemma}\label{lem:b_wT_coercif}
    If \( \cavc(\gamma) \neq -1 \), for all \( \gamma \in \Gamma \), the bilinear form \( b_\omega \) defined in \cref{eq:form_b} is weakly \( \mathtt{T} \)-coercive.
    More precisely, there exists an isomorphism \( \mathtt{T} \in \mathcal{L}(\Hr^1(\omega)) \), a compact operator \( \Ct \in \mathcal{L}(\Lr^2(\omega)) \){, \( \alpha > 0 \), and \( \beta \in \Rb \)} such that \( b_\omega \) satisfies a G{\"a}rding's inequality of the form:
    \[
        b_\omega(u, \mathtt{T} u) \geq \alpha \norm{u}^2_{\Hr^1(\omega)} - \beta  \norm{ \Ct u}^2_{\Lr^2(\omega)},
        \quad \forall u \in \Hr^1(\omega).
    \]
\end{lemma}
\begin{proof}
    When \( \cavc < 0 \) is constant, one can use \( \mathtt{T} \) provided in~\cite{BoCaCi18} and the proof follows the one of~\cite[Lemma 2]{BoCaCi18}.
    When \( \cavc \in \Cs^\infty(\Omega) \) non-constant, since \( \partial \Omega \) is a smooth interface, it can always be seen as locally straight, then Theorems 3.10 and 4.3 in~\cite{BoChCi12} apply and provide the needed results.
\end{proof}

\begin{lemma}\label{lem:Scat_pb_wellposed}
    If \( \cavc(\gamma) \neq -1 \), for all \( \gamma \in \Gamma \), the operator \( \Pt \) is Fredholm of index 0 and \cref{eq:ScatteringProblem} is well-posed.
    Moreover, there exists a stability constant \( C(k) > 0 \) such that
    \begin{equation}\label{eq:esti_sc_gen}
        \norm{u}_{\Lr^2(\Db(0,\rho))} \le C(k) \plr{\norm{f}_{\Lr^2(\Rb^2)} + \norm{g}_{\Lr^2(\Gamma)}},
    \end{equation}
    for any open disk of radius \( \rho \) such that \( \overline{\Omega} \cup \supp(f) \subset \Db(0,\rho) \).
\end{lemma}
\begin{proof}
    Let \( \mathbb{D}(0, \rho) \) be a disk a radius \( \rho \) such that \( \Omega \) is compactly embedded in \( \Db(0, \rho) \), and \( f \in \Lr^2(\mathbb{D}(0, \rho)) \).
    Following~\cite{BCCC16}, we use a Dirichlet-to-Neumann map, denoted \( \mathcal{S} \), to rewrite \cref{eq:ScatteringProblem} in \( \mathbb{D}(0, \rho) \): Find \( {u} \in \Hr^1(\mathbb{D}(0, \rho) ) \) such that
    \begin{align}
        \begin{dcases}
            -\Div\plr{\lapc^{-1}\, \nabla u} - k^2 u = f
             & \text{in } \mathbb{D}(0, \rho)
            \\[0.5ex]
            \clr{u}_\Gamma = 0 \quad\text{and}\quad \clr{\lapc^{-1}\, \partial_n u}_\Gamma = g
             & \text{across } \Gamma
            \\[0.5ex]
            \partial_r u  = \mathcal{S}u
             & \text{on } \partial \mathbb{D}(0, \rho)
        \end{dcases}\label{eq:ScatteringProblem_bounded}
    \end{align}
    Lemma 1 in~\cite{BCCC16} shows that problems \cref{eq:ScatteringProblem_bounded}-\cref{eq:ScatteringProblem} admits at most one solution. Following~\cite[Section 2]{BCCC16}, using the properties of \( \mathcal{S} \) and the fact that \( \Kt \colon u \mapsto  - k^2 u \) is compact, one simply needs to establish that the operator \( \Pt \colon u \mapsto -\Div\plr{ \lapc^{-1} \, \nabla u } \) is Fredholm to conclude.
    From~\cite[Proposition 2.6]{carvalho2015etude}, it is equivalent to show that \( b|_{\mathbb{D}(0, \rho)} \) in \cref{eq:form_b} is weakly \( \mathtt{T} \)-coercive, which is established by \cref{lem:b_wT_coercif}.
    Well-posedness of \cref{eq:ScatteringProblem_bounded} in Hadamard's sense gives u that there exists \( \widetilde{C}(k) > 0 \) such that
    \[
        \norm{u}_{\Hr^1(\mathbb{D}(0, \rho))} \leq \widetilde{C}(k) \plr{\norm{g}_{\Lr^2(\Gamma)} + \norm{f}_{\Lr^2(\Db(0, \rho))}}.
    \]
    For \cref{eq:ScatteringProblem}, using Poincaré's inequality this leads to
    \begin{equation}\label{eq:wp_estimate}
        \norm{u}_{\Lr^2(\mathbb{D}(0, \rho))}
        \leq C(k) \plr{\norm{g}_{\Lr^2(\Gamma)} + \norm{f}_{\Lr^2(\Rb^2)}}.
    \end{equation}
\end{proof}

\begin{lemma}\label{lem:P_spectral}
    If \( \cavc(\gamma) \neq -1 \), for all \( \gamma \in \Gamma \), then \( (P, \Dc(P)) \) is self-adjoint, and its spectrum is such that \( \spess(P) = \Rb_+ \) and \( \spdis(P) \subset \Rb_-^* \).
\end{lemma}
\begin{proof}
    The proof is given by applying Theorem 4.2, Propositions 4.5 and 4.6 in~\cite[Chapter 4]{carvalho2015etude}.
    Consider \( \lambda \in \Cb \setminus \Rb \) and the problem: Find \( u \in \Hr^1(\Rb^2) \) such that \( b(u,v)- \lambda {(u, v)}_{\Lr^2} = {(f, v)}_{\Lr^2} \), \( \forall v \in \Hr^1(\Rb^2) \), with \( f \in \Lr^2(\Rb^2) \).
    Using \cref{lem:b_wT_coercif}, \( b \) is weakly \( \mathtt{T} \)-coercive and the above problem is well-posed (\cref{lem:Scat_pb_wellposed}). This shows that \( P \) is self-adjoint. Given \( \lambda \in \spess(P) \), consider \( {(u_n)}_n \in \Dc(P) \) such \( \norm{u_n}_{\Lr^2(\Rb^2)} = 1 \), \( u_n \rightharpoonup 0 \) weakly in \( \Lr^2 \) and such that \( \norm{Pu_n - \lambda u_n}_{\Lr^2} \to 0 \). Using \cref{lem:b_wT_coercif}, we have
    \[
        \plr{Pu_n, \Tt u_n}_{\Lr^2(\Rb^2)} \geq \alpha \norm{u_n}^2_{\Hr^1(\Rb^2)} - \beta  \norm{ \Ct u_n}^2_{\Hr^1(\Rb^2)} \geq  - \beta  \norm{ \Ct u_n}^2_{\Lr^2(\Rb^2)},
    \]
    and we note that
    \[
        \plr{Pu_n, \Tt u_n}_{\Lr^2(\Rb^2)} = \lambda {(u_n,\Tt u_n)}_{\Lr^2}  = \lambda + \lambda (u_n, (\Tt -\Ir)u_n)_{\Lr^2(\Rb^2)}.
    \]
    Since \( u_n \rightharpoonup 0 \) weakly in \( \Lr^2 \), one can show that \(\norm{ \Ct u_n}^2_{\Lr^2(\Rb^2)} \to 0 \), \((u_n, (\Tt -\Ir)u_n)_{\Lr^2(\Rb^2)} \to 0\) strongly, which leads to \( \lambda \geq 0 \).
    On the other hand, for \( \lambda \geq 0 \), one can build a Weyl sequence \( {(u_n)}_n \in \Dc(P) \) such \( \norm{u_n}_{\Lr^2(\Rb^2)} = 1 \), \( u_n\rightharpoonup 0 \) weakly in \( \Lr^2 \) and such that \( \norm{Pu_n - \lambda u_n}_{\Lr^2} \to 0 \).
    Rellich lemma allows us to show that there are no eigenvalues in \( \spess(P) \).
    Finally, \( P \) doesn't admit a lower bound (details can be found in~\cite[Section 4.2.2]{carvalho2015etude}): one can consider a sequence \( {(u_n)}_n \in \Dc(P) \) with support strictly included in \( \Omega \) such that the numerical range \( {(P u_n, u_n)}_{\Lr^2} \to -\infty \) (recall that \( \cavc <0 \)), which shows that \( \spdis(P) \subset \Rb_-^* \).
\end{proof}

\section{Proofs and additional results for the asymptotic expansions}\label{app:sc_asy_details}

\subsection{Proof of \texorpdfstring{\cref{lem:sol0}}{lem:sol0}}\label{sc_asy_details_0}

\begin{proof}
    We solve \cref{eq:prob_P0} as ordinary differential equations with \( s \in \Tb_L \) as a parameter.
    The conditions \( \vp_0^\pm(s, \cdot) \in \Ss(\Rb_\pm) \) give the following restrictions \( {\theta_0'(s)}^2 + {\eta_0(s)}^2\ls_0 \in \Cb \setminus \Rb_- \) and \( {\theta_0'(s)}^2 - \ls_0 \in \Cb \setminus \Rb_- \).
    If one of the above restrictions is false, then there are no solutions \( \vp^\pm(s, \cdot) \) in \( \Ss(\Rb_\pm) \).
    Under those restrictions, there exists \( \alpha(s), \beta(s) \in \Rb \) such that \( \alpha(s)\beta(s) \neq 0 \),
    \[
        \vp_0^-(s, \rho) = \alpha(s) e^{ \rho \sqrt{{\theta_0'(s)}^2 + {\eta_0(s)}^2\ls_0} },
        \quad \text{and} \quad
        \vp_0^+(s, \rho)  = \beta(s) e^{ -\rho \sqrt{{\theta_0'(s)}^2 - \ls_0} },
    \]
    where the square roots are chosen to be in \( \Cb^{\frac{1}{2}} \).
    The first transmission condition \( \vp_0^-(s,0) = \vp_0^+(s,0) \) implies that \( \alpha(s) = \beta(s) \).
    Then the second transmission condition
    \[
        -{\eta_0(s)}^{-2}\, \partial_\rho\vp_0^-(s,0) = \partial_\rho\vp_0^+(s,0)
    \]
    give us
    \[
        -{\eta_0(s)}^{-2} \sqrt{{\theta_0'(s)}^2 + {\eta_0(s)}^2\ls_0}
        = -\sqrt{{\theta_0'(s)}^2-\ls_0},
    \]
    leading to the eikonal equation
    \[
        {\theta_0'(s)}^2 = \frac{\ls_0}{1-{\eta_0(s)}^{-2}}
        = \seta\ls_0\, \abs{ 1-{\eta_0(s)}^{-2} }^{-1}.
    \]
    While this equation does not have a unique solution, one simply selects one (see \cref{rem:sol0_unicity}).
    Here we choose
    \[
        \theta_0(s) = \sqrt{\seta\ls_0} \int_0^s \abs{ 1-{\eta_0(t)}^{-2} }^{-\frac{1}{2}} \di{t}
    \]
    and from the condition \( \exp\plr{\frac{\im}{h} \theta_0} \in \Cs^\infty(\Tb_L) \), we deduce that \( \exp\plr{\frac{\im}{h} \theta_0(L)} = \exp\plr{\frac{\im}{h} \theta_0(0)} \) which implies that there exists \( m \in \Nb \) such that
    \[
        2\pi m = \frac{ \theta_0(L) - \theta_0(0) }{h}
        = \frac{\sqrt{\seta\ls_0}}{h} \int_0^L \abs{ 1-{\eta_0(t)}^{-2} }^{-\frac{1}{2}} \di{t}.
    \]
    By choosing \( h = \frac{L}{2\pi m} \) for \( m \in \Nb^* \), we get
    \( 1 = \sqrt{\seta\ls_0} \alr{ \abs{1-\eta_0^{-2}}^{-\frac{1}{2}} }
    = \sqrt{\seta\ls_0} \alr{\ftz} \)
    which gives \( \ls_0 = \seta \alr{\ftz}^{-2} \).
    Then with the relation \( \ftz^2 = \varsigma {(1-\eta_0^{-2})}^{-1} \) we obtain that
    \[
        \sqrt{{\theta_0'(s)}^2 + {\eta_0(s)}^2\ls_0} = \htz(s)\, \eta_0(s) > 0
        \quad \text{and} \quad
        \sqrt{{\theta_0'(s)}^2-\ls_0} = \htz(s)\, {\eta_0(s)}^{-1} > 0,
    \]
    which concludes the proof.
\end{proof}

\subsection{Proof of \texorpdfstring{\cref{lem:soln}}{lem:soln}}\label{sc_asy_details_n}

\begin{proof}
    For \( (s, \rho) \in \Tb_L \times \Rb_\pm \), we define \( \e^\pm(s, \rho) = \exp\plr{ -|\rho|\ \htz(s)\, {\eta_0(s)}^{\mp 1} } \).
    We proceed by induction on \( n \).
    For \( n = 0 \), \cref{lem:sol0} gives \( (\vp_0^\pm, \theta_0, \ls_0) \) the solution of \( (\Pc_0) \) defined in \cref{eq:prob_P0}.
    Let \( n \ge 1 \), from the definition of \( S_{n-1}^\pm \) in \cref{eq:Spm}, there exists \( Q_{n-1}^\pm \in \Cs^\infty(\Tb_L, \Pb) \) such that \( S_{n-1}^\pm = Q_{n-1}^\pm \, \e^\pm \).
    Using Lemma A.1 in~\cite{BalDauMoi21}, we can solve the two ODEs in \cref{eq:prob_Pn_expand} with the source terms \( S_{n-1}^\pm \). We find that there exists \( \widetilde{P}_n^\pm \in \Cs^\infty(\Tb_L, \Pb) \) such that \( \widetilde{\vp}_n^\pm = \rho\widetilde{P}_n^\pm \, \e^\pm \), \( \partial_\rho^2 \widetilde{\vp}_n^- -\htz^2\, \eta_0^2\, \widetilde{\vp}_n^- = \eta_0^2\,S_{n-1}^- \), and \( \partial_\rho^2 \widetilde{\vp}_n^+ - \htz^2\, \eta_0^{-2}\, \widetilde{\vp}_n^+ = -S_{n-1}^+ \).
    Then, solving the two ODEs in \cref{eq:prob_Pn_expand} with the source terms \( (2\htz\theta_0' + \eta_0^2\ls_n) \vp_0^- \) and \( (2\htz\theta_0' - \ls_n) \vp_0^- \), for \( (s, \rho) \in \Tb_L \times \Rb_\pm \), we obtain
    \begin{subequations}\label{eq:phi_expr}
        \begin{align}
            \vp_n^-(s, \rho) & = \alpha(s)\rho\plr{\frac{\eta_0(s)\, \ls_n}{2\htz(s)} + \frac{\theta_n'(s)}{\eta_0(s)} + \frac{\widetilde{P}_n^-(s, \rho)}{\alpha(s)}} \, \e^-(s, \rho),
            \\
            \vp_n^+(s, \rho) & = \alpha(s)\rho\plr{\frac{\eta_0(s)\, \ls_n}{2\htz(s)} - \eta_0(s)\, \theta_n'(s) + \frac{\widetilde{P}_n^+(s, \rho)}{\alpha(s)}} \, \e^+(s, \rho).
        \end{align}
    \end{subequations}
    The first transmission condition \( \vp_n^-(\cdot, 0) = \vp_n^+(\cdot, 0) \) is satisfied because \( \vp_n^\pm(\cdot, 0) = 0 \).
    Using the second transmission condition \( -\eta_0^{-2}\, \partial_\rho\vp_n^-(\cdot, 0) = \partial_\rho\vp_n^+(\cdot, 0) \) and the expressions in \cref{eq:phi_expr}, we get
    \[
        -\eta_0^{-2} \plr{ \frac{\eta_0(s)\, \ls_n}{2\htz(s)} + \frac{\theta_n'(s)}{\eta_0(s)} + \frac{\widetilde{P}_n^-(s, 0)}{\alpha(s)} }
        = \frac{\eta_0(s)\, \ls_n}{2\htz(s)} - \eta_0(s)\, \theta_n'(s) + \frac{\widetilde{P}_n^+(s, 0)}{\alpha(s)}.
    \]
    Solving for \( \theta_n' \) and integrating yields
    \[
        \theta_n(s) = \int_0^s \frac{\ls_n}{2\htz(t) (1-{\eta_0(t)}^{-2})} + \frac{\eta_0(t) \widetilde{P}_n^-(t,0) + {\eta_0(t)}^3\widetilde{P}_n^+(t,0)}{\alpha(t)\, ({\eta_0(t)}^4 - 1)} \di{t}.
    \]
    Now, the condition \( \exp(\im \, h^{n-1}\, \theta_n) \in \Cs^\infty(\Tb_L) \) imposes \( \theta_n(L) = \theta_n(0) \), solving for \( \ls_n \) and using the relation \( {\ftz(t)}^2\, (1-\eta_0^{-2}) = \seta \) yields
    \[
        \ls_n = -\frac{2\, \varsigma}{\alr{\ftz}^2}\ \alr{\frac{\eta_0 \widetilde{P}_n^-(\cdot, 0) + \eta_0^3\widetilde{P}_n^+(\cdot, 0)}{\alpha \, (\eta_0^4 - 1)}}.
    \]
    Setting \( P_n^\pm(s, \rho)  = \alpha(s)\rho\plr{\frac{\eta_0(s)\, \ls_n}{2\htz(s)} \mp {\eta_0(s)}^{\pm 1} \theta_n'(s) + \frac{\widetilde{P}_n^\pm(s, \rho)}{\alpha(s)}} \) finishes the proof.
\end{proof}

\subsection{Additional results for Schwartz functions}\label{sc_lemma_schwartz}

\begin{lemma}\label{lem:int_hoo}
    Consider \( F \colon (h; s, \rho) \mapsto F(h; s, \rho) \) in \( \Cs^\infty([0,\frac{L}{2\pi}] \times \Tb_L, \Ss(\Rb_\pm)) \), \( \rho > 0 \), and the intervals \( I_-(h) = (-\infty, -\frac{\rho}{h}) \) and \( I_+(h) = (\frac{\rho}{h}, +\infty) \).
    Then
    \[
        \int_{\Tb_L} \int_{I_\pm(h)} |F(h; s, \rho)|^2 \di{\rho}\di{s} = \OO(h^\infty) \quad \text{as } h \to 0.
    \]
\end{lemma}
\begin{proof}
    Notice that, for any integer \( N \ge 1 \), there exists a constant \( C_N > 0 \) such that \( |\rho^N F(h; s, \rho)| \le C_N \) for all \( (h; s, \rho) \in [0,\frac{L}{2\pi}] \times \Tb_L \times \Rb_\pm \).
    Hence,
    \[
        \int_{\Tb_L} \int_{I_\pm(h)} |F(h; s, \rho)|^2 \di{\rho}\di{s} \le \frac{C_N\, L}{(2N-1)\, \rho^{2N-1}}\ h^{2N-1},
    \]
    which finishes the proof.
\end{proof}

\subsection{Additional results used in \texorpdfstring{\cref{sec:construct_quasi}}{~\ref{sec:construct_quasi}}}\label{sec:theta_1}

\begin{lemma}\label{lem:theta_1}
    For \( s \in \Tb_L \),
    \begin{multline*}
        \theta_1(s) = \int_0^s \frac{\ls_1}{\ls_0} \htz(t)
        + \frac{({\eta_0(t)}^2 - 1)\, \kappa(t)}{2\, \eta_0(t)}
        + \frac{\eta_1(t)}{2\, {\eta_0(s)}^2\, ({\eta_0(t)}^2-1)}\\
        + \im\frac{({\eta_0(t)}^4+3)\, \eta_0'(t)}{2\, \eta_0(t)\, ({\eta_0(t)}^4-1)}
        + \im\frac{\alpha'(t)}{\alpha(t)} \di{t}.
    \end{multline*}
\end{lemma}
\begin{proof}
    This follows from the computation performed in \cref{sc_asy_details_n} where we solve \cref{eq:prob_Pn_expand} for \( n = 1 \) with
    \begin{align*}
        \eta_0^2 S_0^-
         & = \begin{multlined}[t][0.8\textwidth]
                 2\rho\frac{\eta_1}{\eta_0} \partial_\rho^2\vp_0^-
                 - \plr{\kappa - \frac{2\eta_1}{\eta_0}} \partial_\rho\vp_0^-
                 - 2\im \theta_0' \partial_s\vp_0^-
                 \\
                 - \plr{2\rho {\theta_0'}^2 \plr{\kappa + \frac{\eta_1}{\eta_0}} + \im \theta_0'' - 2\im \theta_0'\frac{\eta_0'}{\eta_0}} \vp_0^-,
             \end{multlined}
        \\
        -S_0^+
         & = -\kappa \partial_\rho\vp_0^+
        - 2\im \theta_0' \partial_s\vp_0^+
        - \plr{2\rho {\theta_0'}^2 \kappa + \im \theta_0''} \vp_0^+.
    \end{align*}
\end{proof}

\end{document}